\documentclass[11pt]{article}
\usepackage{amssymb,amsmath, amsthm}
\usepackage{amssymb}
\usepackage{enumerate}
\usepackage{hyperref}
\usepackage{cleveref}
\crefformat{footnote}{#2\footnotemark[#1]#3}
\usepackage[algoruled,boxruled, linesnumbered, shortend, lined]{algorithm2e}
\usepackage{algpseudocode}
\usepackage{cite}
\usepackage{enumitem}
\setenumerate[0]{leftmargin=*}
\usepackage{caption}
\usepackage{graphicx}

\usepackage{mathtools}
\DeclarePairedDelimiter{\ceil}{\lceil}{\rceil}

\numberwithin{equation}{section}

\newcommand{\ve}{\varepsilon}
\DeclareMathOperator*{\proj}{proj}
\DeclareMathOperator*{\dom}{dom\,}
\DeclareMathOperator*{\lip}{lip\,}

\DeclareMathOperator*{\diam}{diam}

\newcommand{\ip}[1] {\langle #1 \rangle }
\newcommand{\norm}[1] {\left \| #1 \right \|}
\newcommand{\inclu}[0] {\ar@{^{(}->}}

\newcommand{\dist}{{\rm dist}}
\newcommand{\prox}{{\rm prox}}
\newcommand{\R}{{\bf R}}

\usepackage[margin=1.1in]{geometry}

\newcommand{\argmin}{\operatornamewithlimits{argmin}}
\newcommand{\lf}{\operatornamewithlimits{liminf}}

\SetAlgoSkip{bigskip}
\SetKwInput{Input}{Input\hspace{34pt}{ }}
\SetKwInput{Output}{Output\hspace{27pt}{ }}
\SetKwInput{Initialize}{Initialize\hspace{2pt}{ }}
\SetKwInput{Parameters}{Parameters\hspace{2pt}{ }}

\newtheorem{theorem}{Theorem}[section]
\newtheorem{proposition}[theorem]{Proposition}
\newtheorem{lemma}[theorem]{Lemma}
\newtheorem{remark}[theorem]{Remark}
\newtheorem{defn}[theorem]{Definition}

\newtheorem{example}{Example}[section]

\title{Efficiency of minimizing compositions of convex functions and smooth maps
 \thanks{University of Washington, Department of Mathematics, 
		Seattle, WA 98195; Research of Drusvyatskiy and Paquette was partially supported by the AFOSR award FA9550-15-1-0237 and NSF DMS 1651851.}}

\begin{document}
	
	\author{D. Drusvyatskiy 
		\thanks{
			E-mail: ddrusv@uw.edu;	\texttt{http://www.math.washington.edu/{\raise.17ex\hbox{$\scriptstyle\sim$}}ddrusv/}}
		\and
		$\textrm{C. Paquette}$
		\thanks{
			E-mail: yumiko88@uw.edu;			
		}}	

\date{}	
\maketitle

\begin{abstract}
We consider global efficiency of algorithms for  minimizing a sum of a convex function and a composition of a Lipschitz convex function with a smooth map. The basic algorithm we rely on is the prox-linear method, which in each iteration solves a regularized subproblem formed by linearizing the smooth map. When the subproblems are solved exactly, the method has  efficiency $\mathcal{O}(\varepsilon^{-2})$, akin to gradient descent for smooth minimization.  We show that when the subproblems can only be solved by first-order methods, a simple combination of smoothing, the prox-linear method, and a fast-gradient scheme yields an algorithm with  complexity $\widetilde{\mathcal{O}}(\varepsilon^{-3})$. The technique readily extends to minimizing an average of $m$ composite functions,  with complexity $\widetilde{\mathcal{O}}(m/\varepsilon^{2}+\sqrt{m}/\varepsilon^{3})$ in expectation.
We round off the paper with an inertial prox-linear method that automatically accelerates in presence of convexity. 
\end{abstract}

\bigskip

\noindent\textbf{Key words.} 
Composite minimization, fast gradient methods, Gauss-Newton, prox-gradient, inexactness, complexity, smoothing, incremental methods, acceleration
\vspace{0.1cm}

\noindent\textbf{AMS Subject Classification.}  \textit{Primary}  97N60, 90C25; 
\textit{Secondary} 90C06, 90C30.


\section{Introduction}

In this work, we consider the class of  {\em composite optimization problems} 
\begin{equation}\label{eqn:key_comp}
	\min_x\, F(x):=g(x)+h(c(x)),
\end{equation}
where $g\colon\R^d\to\R\cup\{\infty\}$ and  $h\colon\R^m\to\R$ are closed convex functions and $c\colon\R^d\to\R^m$ is a smooth map.  Regularized nonlinear least squares \cite[Section 10.3]{NW} and exact penalty formulations of nonlinear programs \cite[Section  17.2]{NW} are classical examples, while notable contemporary instances include robust phase retrieval \cite{duchi_ruan,duchi_ruan_PR} and matrix factorization problems such as NMF \cite{Gil_surv,gill_siam_news_views}. 
The setting where $c$ maps to the real line and $h$ is the identity function, \begin{equation}\label{eqn:add_comp_intro}
	\min_x~ c(x)+g(x),
\end{equation}
is  now commonplace in large-scale optimization. 
In this work, 
we use the term {\em additive composite minimization} for \eqref{eqn:add_comp_intro} to distinguish it from the more general composite  class \eqref{eqn:key_comp}. 

The proximal gradient algorithm, investigated by Beck-Teboulle \cite{beck} and Nesterov \cite[Section 3]{nest_conv_comp}, is a popular first-order method
for additive composite minimization. Much of the current paper will center around the 
{\em prox-linear method}, which is a direct extension of the prox-gradient algorithm to the entire problem class \eqref{eqn:key_comp}.
In each iteration, the prox-linear method linearizes the smooth map $c(\cdot)$ and solves the {\em proximal subproblem}:
\begin{equation}\label{eqn:prox-subprob}
	x_{k+1}=\argmin_{x}\, \Big\{g(x)+h\Big(c(x_k)+\nabla c(x_k)(x-x_k)\Big)+\tfrac{1}{2t}\|x-x_k\|^2\Big\},
\end{equation}
for an appropriately chosen parameter $t>0$.  The underlying assumption here is that the strongly convex proximal subproblems 
\eqref{eqn:prox-subprob} can be solved efficiently. This is indeed reasonable in some circumstances. For example, one may have available specialized methods for the proximal subproblems, or interior-point points methods may be available for moderate dimensions $d$ and $m$, or it may be that case that computing an accurate estimate of $\nabla c(x)$ is already the bottleneck (see e.g. Example~\ref{ex:grey}). The prox-linear method was recently investigated
in \cite{prox,prox_error,composite_cart,nest_GN}, though the ideas behind the algorithm and of its trust-region variants are  much older \cite{composite_cart,powell_paper,burke_com,yu_super,steph_conv_comp,fletcher_back,pow_glob}.
The scheme \eqref{eqn:prox-subprob} reduces to the popular prox-gradient algorithm for additive composite minimization, while
for nonlinear least squares, the algorithm is closely related to the Gauss-Newton algorithm \cite[Section 10]{NW}. 


Our work focuses on global efficiency estimates of numerical methods. Therefore, in line with standard assumptions in the literature, we assume that $h$ is $L$-Lipschitz  and the Jacobian map $\nabla c$ is $\beta$-Lipschitz. 
As in the analysis of the prox-gradient method in Nesterov \cite{smooth_min_nonsmooth,intro_lect}, it is convenient to measure the progress of the prox-linear method in terms of the scaled steps, called the {\em prox-gradients}: $$\mathcal{G}_t(x_k):=t^{-1}(x_k-x_{k+1}).$$ 
A short argument shows that with the optimal choice $t=(L\beta)^{-1}$, the prox-linear algorithm will find a point $x$ satisfying $\|\mathcal{G}_{\frac{1}{L\beta}}(x)\|\leq \varepsilon$ after at most $\mathcal{O}(\frac{L\beta}{\varepsilon^{2}}(F(x_0)-\inf F))$ iterations; see e.g. \cite{composite_cart,prox_error}. We mention in passing that iterate convergence under the K{\L}-inequality was recently shown in \cite{glob_conv,compBP}, while local linear/quadratic rates under appropriate regularity conditions were proved in \cite{prox_error,quad_conv,nest_GN}. 
The contributions of our work are as follows.

\begin{enumerate}[topsep=0pt,itemsep=0ex,partopsep=1ex,parsep=1ex]
	\item {\bf (Prox-gradient and the Moreau envelope)} 
	The size of the prox-gradient $\|\mathcal{G}_{t}(x_k)\|$ plays a basic role in this work. In particular, all convergence rates are stated in terms of this quantity. Consequently, it is important to understand precisely what this quantity entails about the quality of the point $x_k$ (or $x_{k+1}$).
	For additive composite problems \eqref{eqn:add_comp_intro}, the situation is clear. Indeed, the proximal gradient method generates iterates satisfying $F'(x_{k+1};u)\geq -2\|\mathcal{G}_{\frac{1}{\beta}}(x_k)\|$ for all unit vectors $u$, where $F'(x;u)$ is the directional derivative of $F$ at $x$ in direction $u$ \cite[Corollary 1]{nest_conv_comp}. Therefore, a small prox-gradient $\|\mathcal{G}_{\frac{1}{\beta}}(x_k)\|$ guarantees that $x_{k+1}$ is nearly stationary for the problem, since the derivative of $F$ at $x_{k+1}$ in any unit direction is nearly nonnegative. 
	For the general composite class \eqref{eqn:key_comp}, such a conclusion is decisively false: the prox-linear method will typically generate an iterate sequence along which $F$ is differentiable with gradient norms $\|\nabla F(x_k)\|$ uniformly bounded away from zero, in spite of the norms $\|\mathcal{G}_{\frac{1}{L\beta}}(x_k)\|$ tending to zero.\footnote{See the beginning of Section~\ref{subsec:stepsize_stat} for a simple example of this type of behavior.} Therefore, one must justify the focus on the norm $\|\mathcal{G}_{\frac{1}{L\beta}}(x_k)\|$ by other means.
	To this end, our first contribution is Theorem~\ref{thm:prox_approx}, where we prove that  $\|\mathcal{G}_{\frac{1}{L\beta}}(x)\|$ is proportional to the norm of the true gradient of the Moreau envelope of $F$ --- a well studied smooth approximation of $F$ having identical stationary points.
	An immediate consequence is that even though $x$ might not be nearly stationary for $F$,  a small prox-gradient  $\|\mathcal{G}_{\frac{1}{L\beta}}(x)\|$ guarantees that $x$ is near some point $\hat x$ (the proximal point), which is nearly stationary for $F$. In this sense, a small prox-gradient $\|\mathcal{G}_{\frac{1}{L\beta}}(x)\|$ is informative about the quality of $x$. We note  that an earlier version of this conclusion based on a more indirect argument, appeared in \cite[Theorem 5.3]{prox_error}, and was used to derive linear/quadratic rates of convergence for the prox-linear method under suitable regularity conditions.
	
	\item {\bf (Inexactness and first-order methods)}  For the general composite class \eqref{eqn:key_comp}, coping with inexactness in the proximal subproblem solves \eqref{eqn:prox-subprob} is unavoidable. We perform an inexact analysis of the prox-linear method based on two natural models of inexactness: $(i)$ near-optimality in function value and $(ii)$ near-stationarity in the dual. 		
	Based on the inexact analysis, it is routine to derive overall efficiency estimates for the prox-linear method, where the proximal subproblems are themselves solved by first-order algorithms.  
Unfortunately, the efficiency estimates we can prove for such direct methods appear to either be unsatisfactory or the algorithms themselves appear not to be very practical (Appendix~\ref{sec:app_dual_meth}). Instead, we present algorithms based on a smoothing technique.

	\item {\bf (Complexity of first-order methods through smoothing)} 
	Smoothing is a common technique in nonsmooth optimization. The seminal paper of Nesterov \cite{smooth_min_nonsmooth}, in particular, derives convergence guarantees for algorithms based on infimal convolution smoothing in structured convex optimization. 
	In the context of 
	the composite class \eqref{eqn:key_comp}, 
	smoothing is indeed appealing. In the simplest case, one replaces the function $h$ by a smooth approximation and solves the resulting smooth problem instead. 
	
	We advocate running an inexact prox-linear method on the smooth approximation, with the proximal subproblems approximately solved by fast-gradient methods. To state the resulting complexity bounds, let us suppose that there is a finite upper bound  on the operator norms $\|\nabla c(x)\|_{\textrm{op}}$ over all $x$ in the domain of $g$, and denote it by $\displaystyle\|\nabla c\|$.\footnote{It is sufficient for the inequality $\|\nabla c\|\geq \|\nabla c(x_k)\|_{\textrm{op}}$ to hold just along the iterate sequence $x_k$ generated by the method; in particular, $\|\nabla c\|$ does not need to be specified when initializing the algorithm.}
	We prove that the outlined scheme requires at most 
	\begin{equation}\label{eqn:eff_intro_baseline}
		\widetilde{\mathcal{O}}\left(\frac{L^2\beta\|\nabla c\|}{\varepsilon^{3}}(F(x_0)-\inf F)\right)
	\end{equation}
	evaluations of $c(x)$, matrix vector products $\nabla c(x)v$, $\nabla c(x)^Tw$, and proximal operations of $g$ and $h$  to find a point $x$ satisfying $\|\mathcal{G}_{\frac{1}{L\beta}}(x)\|\leq \varepsilon$. 
	To the best of our knowledge, this is the best known complexity bound for the problem class \eqref{eqn:key_comp} among first-order methods. Here, the symbol $\widetilde{\mathcal{O}}$ hides logarithmic terms.\footnote{If a good estimate on the gap $F(x_0)-\inf F$ is known, the logarithmic terms can be eliminated by a different technique, described in Appendix~\ref{sec:app_dual_meth}.}
	
	\item  {(\bf Complexity of finite-sum problems)}
	Common large-scale problems in machine learning and high dimensional statistics lead to minimizing an average of a large number of functions. Consequently, we consider the finite-sum extension of the composite problem class,
	$$\min_{x} ~ F(x):=\frac{1}{m}\sum_{i=1}^m h_i(c_i(x))+g(x),$$
	where now each $h_i$ is $L$-Lipschitz and each $c_i$ is $C^1$-smooth with $\beta$-Lipschitz gradient. 
	Clearly, the finite-sum problem is itself an instance of \eqref{eqn:key_comp} under the identification
	$h(z_i,\ldots,z_m):=\frac{1}{m}\sum_{i=1}^m h_i(z_i)$ and $c(x):= (c_1(x),\ldots,c_m(x))$. In this structured context, however, the complexity of an algorithm is best measured in terms of the number of individual evaluations $c_i(x)$ and $\nabla c_i(x)$, dot-product evaluations $\nabla c_i(x)^Tv$, and proximal operations $\prox_{th_i}$ and $\prox_{tg}$ the algorithm needs  to find  a point $x$ satisfying 
	$\|\mathcal{G}_{\frac{1}{{L}{\beta}}}(x)\|\leq \varepsilon$. A routine computation shows that the efficiency estimate \eqref{eqn:eff_intro_baseline} of the basic inexact prox-linear method described above leads to the complexity 
	\begin{equation}\label{eqn:stupidInterpr_baseintro}
		\mathcal{\widetilde{O}}\left(  \frac{ m\cdot{L}^2{\beta}{\|\nabla c\|}}{\varepsilon^3}\left (F(x_0)- \inf F \right ) \right),
	\end{equation} 
where abusing notation, we use ${\|\nabla c\|}$ to now denote an upper bound on $\|\nabla c_i(x)\|$ over all $i=1,\ldots,m$ and $x\in \dom g$. We show that a better complexity in expectation is possible by incorporating (accelerated)-incremental methods \cite{catalyst,accsdca,frostig,conjugategradient,accsvrg} for the proximal subproblems.
	The resulting randomized algorithm will generate a  point $x$ satisfying 
	$$\mathbb{E}[\|\mathcal{G}_{\frac{1}{{L}{\beta}}}(x)\|]\leq \varepsilon,$$
	after at most 
	$$
	\mathcal{\widetilde{O}}\left(\left(\frac{m{L}{\beta} }{\varepsilon^2}+\frac{\sqrt{m}\cdot{L}^2{\beta}{\|\nabla c\|}}{\varepsilon^3}\right)\cdot(F(x_0)-\inf F)\right)$$
	basic operations. Notice that the coefficient of $1/\varepsilon^3$ scales at worst as $\sqrt{m}$ --- a significant improvement over \eqref{eqn:stupidInterpr_baseintro}. We note that a different  complementary approach, generalizing stochastic subgradient methods, has been recently pursued by Duchi-Ruan \cite{duchi_ruan}.
	
	\item {\bf (Acceleration)}  
	The final contribution of the paper concerns acceleration of the (exact) prox-linear method.
	For additive composite problems, with $c$ in addition  convex, the prox-gradient method is suboptimal from the viewpoint of computational complexity \cite{intro_lect,complexity}. Accelerated gradient methods, beginning with Nesterov \cite{nest_orig} and extended by Beck-Teboulle \cite{beck} achieve a superior rate in terms of function values. Later,
	Nesterov in \cite[Page 11, item 2]{nest_optima} showed that essentially the same accelerated schemes also achieve a superior rate of $\mathcal{O}((\frac{\beta}{\varepsilon})^{2/3})$ in terms of stationarity, and even a faster rate is possible by first regularizing the problem \cite[Page 11, item 3]{nest_optima}.\footnote{The short paper \cite{nest_conv_comp} only considered smooth unconstrained minimization; however, a minor modification of the proof technique extends to the convex additive composite setting.} Consequently, desirable would be an algorithm that {\em automatically} accelerates in presence of  convexity, while performing no worse than the prox-gradient method on nonconvex instances.
	In the recent manuscript \cite{ghadimi_lan}, Ghadimi and Lan described such a scheme for additive composite problems. Similar acceleration techniques have also been used for exact penalty formulations of nonlinear programs \eqref{eqn:key_comp} with numerical success, but without formal justification; the paper \cite{iter_line_bur} is a good example.

	In this work, we extend the accelerated algorithm of Ghadimi-Lan~\cite{ghadimi_lan}  for additive composite problems to the entire problem class \eqref{eqn:key_comp}, with inexact subproblem solves. Assuming the diameter  $M:=\diam(\dom g)$ is finite, the scheme comes equipped with the guarantee
	\[  \min_{j = 1, \ldots, k}\, \norm{\mathcal{G}_{\frac{1}{2L\beta}}(x_j)}^2 \le (L\beta M)^2\cdot \mathcal{O}\left(\frac{1}{k^3}
	+  \frac{ c_2 }{ k^2 } + \frac{c_1  }{ k }
	\right),
	\]
	where the constants $0\leq c_1\leq c_2\leq 1$ quantify ``convexity-like behavior'' of the composition.	The inexact analysis of the proposed accelerated method based on functional errors is inspired by and shares many features with the seminal papers \cite{inex_comput,catalyst} for convex additive composite problems \eqref{eqn:add_comp_intro}.

\end{enumerate}

The outline of the manuscript is as follows. Section~\ref{sec:not} records basic notation that we use throughout the paper.  In Section~\ref{sec:comp_prob_class}, we introduce the composite problem class, first-order stationarity, and the basic prox-linear method. Section~\ref{subsec:stepsize_stat} discusses weak-convexity of the composite function and the relationship of the prox-gradient with the gradient of the Moreau envelope. Section~\ref{sec:inex_prox_lin} analyzes inexact prox-linear methods based on two models of inexactness: near-minimality and dual near-stationarity. In Section~\ref{sec:overall_comp}, we derive efficiency estimates of first-order methods for the composite problem class, based on a smoothing strategy. Section~\ref{sec:fin_sum_prob} extends the aforementioned results to problems where one seeks to minimize a finite average of  composite functions.
The final Section~\ref{sec:accel_conv_comp}  
discusses an inertial prox-linear algorithm that is adaptive to convexity.
\section{Notation}\label{sec:not}
The notation we follow is standard. Throughout, we consider a Euclidean space, denoted by $\R^d$, with an inner product $\langle\cdot,\cdot \rangle$ and the induced norm $\|\cdot\|$. Given a linear map $A\colon\R^d\to\R^l$, the adjoint $A^*\colon\R^l\to\R^d$ is the unique linear map satisfying
$$\langle Ax,y\rangle= \langle x,A^*y\rangle \qquad \textrm{for all } x\in \R^d, y\in \R^l.$$
The operator norm of $A$, defined as $\displaystyle \|A\|_{\rm op}:=\max_{\|u\|\leq 1} \|Au\|$, coincides with the maximal singular value of $A$ and satisfies  $\|A\|_{\textrm{op}}=\|A^*\|_{\textrm{op}}$.
For any map $F\colon\R^d\to\R^m$, we set
$$\lip(F):=\sup_{x\neq y}~ \frac{\|F(y)-F(x)\|}{\|y-x\|}.$$ 
In particular, we say that  $F$ is $L$-Lipschitz continuous, for some real $L\geq 0$, if the inequality $\lip(F)\leq L$ holds.
Given a set $Q$ in $\R^d$, the
{\em distance} and {\em projection} of a point $x$ onto $Q$ are given by
\begin{align*}
	\dist(x;Q)&:=\inf_{y\in Q}~ \|y-x\|,\qquad \proj(x;Q):=\argmin_{y\in Q}~ \|y-x\|,
\end{align*}
respectively.
The extended-real-line is the set $\overline{\R}:=\R\cup\{\pm\infty\}$. The {\em domain} and the {\em epigraph} of any function $f\colon\R^d\to\overline \R$ are the sets
\begin{align*}
	\textrm{dom}\, f&:=\{x\in \R^d: f(x)<+\infty\},\qquad
	\textrm{epi}\, f:=\{(x,r)\in \R^d\times \R: f(x)\leq r\},
\end{align*} 
respectively. We say that $f$ is {\em closed} if its epigraph, $\textrm{epi}\, f$, is a closed set. Throughout, we will assume that all functions that we encounter are {\em proper}, meaning they have nonempty domains and never take on the value $-\infty$.
The indicator function of a set $Q\subseteq\R^d$, denoted by $\delta_Q$, is defined to be zero on $Q$ and $+\infty$ off it.

Given a convex function $f\colon\R^d\to\overline{\R}$, a vector $v$ is called a {\em subgradient} of $f$ at a point $x\in \dom f$ if the inequality
\begin{equation}\label{eqn:subgrad_conv}
	f(y)\geq f(x)+\langle v,y-x\rangle\qquad \textrm{ holds for all } y\in \R^d.
\end{equation}
The set of all subgradients of $f$ at $x$ is denoted by $\partial f(x)$, and is called the {\em subdifferential} of $f$ at $x$.
For any point $x\notin \dom f$, we set $\partial f(x)$ to be the empty set.
With any convex function $f$, we associate the {\em Fenchel conjugate}  $f^{\star}\colon\R^d\to\overline\R$, defined by 
$$f^{\star}(y):=\sup_x\, \{\langle y,x\rangle-f(x)\}.$$
If $f$ is closed and convex, then equality $f=f^{\star \star}$ holds and we have the equivalence
\begin{equation}\label{eqn:reverse_subgrad_conj}
	y\in \partial f(x) \qquad \Longleftrightarrow \qquad x\in \partial f^{\star}(y).
\end{equation}


For any function $f$ and real $\nu>0$, the {\em Moreau envelope} and the 
{\em proximal mapping}  are defined by 
\begin{align*}
	f_{\nu}(x)&:=\inf_{z}\, \left\{ f(z)+\frac{1}{2\nu}\|z-x\|^2\right\}, \\
	\prox_{{\nu}f}(x) &:=\argmin_{z}\, \left\{ f(z)+\frac{1}{2{\nu}}\|z-x\|^2\right\},
\end{align*}
respectively. In particular, the Moreau envelope of an indicator function $\delta_Q$ is simply the map $x\mapsto\frac{1}{2{\nu}}\dist^2(x;Q)$ and the proximal mapping of $\delta_Q$ is the projection $x\mapsto\proj(x;Q)$. The following lemma lists well-known regularization properties of the Moreau envelope.

\begin{lemma}[Regularization properties of the envelope]\label{lem:lip_cont} Let $f\colon\R^d\to\R$ be a closed, convex function. Then $f_{\nu}$ is convex and $C^1$-smooth with
	$$\nabla f_{\nu}(x)=\nu^{-1}(x-\prox_{\nu f}(x))
	\quad \textrm{ and }\qquad \lip(\nabla f_{\nu})\leq \tfrac{1}{\nu}.$$
	If in addition $f$ is $L$-Lipschitz, then the envelope $f_{\nu}(\cdot)$ is $L$-Lipschitz and  satisfies
	\begin{equation}\label{eqn:moreau}
		0\leq f(x)- f_{\nu}(x)\leq \frac{L^2\nu}{2}\qquad \textrm{ for all } x\in\R^d.	
	\end{equation}
	
\end{lemma}
\begin{proof}
	The expression $\nabla f_{\nu}(x) =\nu^{-1}(x-\prox_{\nu f}(x))=\nu^{-1}\cdot \prox_{(\nu f)^*}(x)$ can be found in \cite[Theorem 31.5]{rock}. The inequality $\lip(\nabla f_{\nu})\leq \frac{1}{\nu}$ then follows since the proximal mapping of a closed convex function is 1-Lipschitz \cite[pp. 340]{rock}. The expression \eqref{eqn:moreau} follows from rewriting
	$f_{\nu}(x)=(f^{\star}+\frac{\nu}{2}\|\cdot\|^2)^{\star}(x)=\sup_{z}\, \{\langle x,z\rangle-f^{\star}(z)-\frac{\nu}{2}\|z\|^2\}$ (as in e.g. \cite[Theorem 16.4]{rock}) and noting that the domain of $f^{\star}$ is bounded in norm by $L$. Finally, to see that $f_{\nu}$ is $L$-Lipschitz, observe 
	$\nabla f_{\nu}(x)\in \partial f(\prox_{\nu f}(x))$ for all $x$, and hence $\|\nabla f_{\nu}(x)\|\leq \sup\{ \|v\|:y\in\R^d, v\in \partial f(y)\}\leq L$.
\end{proof}

\section{The composite problem class}\label{sec:comp_prob_class}
This work centers around nonsmooth and nonconvex optimization problems of the form
\begin{equation}\label{eqn:comp2}
	\min_x~ F(x):=g(x)+h(c(x)).
\end{equation}
Throughout, we make the following assumptions on the functional components of the problem:
\begin{enumerate} 
	\item $g\colon\R^d\to\overline{\R}$ is a proper, closed, convex function;
	\item $h\colon\R^m\to\R$ is a convex and $L$-Lipschitz continuous function:
	$$|h(x)-h(y)|\leq L\|x-y\| \qquad \textrm{for all }x,y\in\R^m;$$
	\item $c\colon\R^d\to\R^m$ is a $C^1$-smooth mapping with a $\beta$-Lipschitz continuous Jacobian map:
	$$\|\nabla c(x)-\nabla c(y)\|_{\textrm{op}}\leq \beta\|x-y\|\qquad \textrm{for all }x,y\in\R^d.$$
\end{enumerate}
The values $L$ and $\beta$ will often multiply each other; hence, we define the constant $\mu:=L\beta.$

\subsection{Motivating examples}
It is instructive to consider some motivating examples fitting into the  framework \eqref{eqn:comp2}.

\begin{example}[Additive composite minimization]\label{exa:comp}
	{\rm The most prevalent example of the composite class \eqref{eqn:comp2} is additive composite minimization. In this case, the map $c$ maps to the real line and $h$ is the identity function:
		\begin{equation}\label{eqn:add_comp}
			\min_x~ c(x)+g(x).
		\end{equation}
		Such problems appear often in statistical learning and imaging, for example. 
		Numerous algorithms are available, especially when $c$ is convex, such as proximal gradient methods and their accelerated variants \cite{beck,nest_conv_comp}. We will often compare and contrast techniques for general composite problems \eqref{eqn:comp2} with those specialized to this additive composite setting.
	}
\end{example}

\begin{example}[Nonlinear least squares]\label{exa:nls}
	{\rm
		The composite problem class also captures nonlinear least squares problems with bound constraints:
		\begin{align*}
			\min_x~ \|c(x)\|\qquad \textrm{subject to}\qquad l_i\leq x_i\leq u_i \quad\textrm{ for }i=1,\ldots,m.
		\end{align*}
		Gauss-Newton type algorithm \cite{lev,marq,more_lev_marq} are often the methods of choice for such problems.
	}
\end{example}

\begin{example}[Exact penalty formulations]\label{exa:ep}
	{\rm
		Consider a nonlinear optimization problem:
		\begin{align*}
			\min_x~ \{f(x): G(x)\in \mathcal{K}\},
		\end{align*}
		where $f\colon\R^d\to\R$ and $G\colon\R^d\to\R^m$ are smooth mappings and 
		$\mathcal{K}\subseteq \R^m$ is a closed convex cone.
		An accompanying {\em penalty formulation} -- ubiquitous in nonlinear optimization \cite{PG,burke_exact,CC,eremin,KNITRO} -- takes the form 
		$$\min_x~ f(x)+\lambda \cdot \theta_{\mathcal{K}}(G(x)),$$
		where $\theta_{\mathcal K}\colon\R^m\to\R$ is a nonnegative convex function that is zero only on $\mathcal{K}$ and $\lambda>0$
		is a penalty parameter. For example, $\theta_{\mathcal{K}}(y)$ is often the  distance  of $y$ to the convex cone $\mathcal{K}$ in some norm. This is an example of 
		\eqref{eqn:comp2} under the identification $c(x)=(f(x),G(x))$ and $h(f,G)=f+\lambda \theta_{\mathcal K}(G)$. 
	}
\end{example}

\begin{example}[Statistical estimation]\label{exa:huber}
	{\rm
		Often, one is interested in minimizing an error between a nonlinear process model $G(x)$ and observed data $b$ through a misfit measure $h$. The resulting problem takes the form
		$$\min_x~ h\big(b-G(x)\big)+g(x),$$
		where $g$ may be a convex surrogate encouraging prior structural information on $x$, such as the $l_1$-norm, squared $l_2$-norm, or the indicator of the nonnegative orthant. The misfit $h=\|\cdot\|_2$, in particular, appears in nonlinear least squares. The $l_1$-norm  $h=\|\cdot\|_1$ is used in the Least Absolute Deviations (LAD) technique in regression \cite{LAD1,LAD2},  Kalman smoothing with impulsive disturbances \cite{kalman}, and for robust phase retrieval \cite{duchi_ruan}.
		
		Another popular class of misfit measures $h$ is a sum $h=\sum_i h_{\kappa}(y_i)$ of Huber functions
		\[ h_{\kappa}(\tau)=\begin{cases} 
		\frac{1}{2\kappa} \tau^2 &, \tau\in [-\kappa,\kappa]\\
		|\tau| - \frac{\kappa}{2} &, \textrm{otherwise} 
		\end{cases}
		\]
		The Huber function figures prominently in robust regression \cite{Clark85,hub_num,Hub,LiW98}, being much less sensitive to outliers than the least squares penalty due to its linear tail growth. The function $h$ thus defined is smooth with $\lip(\nabla h)\sim 1/\kappa$. Hence, in particular, the term $h(b-G(x))$ can be treated as a smooth term reducing to the setting of additive composite minimization (Example~\ref{exa:comp}). On the other hand, we will see that because of the poor conditioning of the gradient $\nabla h$, methods that take into account the non-additive composite structure can have better efficiency estimates.
	}
\end{example}

\begin{example}[Grey-box minimization]\label{ex:grey}
	{\rm In industrial applications, one is often interested in functions that are available only {\em implicitly}. For example, function and derivative evaluations may require execution of an expensive simulation. Such problems often exhibit an underlying composite structure $h(c(x))$.
		The penalty function $h$ is known (and chosen) explicitly and is simple, whereas the  mapping $c(x)$ and the Jacobian $\nabla c(x)$ might only be available through a simulation. Problems of this type are sometimes called {\em grey-box minimization problems}, in contrast to black-box minimization. The explicit separation of the hard-to-compute mapping $c$ and the user chosen penalty $h$ can help in designing algorithms. See for example Conn-Scheinberg-Vicente \cite{DFO_book} and Wild \cite{pounders}, and references therein.
	}
\end{example}

\subsection{First-order stationary points for composite problems}\label{sec:subdiff_WC}
Let us now explain the goal of algorithms for the problem class \eqref{eqn:comp2}.
Since the optimization problem \eqref{eqn:comp2} is nonconvex, it is natural to seek points $x$ that are only first-order stationary. 
One makes this notion precise through subdifferentials (or generalized derivatives), which have a very explicit representation for our problem class. We recall here the relevant definitions, following the monographs of Mordukhovich \cite{mord1} and Rockafellar-Wets \cite{RW98}.

Consider an arbitrary function $f\colon\R^d\to\overline  \R$ and a point $\bar x$ with $f(\bar x)$ finite. The  {\em Fr\'{e}chet subdifferential} of $f$ at $\bar x$, denoted $\hat \partial f(\bar x)$, is the set of all vectors $v$ satisfying 
$$f(x)\geq f(\bar x)+\langle v,x-\bar x\rangle+o(\|x-\bar x\|)\qquad \textrm{ as }x\to \bar x.$$ Thus the inclusion $v\in\hat\partial f(\bar x)$ holds precisely when the affine function $x\mapsto f(\bar x)+\langle v,x-\bar x\rangle$ underestimates $f$ up to first-order near $\bar x$. In general, the limit of Fr\'{e}chet subgradients $v_i\in \hat\partial f(x_i)$, along a sequence $x_i\to\bar x$, may not be a Fr\'{e}chet subgradient at the limiting point $\bar x$. Hence, one formally enlarges the Fr\'{e}chet subdifferential and defines the {\em limiting subdifferential} of $f$ at $\bar x$, denoted $\partial f(\bar x)$, to consist of all vectors $v$ for which there exist sequences $x_i$ and $v_i$, satisfying $v_i\in \partial f(x_i)$ and $(x_i,f(x_i),v_i)\to (\bar x, f(\bar x),v)$. We say that $x$ is {\em stationary} for $f$ if the inclusion $0\in \partial f(x)$ holds. 

For convex functions $f$, the subdifferentials  $\hat \partial f(x)$ and $\partial f(x)$ coincide with the subdifferential in the sense of convex analysis \eqref{eqn:subgrad_conv}, while for $C^1$-smooth functions $f$, they consist only of the gradient $\nabla f(x)$.
Similarly, the situation simplifies for the composite problem class \eqref{eqn:comp2}: the two subdifferentials $\hat\partial F$ and $\partial F$ coincide and admit an intuitive representation through a chain-rule \cite[Theorem 10.6, Corollary 10.9]{RW98}.

\begin{theorem}[Chain rule]
	For the composite function $F$, defined in \eqref{eqn:comp2}, the Fr\'{e}chet and limiting subdifferentials coincide and admit the  representation
	$$\partial F(x)=\partial g(x)+\nabla c(x)^*\partial h(c(x)).$$
\end{theorem}
In summary, the algorithms we consider aim to find stationary points of $F$, i.e. those points $x$ satisfying $0\in \partial F(x)$. 
In ``primal terms'', it is worth noting that a point $x$ is stationary for $F$ if and only if the directional derivative of $F$ at $x$ is nonnegative in every direction \cite[Proposition 8.32]{RW98}. More precisely, the equality holds:
\begin{equation}\label{eqn:subdif_direc_der}
	\dist(0;\partial F(x))=-\inf_{v:\, \|v\|\leq 1} F'( x;v),
\end{equation}
where $F'( x;v)$ is the directional derivative of $F$ at $x$ in direction $v$ \cite[Definition 8.1]{RW98}. 

\subsection{The prox-linear method}\label{sec:prox_lin_acc}
The basic algorithm we rely on for the composite problem class is the so-called prox-linear method. To motivate this scheme, let us first consider the setting of additive composite minimization \eqref{eqn:add_comp}.  
The most basic algorithm in this setting is the {\em proximal gradient method} \cite{beck,nest_conv_comp}
\begin{equation}\label{eqn:prox_grad}
	x_{k+1}:=\argmin_{x}~ \left\{c(x_k)+\langle \nabla c(x_k),x-x_k\rangle+g(x)+\frac{1}{2t}\|x-x_k\|^2\right\},
\end{equation}
or equivalently
$$x_{k+1}=\prox_{tg}\left(x_k-t\nabla c(x_k)\right).$$
Notice that an underlying assumption here is that the proximal map $\prox_{tg}$ is computable.

Convergence analysis of the prox-gradient algorithm derives from the fact that the function minimized in \eqref{eqn:prox_grad} is an upper model of $F$ whenever $t\leq\beta^{-1}$.  This majorization viewpoint quickly yields an algorithm for the entire problem class \eqref{eqn:comp2}. The so-called {\em prox-linear algorithm} iteratively linearizes the map $c$ and solves a proximal subproblem. To formalize the method, we use  the following notation.
For any points $z, y \in \R^d$ and a real $t > 0$, define 
\begin{align*}
	F (z; y)\, &:=\, g(z) +  h \Big (c(y) + \nabla c(y) (z-y) \Big ),\\
	F_t (z; y) \,&:=\, F(z;y) + \frac{1}{2t} \norm{z-y}^2,\\
	S_t(y)\,&:=\,\argmin_z~ F_t(z; y).
\end{align*}

Throughout the manuscript, we will routinely use the following estimate on the error in approximation $|F(z)-F(z;y)|$. We provide a quick proof for completeness.
\begin{lemma}\label{lem:upper_lower}
	For all $x,y\in \dom g$,  the inequalities hold:
	\begin{equation}\label{eqn:ineq_min_max}
		-\frac{\mu}{2}\|z-y\|^2\leq F(z)-F(z;y)\leq \frac{\mu}{2}\|z-y\|^2.
	\end{equation}
\end{lemma}
\begin{proof}
	Since $h$ is $L$-Lipschitz, we have $|F(z)-F(z;y)|\leq L\big\|c(z)-\big(c(y) + \nabla c(y) (z-y)\big)\big\|$.
	The fundamental theorem of calculus, in turn, implies
	\begin{equation*}
		\begin{aligned}
			\big\|c(z)-\big(c(y) + \nabla c(y) (z-y)\big)\big\|&=\left\|\int_{0}^1 \Big(\nabla c(y+t(z-y))-\nabla c(y)\Big)(z-y)\,dt\right\|\\
			&\leq \int_{0}^1 \left\|\nabla c(y+t(z-y))-\nabla c(y)\right\|_{\textrm{op}} \|z-y\|\,dt\\
			&\leq \beta \|z-y\|^2 \left(\int_{0}^1 t \, dt\right)=\frac{\beta}{2} \|z-y\|^2.
		\end{aligned}
	\end{equation*}
	The result follows.
\end{proof}

In particular, Lemma~\ref{lem:upper_lower} implies that $F_{t}(\cdot;y)$ is an upper model for $F$ for any $t\leq \mu^{-1}$, meaning $F_{t}(z;y)\geq F(z)$ for all points $y,z\in \dom g$.
The {\em prox-linear method}, formalized in Algorithm~\ref{alg: prox_lin}, is then simply the recurrence
$x_{k+1}=S_{t}(x_k).$ Notice that we are implicitly assuming here that the proximal subproblem \eqref{eqn:simple_prox} is solvable. We will discuss the impact of an inexact evaluation of $S_{t}(\cdot)$ in Section~\ref{sec:inex_prox_lin}.
Specializing to the  additive composite setting \eqref{eqn:add_comp}, equality
$S_t(x)=\prox_{tg}(x-t\nabla c(x))$
holds and the prox-linear method reduces to the familiar prox-gradient iteration \eqref{eqn:prox_grad}.

%
%
%
%
%
%
%

{\LinesNotNumbered
	\begin{algorithm}[h!]
		\SetKw{Null}{NULL}
		\SetKw{Return}{return}
		\Initialize{A point $x_0 \in \text{dom} \,
			g$ and a real $t>0$.}
		{\bf Step k:} ($k\geq 0)$
		Compute 
		\begin{equation}\label{eqn:simple_prox}
			x_{k+1}=\argmin_{x} \left\{g(x)+ h\Big(c(x_k)+\nabla c(x_k)(x-x_k)\Big)+\frac{1}{2t}\|x-x_k\|^2\right\}.
		\end{equation}
		\caption{Prox-linear method}
		\label{alg: prox_lin}
\end{algorithm}}

The convergence rate of the prox-linear method is best stated in terms of the {\em prox-gradient} mapping
$$\mathcal{G}_{t}(x):=t^{-1}(x-S_t(x)).$$
Observe that the optimality conditions for the proximal subproblem $\min_z F_t(z;x)$ read
$$\mathcal{G}_t(x)\in \partial g(S_t(x))+\nabla c(x)^*\partial h(c(x)+\nabla c(x)(S_t(x)-x)).$$
In particular, it is straightforward to check that with any $t>0$, a point $x$ is stationary for $F$ if and only if equality $\mathcal{G}_t(x)=0$ holds. Hence, the norm $\|\mathcal{G}_t(x)\|$ serves as a measure of ``proximity to stationarity''. In Section~\ref{subsec:stepsize_stat}, we will establish  a much more rigorous justification for why the norm $\|\mathcal{G}_{t}(x)\|$ provides a reliable basis for judging the quality of the point $x$.
Let us review here the rudimentary convergence guarantees of the method in terms of the prox-gradient, as presented for example in \cite[Section 5]{prox_error}. We provide a quick proof for completeness.
\begin{proposition}[Efficiency of the pure prox-linear method]\label{prop:basic_prox_lin}
	Supposing $t\leq \mu^{-1}$, the iterates generated by Algorithm~\ref{alg: prox_lin} satisfy
	\begin{align*}
		\min_{j=0, \hdots, N-1} \norm{\mathcal{G}_t(x_j)}^2  \le \frac{2t^{-1} \big ( F(x_0)-F^*\big)}{N},
	\end{align*}
	where we set $\displaystyle F^*:=\lim_{N\to\infty} F(x_N)$.
\end{proposition}
\begin{proof}
	Taking into account that  $F_t(\cdot;x_k)$ is strongly convex with modulus $1/t$, we obtain
	\begin{align*}
		F(x_k) = F_t(x_k; x_k) &\ge  F_t(x_{k+1}; x_k)  + \tfrac{t}{2}
		\norm{\mathcal{G}_t(x_k)}^2 \ge F(x_{k+1})  + \tfrac{t}{2}
		\norm{\mathcal{G}_t(x_k)}^2.
	\end{align*}
	Summing the inequalities yields 
	\begin{align*}
		\min_{j=0, \hdots, N-1} \norm{\mathcal{G}_t(x_j)}^2  \le \frac{1}{N}
		\sum_{j=0}^{N-1}
		\norm{\mathcal{G}_t(x_j)}^2
		&\le \frac{2t^{-1} \big ( F(x_0)-F^*\big)}{N},
	\end{align*}
	as claimed.
\end{proof}

\section{Prox-gradient size $\|\mathcal{G}_{t}\|$ and approximate stationarity}\label{subsec:stepsize_stat}
Before continuing the algorithmic development, let us take a closer look at what the measure $\|\mathcal{G}_{t}(x)\|$ tells us about ``near-stationarity'' of the point $x$.
Let us first consider the additive composite setting \eqref{eqn:add_comp}, where the impact of the measure $\|\mathcal{G}_{t}(x)\|$ on near-stationarity is well-understood.
As discussed on page \pageref{alg: prox_lin}, the prox-linear method reduces to the prox-gradient recurrence
$$x_{k+1}=\prox_{g/\beta}\left(x_k-\tfrac{1}{\beta}\cdot\nabla  c(x_k)\right).$$
First-order optimality conditions for the proximal subproblem amount to
the inclusion
$$\mathcal{G}_{\frac{1}{\beta}}(x_k)\in \nabla c(x_k)+\partial g(x_{k+1}),$$
or equivalently
$$\mathcal{G}_{\frac{1}{\beta}}(x_k)+(\nabla c(x_{k+1})-\nabla c(x_k))\in \nabla c(x_{k+1})+\partial g(x_{k+1}).$$
Notice that the right-hand-side is exactly $\partial F(x_{k+1})$. Taking into account that $\nabla c$ is $\beta$-Lipschitz, we deduce
\begin{equation}\label{eqn:stat_est}
	\begin{aligned}
		\dist(0;\partial F(x_{k+1}))&\leq \|\mathcal{G}_{\frac{1}{\beta}}(x_k)\|+\|\nabla c(x_{k+1})-\nabla c(x_k)\|\\
		&\leq 2\|\mathcal{G}_{\frac{1}{\beta}}(x_k)\|.
	\end{aligned}
\end{equation}
Thus the inequality  $\|\mathcal{G}_{\frac{1}{\beta}}(x_k)\|\leq \varepsilon/2$ indeed guarantees that $x_{k+1}$ is nearly stationary for $F$ in the sense that $\dist(0;\partial F(x_{k+1}))\leq \varepsilon$. Taking into account \eqref{eqn:subdif_direc_der}, we deduce the bound on directional derivative $F'(x;u)\geq -\varepsilon$ in any unit direction $u$.
With this in mind, the guarantee of Proposition~\ref{prop:basic_prox_lin} specialized to the prox-gradient method can be found for example in \cite[Theorem 3]{nest_conv_comp}. 

The situation is dramatically different for the general composite class \eqref{eqn:comp2}. 
When $h$ is nonsmooth, the quantity $\dist(0;\partial F(x_{k+1}))$ will typically  not even tend to zero in the limit, in spite of $\|\mathcal{G}_{\frac{1}{\beta}}(x_k)\|$ tending to zero. 
For example, the prox-linear algorithm applied to the univariate function
$f(x)=|x^2-1|$ and initiated at  $x>1$, will generate a decreasing sequence $x_k\to 1$ with $f'(x_k)\to 2$.\footnote{Notice $f$ has three stationary points $\{-1,0,1\}$. Fix $y>1$ and observe that $x$ minimizes $ f_{t}(\cdot;y)$ if and only if $\tfrac{y-x}{2ty}\in\partial |\cdot|(y^2-1+2y(x-y))$.
	Hence $\tfrac{y-x}{2ty}\cdot(y^2-1+2y(x-y))\geq 0$.
	The inequality $x\leq 1$ would immediately imply a contradiction. Thus the inequality $x_0>1$ guarantees $x_k>1$ for all $k$. The claim follows.} 

Thus we must look elsewhere for an interpretation of the quantity $\|\mathcal{G}_{\frac{1}{\mu}}(x_k)\|$. We will do so by focusing on the Moreau envelope $x \mapsto F_{\frac{1}{2\mu}}(x)$ --- a function that serves as a $C^1$-smooth approximation of $F$ with the same stationary points.
We argue in Theorem~\ref{thm:prox_approx} that the norm of the prox-gradient $\|\mathcal{G}_{\frac{1}{\mu}}(x_k)\|$ is informative because $\|\mathcal{G}_{\frac{1}{\mu}}(x_k)\|$ is proportional to the norm of the true gradient of the Moreau envelope $\|\nabla F_{\frac{1}{2\mu}}(x)\|$. Before proving this result, we must first establish some basic properties of the Moreau envelope, which will follow from weak convexity of the composite function $F$; this is the content of the following section.

\subsection{Weak convexity and the Moreau envelope of the composition }
We will need the following standard definition.

\begin{defn}[Weak convexity]
	{\rm
		We say that a function  $f\colon\R^d\to\overline{\R}$ is $\rho$-{\em weakly convex on a set }$U$ if 
		for any  points $x,y\in U$  and $a\in[0,1]$, the approximate secant inequality holds: $$f(ax+(1-a)y)\leq a f(x)+(1-a)f(y)+\rho a(1-a)\|x-y\|^2.$$}
\end{defn} 

It is well-known that for a locally Lipschitz function $f\colon\R^d\to\R$, the following are equivalent; see e.g. \cite[Theorem 3.1]{fill_grap}.
\begin{enumerate}
	\item {\bf (Weak convexity)} $f$ is $\rho$-weakly convex on $\R^d$.
	\item {\bf (Perturbed convexity)} The function $f+\frac{\rho}{2}\|\cdot\|^2$ is convex on $\R^d$.
	\item {\bf(Quadratic lower-estimators)} For any  $x,y\in \R^d$  and $v\in \partial f(x)$, the inequality 
	$$f(y)\geq f(x)+\langle v,y-x\rangle -\frac{\rho}{2}\|y-x\|^2 \qquad \textrm{ holds}.$$
\end{enumerate}

In particular, the following is true.

\begin{lemma}[Weak convexity of the composition]\label{lem:prox_reg}\hfill\\
	The function $h\circ c$ is $\rho$-weakly convex on $\R^d$ for some $\rho\in [0,\mu]$.
\end{lemma}
\begin{proof}
	To simplify notation, set $\Phi:=h\circ c$.
	Fix two points $x,y\in \R^d$ and a vector $v\in \partial \Phi(x)$.
	We can write $v=\nabla c(x)^*w$ for some vector $w\in \partial h(c(x))$. Taking into account convexity of $h$ and the inequality $\|c(y)-c(x)-\nabla c(x)(y-x)\|\leq \frac{\beta}{2}\|y-x\|^2$, we then deduce 
	\begin{align*}
		\Phi(y)=h(c(y))\geq h(c(x))+\langle w,c(y)-c(x)\rangle &\geq \Phi(x)+\langle w,\nabla c(x)(y-x)\rangle -\frac{\beta\|w\|}{2}\|y-x\|^2\\
		&\geq \Phi(x)+\langle v,y-x\rangle-\frac{\mu}{2}\|y-x\|^2.
	\end{align*}
	The result follows.
\end{proof}

Weak convexity of $F$ has an immediate consequence on the Moreau envelope $F_{\nu}$. 

\begin{lemma}[Moreau envelope of the composite function]\label{lem:mor_env}
	Fix  $\nu\in (0,1/\mu)$. Then the proximal map $\prox_{\nu F}(x)$ is well-defined and single-valued, while the Moreau envelope $F_{\nu}$ is $C^1$-smooth with gradient
	\begin{equation}\label{eqn:der_prox}
		\nabla F_{\nu}(x)=\nu^{-1}(x-\prox_{\nu F}(x)).
	\end{equation}
	Moreover,  stationary points of $F_{\nu}$ and of $F$ coincide.
\end{lemma}
\begin{proof}
	Fix  $\nu\in (0,1/\mu)$.
	Lemma~\ref{lem:prox_reg} together with \cite[Theorem 4.4]{prox_reg} immediately imply that $\prox_{\nu F}(x)$ is well-defined and single-valued, while the Moreau envelope $F_{\nu}$ is $C^1$-smooth with gradient given by \eqref{eqn:der_prox}.  Equation \eqref{eqn:der_prox} then implies that $x$ is  stationary for $F_{\nu}$ if and only if 
	$x$ minimizes the function $\varphi(z):= F(z)+\frac{1}{2\nu}\|z-x\|^2$. Lemma~\ref{lem:prox_reg} implies that $\varphi$ is strongly convex, and therefore the unique minimizer $z$ of $\varphi$ is characterized by $\nu^{-1}(x-z)\in \partial F(z)$. Hence stationary points of $F_{\nu}$ and of $F$ coincide. 
\end{proof}

Thus  for $\nu\in (0,1/\mu)$, stationary points of $F$ coincide with those of the $C^1$-smooth function $F_{\nu}$. More useful would be to understand the impact of $\|\nabla F_{\nu}(x)\|$ being small, but not zero.
To this end, observe the following.
Lemma~\ref{lem:mor_env} together with the definition of the Moreau envelope implies that for any $x$, the point $\hat x:=\prox_{\nu F}(x)$ satisfies

\begin{equation}\label{eqn:perturb_opt_orig}
	\left\{\begin{array}{cl}
		\|\hat{x}-x\|&\leq  \nu\|\nabla F_{\nu}(x)\|,\\ 
		F(\hat x)&\leq F(x),\\
		\dist(0;\partial F(\hat{x}))&\leq \|\nabla F_{\nu}(x)\|.
	\end{array}\right. 
\end{equation}
Thus a small gradient $\|\nabla F_{\nu}(x)\|$ implies that $x$ is {\em near} a point $\hat x$ that is {\em nearly stationary} for $F$.


\subsection{Prox-gradient and the gradient of the Moreau envelope}
The final ingredient we need to prove Theorem~\ref{thm:prox_approx} is the following lemma \cite[Theorem 2.4.1]{Borwein-Zhu}; we provide a short proof for completeness. 
\begin{lemma}[Quadratic penalization principle]\label{lem:smooth_var_princ}
	Consider a closed function $f\colon\R^d\to\overline\R$ and suppose the inequality $f(x)-\inf f\leq \varepsilon$ holds for some point $x$ and real $\ve>0$. Then for any $\lambda>0$, the inequality holds:
	$$\|\lambda^{-1}(x-\prox_{\lambda f}(x))\|\leq  \sqrt{\frac{2\varepsilon}{\lambda}}$$
	If $f$ is $\alpha$-strongly convex (possibly with $\alpha=0$), then the estimate improves to 
	$$\|\lambda^{-1}(x-\prox_{\lambda f}(x))\|\leq  \sqrt{\frac{\varepsilon}{\lambda(1+\tfrac{\lambda\alpha}{2})}}.$$
\end{lemma}
\begin{proof}
	Fix a point $\displaystyle y\in\argmin_z \left\{f(z)+\frac{1}{2\lambda}\|z-x\|^2\right\}$. We deduce
	$$f(y)+\frac{1}{2\lambda}\|y-x\|^2\leq f(x)\leq f^*+\ve\leq f(y)+\ve.$$
	Hence we deduce $\lambda^{-1}\|y-x\|\leq \sqrt{\frac{2\varepsilon}{\lambda}}$, as claimed. 
	If $f$ is $\alpha$-strongly convex, then the function
	$z\mapsto f(z)+\frac{1}{2\lambda}\|z-x\|^2$ is $(\alpha+\lambda^{-1})$-strongly convex and therefore
	$$\left(f(y)+\frac{1}{2\lambda}\|y-x\|^2\right)+\frac{\lambda^{-1}+\alpha}{2}\|y-x\|^2\leq f(x)\leq f^*+\ve\leq f(y)+\ve.$$
	The claimed inequality follows along the same lines.
\end{proof}

We can now quantify the precise relationship between the norm of the prox-gradient $\|\mathcal{G}_{t}(x)\|$ and the norm of the true gradient of the Moreau envelope $\|\nabla F_{\frac{t}{1+t \mu}}(x)\|$.

\begin{theorem}[Prox-gradient and near-stationarity]\label{thm:prox_approx}
	For any point $x$ and real constant $t > 0$, the inequality holds: 
	\begin{equation}\label{eqn:comp_prox_lin_reg}
		\tfrac{1}{(1+\mu t) (1+ \sqrt{\mu t})} \left\|\nabla F_{\frac{t}{1+t\mu}}(x)\right\|\leq \left\|\mathcal{G}_{t}(x)\right\|\leq \tfrac{1+2 t \mu }{1+t \mu} \left ( \sqrt{\tfrac{t \mu }{1 + t \mu} } +1 \right ) \left\|\nabla F_{\frac{t}{1+t\mu}}(x)\right\|.
	\end{equation}
\end{theorem}
\begin{proof}
	To simplify notation, throughout the proof set 
	\begin{equation*}
		\begin{aligned}
			{\bar x}&:=S_{t}(x)=\argmin_z  F_{t}(z;x),\\
			\hat x&:=\prox_{\tfrac{t F}{1+t \mu}}(x)=\argmin_z~ \{F(z)+ \tfrac{\mu + t^{-1}}{2} \|z-x\|^2\}.
		\end{aligned}
	\end{equation*}
	Notice that $\hat x$ is well-defined by Lemma~\ref{lem:mor_env}.
	
	We begin by establishing the first inequality in \eqref{eqn:comp_prox_lin_reg}.	
	For any point $z$, we successively deduce
	\begin{equation}\label{eqn:lower-BD}
		\begin{aligned}
			F(z)&\geq F_t(z;x)-\tfrac{\mu + t^{-1}}{2} \|z-x\|^2\geq F_{t}({\bar x};x)+\tfrac{1}{2t}\|{\bar x}-z\|^2-\tfrac{\mu + t^{-1}}{2} \|z-x\|^2\\
			&\geq F({\bar x})+\frac{1}{2t}\|{\bar x}-z\|^2-\tfrac{\mu + t^{-1}}{2} \|z-x\|^2 + \tfrac{t^{-1}-\mu}{2} \|{\bar x}-x \|^2,
		\end{aligned}
	\end{equation}
	where the first and third inequalities follow from \eqref{eqn:ineq_min_max} and the second from strong convexity of $F_{t}(\cdot;x)$.
	
	Define the function $\zeta(z):=F(z)+\tfrac{\mu + t^{-1}}{2}\|z-x\|^2-\frac{1}{2t}\|{\bar x}-z\|^2$ and notice that $\zeta$ is convex by Lemma~\ref{lem:prox_reg}.
	Inequality \eqref{eqn:lower-BD} directly implies
	$$\zeta({\bar x})-\inf \zeta \leq \left (F({\bar x})+\tfrac{\mu + t^{-1}}{2} \|{\bar x}-x\|^2\right )- \left ( F({\bar x}) + \tfrac{t^{-1}-\mu}{2} \| {\bar x} -x \|^2 \right )= \mu\|{\bar x}-x\|^2.$$
	Notice the relation,
	$\prox_{t\zeta}({\bar x})=\prox_{\frac{tF}{1+ t\mu}}(x)=\hat x$. Setting $\lambda:=t$ and $\ve:=\mu\|{\bar x}-x\|^2$ and using Lemma~\ref{lem:smooth_var_princ} (convex case $\alpha=0$) with ${\bar x}$ in place of $x$, we conclude
	$$\sqrt{\tfrac{\mu}{t}}\|{\bar x}-x\|\geq \|t^{-1}({\bar x}-\prox_{t\zeta}({\bar x}))\|=\|t^{-1}({\bar x}-\hat x)\|\geq  \|t^{-1}(x-\hat x)\|-\|t^{-1}({\bar x}-x)\|.$$
	Rearranging and using \eqref{eqn:der_prox} yields the first inequality in  \eqref{eqn:comp_prox_lin_reg}, as claimed. 
	
	We next establish the second inequality in \eqref{eqn:comp_prox_lin_reg}. The argument is in the same spirit as the previous part of the proof.
	For any point $z$, we successively deduce
	\begin{equation}\label{eqn:second_part_main_ineq}
		\begin{aligned}
			F_{t}(z;x)  &\geq (F(z)+\tfrac{\mu + t^{-1}}{2} \|z-x\|^2)-\mu \|z-x\|^2\\
			&\geq F(\hat x)+\tfrac{\mu+t^{-1}}{2} \|\hat x-x\|^2+\tfrac{1}{2t}\|\hat x-z\|^2-\mu\|z-x\|^2,
		\end{aligned}
	\end{equation}
	where the first inequality follows from \eqref{eqn:ineq_min_max} and the second from $t^{-1}$-strong convexity of $z\mapsto F(z)+\tfrac{\mu+t^{-1}}{2}\|z-x\|^2$.
	Define now the function
	$$\Psi(z):=F_{t}(z;x)-\tfrac{1}{2t}\|\hat x-z\|^2+ \mu\|z-x\|^2.$$
	Combining \eqref{eqn:ineq_min_max} and \eqref{eqn:second_part_main_ineq}, we deduce 
	$$\Psi(\hat x)-\inf \Psi\leq \Big(F_{t}(\hat x;x)+\mu\|\hat x -x\|^2\Big) - \Big(F(\hat x)+\tfrac{\mu + t^{-1}}{2}\|\hat x-x\|^2\Big)\leq \mu\|\hat x-x\|^2.$$
	
	Notice that $\Psi$ is strongly convex with parameter $\alpha:=2\mu$.	
	Setting $\varepsilon:=\mu\|\hat x-x\|^2$ and $\lambda=t$, and applying Lemma~\ref{lem:smooth_var_princ} with $\hat x$ in place of $x$, we deduce
	\begin{equation}\label{eqn:est_part_2_prox}
		\sqrt{\tfrac{\mu}{t(1+t\mu)}} \|\hat x -x\|\geq \|t^{-1}(\hat x - \prox_{t\Psi}(\hat x))\|\geq \|t^{-1}(x-\prox_{t\Psi}(\hat x))\|-\|t^{-1}(\hat x-x)\|.
	\end{equation}
	To simplify notation, set $\hat z:=\prox_{t\Psi}(\hat x)$.
	By definition of $\Psi$, equality $$\hat z=\argmin_{z}~\left\{ F_{t}(z;x)+\mu\|z-x\|^2\right\}\qquad \textrm{ holds},$$
	and therefore $2\mu(x-\hat z)\in \partial F_{t}(\hat z;x)$. Taking into account that $F_{t}(\cdot;x)$ is $t^{-1}$-strongly convex, we deduce
	$$\|2\mu(x-\hat z)\|\geq \dist\left(0; \partial F_{t}(\hat z;x)\right)\geq t^{-1} \|\hat z -{\bar x}\|\geq \|t^{-1}(x-{\bar x})\|-\|t^{-1}(x-\hat z)\|.$$
	Rearranging and combining the estimate with \eqref{eqn:der_prox}, \eqref{eqn:est_part_2_prox} yields the second inequality in	\eqref{eqn:comp_prox_lin_reg}.
\end{proof}

In the most important setting $t = 1/\mu$, Theorem~\ref{thm:prox_approx} reduces to the estimate
\begin{equation}\label{eqn:thm_speci}
	\tfrac{1}{4} \left\|\nabla F_{\frac{1}{2\mu}}(x)\right\|\leq \left\|\mathcal{G}_{1/\mu}(x)\right\|\leq \tfrac{3}{2}\left(1+\tfrac{1}{\sqrt{2}}\right)\left\|\nabla F_{\frac{1}{2\mu}}(x)\right\|.
\end{equation} 
A closely related result has recently appeared in \cite[Theorem 5.3]{prox_error}, with a different proof, and has been extended to a more general class of Taylor-like approximations in \cite{taylor}. Combining \eqref{eqn:thm_speci}  and \eqref{eqn:perturb_opt_orig} we deduce that for any point $x$, there exists a point $\hat x$ (namely $\hat x=\prox_{F/2\mu}(x))$) satisfying
\begin{equation}\label{eqn:perturb_opt}
	\left\{\begin{array}{cl}
		\|\hat{x}-x\|&\leq  \frac{2}{\mu} \|\mathcal{G}_{1/\mu}(x)\|,\\ 
		F(\hat x)&\leq F(x),\\
		\dist(0;\partial F(\hat{x}))&\leq 4 \|\mathcal{G}_{1/\mu}(x)\|.
	\end{array}\right. 
\end{equation}
Thus if $\|\mathcal{G}_{1/\mu}(x)\|$ is small, the point $x$ is ``near'' some point $\hat x$ that is ``nearly-stationary'' for $F$. Notice that $\hat x$ is not computable, since it requires evaluation of $\prox_{F/2\mu}$. Computing  $\hat x$ is not the point, however; the sole purpose of $\hat x$ is to certify that $x$ is approximately stationary in the sense of \eqref{eqn:perturb_opt}. 

\section{Inexact analysis of the prox-linear method}\label{sec:inex_prox_lin}
In practice, it is often impossible to solve the proximal subproblems $\min_z F_{t}(z;y)$ exactly. In this section, we explain the effect of inexactness in the proximal subproblems \eqref{eqn:simple_prox} on the overall performance of the prox-linear algorithm. By ``inexactness'', one can mean a variety of concepts.  Two most natural ones are that of
$(i)$ terminating the subproblems based on near-optimality in function value and $(ii)$ terminating based on ``near-stationarity''. 

Which of the two criteria is used depends on the algorithms that are available for solving the proximal subproblems. If primal-dual interior-point methods are applicable, then termination based on near-optimality in function value is most appropriate. 
When the subproblems themselves can only be solved by first-order methods, the situation is less clear. 
In particular, if near-optimality in function value is the goal, then one must use saddle-point methods.
Efficiency estimates of saddle-point algorithms, on the other hand, depend on the diameter of the feasible region, rather than on the quality of the initial iterate (e.g. distance of initial iterate to the optimal solution). Thus saddle-point methods cannot be directly warm-started, that is one cannot easily use iterates from previous prox-linear subproblems to speed up the algorithm for the current subproblem. Moreover, there is a conceptual incompatibility of the prox-linear method with termination based on functional near-optimality. Indeed, the prox-linear  method seeks to make the stationarity measure $\|\mathcal{G}_t(x)\|$ small, and so it seems more fitting that the proximal subproblems are solved based on near-stationarity themselves.
In this section, we consider both termination criteria. The arguments are quick modifications of the proof of Proposition~\ref{prop:basic_prox_lin}.


\subsection{Near-optimality in the subproblems}\label{sec:near-optim}
We first consider the effect of solving the proximal subproblems up to a tolerance on function values. 
Given a tolerance $\varepsilon > 0$, we say that a point $x$ is an \emph{$\varepsilon$-approximate minimizer} of a function $f\colon\R^d\to\overline\R$ whenever the inequality holds: 
\begin{align*}
	f(x) \le \inf \, f + \varepsilon.
\end{align*}
Consider now a sequence of tolerances $\varepsilon_k\geq 0$ for $k=1,2\ldots,\infty$. Then given a current iterate $x_k$, an \emph{inexact prox-linear algorithm} for minimizing $F$ can simply declare $x_{k+1}$ to be an $\varepsilon_{k+1}$-approximate minimizer of $F_t(\cdot;x_{k})$. We record this scheme in Algorithm~\ref{alg: inex_prox_lin}.

{\LinesNotNumbered
	\begin{algorithm}[h!]
		\SetKw{Null}{NULL}
		\SetKw{Return}{return}
		\Initialize{A point $x_0 \in \text{dom} \,
			g$, a real $t>0$, and a sequence $\{\varepsilon_i \}^{\infty}_{i=1}\subset [0,+\infty)$.}
		{\bf Step k:} ($k\geq 0)$
		Set $x_{k+1}$ to be an $\varepsilon_{k+1}\textrm{-approximate minimizer of } F_{t}(\cdot;x_{k})$. \\
		\caption{Inexact prox-linear method: near-optimality}
		\label{alg: inex_prox_lin}
\end{algorithm}}

Before stating convergence guarantees of the method, we record the following observation stating that the step-size of the inexact prox-linear method $\|x_{k+1}-x_k\|$ and the accuracy $\varepsilon_k$ jointly control the size of the true prox-gradient $\|\mathcal{G}_t(x_k)\|$. As a consequence, the step-sizes $\|x_{k+1}-x_k\|$ generated throughout the algorithm can be used as surrogates for the true stationarity measure $\|\mathcal{G}_t(x_k)\|$.  

\begin{lemma} Suppose $x^+$ is an $\varepsilon$-approximate minimizer of $F_t(\cdot; x)$. Then the inequality holds:
	\[ \norm{\mathcal{G}_t(x)}^2 \le 4t^{-1}
	\varepsilon + 2 \norm{t^{-1}(x^+-x)}^2.\]
	\label{lem: funct_grad}
\end{lemma}
\begin{proof} 
	Let $z^*$ be the true minimizer of $F_t(\cdot; x)$. We successively deduce
	\begin{align}
		\norm{\mathcal{G}_t(x)}^2 &\le \frac{4}{t} \cdot \frac{1}{2t}
		\norm{x^+-z^*}^2 + 2 \norm{t^{-1}(x^+-x)}^2 \nonumber \\
		&\le \frac{4}{t} \cdot \big ( F_t(x^+; x) - F_t(z^*;x)
		\big ) + 2 \norm{t^{-1}(x^+-x)}^2 \\
		&\le \frac{4}{t} \cdot \varepsilon + 2 \norm{t^{-1}(x^+-x)}^2 \nonumber,
	\end{align}
	where the first inequality follows from the triangle inequality and the estimate $(a+b)^2\leq 2(a^2+b^2)$ for any reals $a,b$, and the second inequality is an immediate consequence of strong convexity of the function
	$F_t(\cdot;x)$.
\end{proof}

The inexact prox-linear algorithm comes equipped with the following guarantee. 
\begin{theorem}[Convergence of the inexact prox-linear algorithm: near-optimality] {\hfill \\ }Supposing $t\leq \mu^{-1}$, the iterates generated by  Algorithm~\ref{alg: inex_prox_lin} satisfy 
	\[ \min_{j=0, \hdots, N-1}  \norm{\mathcal{G}_t(x_j)}^2  \le
	\frac{2t^{-1} \big ( F(x_0)-F^* +  \sum_{j=1}^{N} \varepsilon_{j} \big
		)}{N},\]
	where we set $\displaystyle F^*:=\lf_{k\to\infty} F(x_k)$.
	\label{thm: funct_pla}
\end{theorem}
\begin{proof}
	Let $x_{k}^*$ be the exact minimizer of 
	$F_t(\cdot;x_k)$. Note then the equality $\mathcal{G}_t(x_k) = t^{-1}(x_{k}^*-x_k)$. Taking into account that  $F_t(\cdot;x_k)$ is strongly convex with modulus $1/t$, we deduce
	\begin{align*}
		F(x_k) = F_t(x_k; x_k) &\ge  F_t(x_{k}^*; x_k) + \tfrac{t}{2}
		\norm{\mathcal{G}_t(x_k)}^2 \ge  F_t(x_{k+1}; x_k) - \varepsilon_{k+1} + \tfrac{t}{2}
		\norm{\mathcal{G}_t(x_k)}^2.
	\end{align*}
	Then the inequality $t\leq \mu^{-1}$ along with \eqref{eqn:ineq_min_max} implies that $F_t(\cdot;x_k)$ is an upper model of $F(\cdot)$ and therefore  
	\begin{equation} F(x_k) \ge F(x_{k+1}) - \varepsilon_{k+1} + \tfrac{t}{2}
		\norm{\mathcal{G}_t(x_k)}^2. \label{eq:unbound_fucnt_error_non_cvx}\end{equation}
	We conclude 
	\begin{align*}
		\min_{j=0, \hdots, N-1} \norm{\mathcal{G}_t(x_j)}^2  \le \frac{1}{N}
		\sum_{j=0}^{N-1}
		\norm{\mathcal{G}_t(x_j)}^2
		&\le
		\frac{2t^{-1} \left ( \sum_{j=0}^{N-1} F(x_j)-F(x_{j+1}) + \sum_{j=0}^{N-1}
			\varepsilon_{j+1} \right )}{N} \\
		&\le \frac{2t^{-1} \big ( F(x_0)-F^* + \sum_{j=0}^{N-1} \varepsilon_{j+1} \big )}{N}.
	\end{align*}
	The proof is complete.
\end{proof}

Thus in order to maintain the rate afforded by the exact prox-linear method, it suffices for the errors $\{\varepsilon_k\}^{\infty}_{k=1}$ to be summable; e.g. set $\varepsilon_k\sim \frac{1}{k^{1+q}}$ with $q>0$.

\subsection{Near-stationarity in the subproblems}\label{sec:near-stat}
In the previous section, we considered the effect of solving the proximal subproblems up to an accuracy in functional error. We now consider instead a model of inexactness for the proximal subproblems based on near-stationarity. 
A first naive attempt would be to consider a point $z$ to be $\varepsilon$-stationary for the proximal subproblem, $\min F_t(\cdot;x)$, if it satisfies
$$\dist(0;\partial_z F_t(z;x))\leq \varepsilon.$$
This assumption, however, is not reasonable since first-order methods for this problem do not produce such points $z$, unless $h$ is smooth. 
Instead, let us look at the Fenchel dual problem. To simplify notation, write the target  subproblem $\min F_t(\cdot;x)$ as 
\begin{equation}\label{eqn:target_rewritten}
	\min_z ~ h(b-Az) + G(z) 
\end{equation}
under the identification $G(z)=g(z)+\frac{1}{2t}\|z-x\|^2$, $A=-\nabla c(x)$, and $b=c(x)-\nabla c(x)x$.  Notice that $G$ is $t^{-1}$-strongly convex and therefore $G^\star$ is $C^1$-smooth with $t$-Lipschitz gradient. The Fenchel dual problem, after negation, takes the form \cite[Example 11.41]{RW98}:
\begin{equation}\label{eqn:fench_dual}
	\min_{w}~ \varphi(w):=G^{\star}(A^* w) - \ip{b, w}+h^{\star}(w) .
\end{equation}
Thus the dual objective function $\varphi$ is a sum of a smooth convex function $G^{\star}(A^* w) - \ip{b, w}$ and the simple nonsmooth convex term $h^{\star}$. 
Later on, when $x$ depends on an iteration counter $k$, we will use the notation $\varphi_k$, $G_k$, $A_k$, $b_k$ instead to make precise that these objects depend on $k$.

Typical first-order methods, such as prox-gradient and its accelerated variants can generate a point $w$ for the problem \eqref{eqn:fench_dual}  satisfying 
\begin{equation}\label{eqn:dual_nearstat}
	\dist(0;\partial \varphi(w))\leq \varepsilon
\end{equation}
up to any specified tolerance $\varepsilon>0$. Such schemes in each iteration only require evaluation of the gradient of the smooth function $G^{\star}(A^* w) - \ip{b, w}$ along with knowledge of a Lipschitz constant of the gradient, and evaluation of the proximal map of $h^{\star}$. For ease of reference, we record  these quantities here in terms of the original functional components of the composite problem \eqref{eqn:comp2}. Since the proof is standard, we have placed it in Appendix~\ref{sec:proofs_append}.

\begin{lemma}\label{lem:expl_grad_comp}
	The following are true for all points $z$ and $w$ and real $t>0$:
	\begin{itemize}	
		\item The equation holds:
		\begin{equation}\label{dual_gradient_formula}
			\prox_{th^{\star}}(w)=t\left(w/t-\prox_{h/t}(w/t)\right).
		\end{equation}
		
		\item 
		The equations hold:
		\begin{equation}\label{eqn:dual_subdiff}
			G^{\star}(z)=(g^{\star})_{1/t}(z+x/t)-\tfrac{1}{2t}\|x\|^2\qquad \textrm{and}\qquad \nabla G^{\star}(z)=\prox_{tg}(x+tz).
		\end{equation}
		Consequently, the gradient map $\nabla \Big(G^{\star}\circ A^* - \ip{\cdot,b}\Big)$ is Lipschitz continuous with constant $t\|\nabla c(x)\|^2_{\textrm{op}}$ and admits the representation:
		\begin{equation}\label{eqn:grad_express}
			\nabla \Big(G^{\star}\circ A^* - \ip{b,\cdot}\Big)(w)=\nabla c(x)\Big(x+\prox_{tg}(x-t\nabla c(x)^*w)\Big)-c(x).
		\end{equation}
	\end{itemize}
\end{lemma}

Thus, suppose we have found
a point $w$ satisfying \eqref{eqn:dual_nearstat}. How can we then generate a primal iterate $x^+$ at which to form the prox-linear subproblem for the next step? The following lemma provides a simple recipe for doing exactly that. It shows how to generate from $w$ a point that is a true minimizer to a slight perturbation of the proximal subproblem.

\begin{lemma}[Primal recovery from dual $\varepsilon$-stationarity]\label{lem:prim_revorery}
	Let $\varphi$ be the function defined in \eqref{eqn:fench_dual}. Fix a point $w\in \dom \varphi$ and a vector $\zeta\in \partial \varphi(w)$. Then the point $\bar x:=\nabla G^{\star}(A^*w)$ is the true minimizer of the problem
	\begin{equation}\label{eqn:perturbed_prob_stat}
		\min_z~ h( \zeta + b - Az ) + G(z).
	\end{equation}
\end{lemma}
\begin{proof}
	Appealing to the chain rule,
	$\partial \varphi(w)=A\nabla G^{\star} (A^*w)-b+\partial h^{\star}(w)$, we deduce  $$\zeta+b\in A\nabla G^{\star} (A^*w)+\partial h^{\star}(w)=A\bar x+\partial h^{\star}(w).$$	The relation \eqref{eqn:reverse_subgrad_conj} then implies  $w\in \partial h(\zeta+b-A\bar x)$. Applying $A^*$ to both sides and rearranging yields
	$$0\in -A^*\partial h(\zeta+b-A\bar x)+A^{*}w\subseteq -A^*\partial h(\zeta+b-A \bar x)+\partial G(\bar x),$$
	where the last inclusion follows from applying \eqref{eqn:reverse_subgrad_conj}  to $G$.
	The right-hand-side is exactly the subdifferential of the objective function in \eqref{eqn:perturbed_prob_stat} evaluated at $\bar x$. The result follows.
\end{proof}

This lemma directly motivates the following inexact extension of the prox-linear algorithm (Algorithm~\ref{alg: inex_prox_lin_subgrad}), based on dual near-stationary points.

{\LinesNotNumbered
	\begin{algorithm}[h!]
		\SetKw{Null}{NULL}
		\SetKw{Return}{return}
		\Initialize{A point $x_0 \in \text{dom} \,
			g$, a real $t>0$, and a sequence $\{\varepsilon_i\}^{\infty}_{i=1}\subset [0,+\infty)$.}
		{\bf Step k:} ($k\geq 0)$
		Find $(x_{k+1},\zeta_{k+1})$ such that $\|\zeta_{k+1}\|\leq \varepsilon_{k+1}$ and 
		$x_{k+1}$ is the minimizer of the function 
		\begin{equation}\label{eqn:perturbed_func}
			z\mapsto g(z)+h\Big(\zeta_{k+1}+c(x_k)+\nabla c(x_k)(z -x_k)\Big)+\frac{1}{2t}\|z-x_k\|^2.
		\end{equation}
		\caption{Inexact prox-linear method:  near-stationarity}
		\label{alg: inex_prox_lin_subgrad}
\end{algorithm}}

Algorithm~\ref{alg: inex_prox_lin_subgrad} is stated in a way most useful for convergence analysis. On the other hand, it is not very explicit. To crystallize the ideas, let us concretely describe how one can implement step $k$ of the scheme. First, we find a point $w_{k+1}$ that is $\varepsilon_{k+1}$-stationary for the dual problem \eqref{eqn:fench_dual}. More precisely, we  find a pair $(w_{k+1},\zeta_{k+1})$ satisfying $\zeta_{k+1}\in \partial \varphi_{k}(w_{k+1})$ and $\|\zeta_{k+1}\|\leq \varepsilon_{k+1}$. We can achieve this by a proximal gradient method (or its accelerated variants) on the dual problem \eqref{eqn:fench_dual}. Then combining Lemma~\ref{lem:prim_revorery} with equation \eqref{eqn:dual_subdiff}, we conclude that we can simply set 
$$x_{k+1}:=\nabla G^{\star}(A^*w_{k+1})=\prox_{tg}(x_k-t\nabla c(x_k)^*w_{k+1}).$$
We record this more explicit description of Algorithm~\ref{alg: inex_prox_lin_subgrad}  in Algorithm~\ref{alg: inex_prox_lin_subgrad_exp}. The reader should keep in mind that even though Algorithm~\ref{alg: inex_prox_lin_subgrad_exp} is more explicit, the convergence analysis we present will use the description in Algorithm~\ref{alg: inex_prox_lin_subgrad}.

{\LinesNotNumbered
	\begin{algorithm}[h!]
		\SetKw{Null}{NULL}
		\SetKw{Return}{return}
		\Initialize{A point $x_0 \in \text{dom} \,
			g$, a real $t>0$, and a sequence $\{\varepsilon_i\}^{\infty}_{i=1}\subset [0,+\infty)$.}
		{\bf Step k:} ($k\geq 0)$
		Define the function

		$$\varphi_k(w):=(g^{\star})_{1/t}\Big(x_k/t-\nabla c(x_k)^* w\Big) - \big\langle c(x_k)-\nabla c(x_k)x_k, w\big\rangle+h^{\star}(w).$$
		
		
		Find a point $w_{k+1}$ satisfying $\dist(0;\partial \varphi_k(w_{k+1}))\leq \varepsilon_{k+1}$.\\
		Set $x_{k+1}=\prox_{tg}(x_k-t\nabla c(x_k)^*w_{k+1})$.
		\caption{Inexact prox-linear method:  near-stationarity (explicit)}
		\label{alg: inex_prox_lin_subgrad_exp}
\end{algorithm}}

Before stating convergence guarantees of the method, we record the following observation stating that the step-size $\|x_{k+1}-x_k\|$ and the error $\varepsilon_{k+1}$ jointly control  the stationarity measure $\|\mathcal{G}_t(x_k)\|$. In other words, one can use the step-size $\|x_{k+1}-x_k\|$, generated throughout the algorithm, as a surrogate for the true 
stationarity measure $\|\mathcal{G}_t(x_k)\|$.

\begin{lemma}\label{lem:pass_tostat} Suppose $x^+$ is a minimizer of the function 
	\begin{equation*}
		z\mapsto g(z)+h\Big(\zeta+c(x)+\nabla c(x)(z -x)\Big)+\frac{1}{2t}\|z-x\|^2
	\end{equation*}
	for some vector $\zeta$.
	Then for any real $t > 0$, the inequality holds:
	\begin{equation}\label{eqn:keyinlema-eqn}
		\norm{\mathcal{G}_t(x)}^2  \le
		8 Lt^{-1} \cdot \|\zeta\| + 2 \norm{t^{-1}(x^+-x)}^2.
	\end{equation}
	\label{lem: dual_grad}
\end{lemma}
\begin{proof} 
	Define the function $$l(z)=g(z)+h\Big(\zeta+c(x)+\nabla c(x)(z -x)\Big)+\frac{1}{2t}\|z-x\|^2.$$
	Let $z^*$ be the true minimizer of $F_t(\cdot;x)$.
	We successively deduce
	\begin{align}
		\norm{\mathcal{G}_t(x)}^2 &\le \frac{4}{t} \cdot \frac{1}{2t}
		\norm{x^+-z^*}^2 + 2 \norm{t^{-1}(x^+-x)}^2 \nonumber \\
		&\le \frac{4}{t} \cdot \big ( F_t(x^+; x) - F_t(z^*;x)
		\big ) + 2 \norm{t^{-1}(x^+-x)}^2\\
		&\leq  \frac{4}{t}(l(x^+)-l(z^*)+2L\|\zeta\|) + 2 \norm{t^{-1}(x^+-x)}^2  \nonumber\\
		&\leq  8t^{-1}L\|\zeta\| + 2 \norm{t^{-1}(x^+-x)}^2 , \nonumber
	\end{align}
	where the first inequality follows from the triangle inequality and the estimate $(a+b)^2\leq 2(a^2+b^2)$ for any reals $a,b$, the second inequality is an immediate consequence of strong convexity of the function
	$F_t(\cdot;x)$, and the third follows from Lipschitz continuity of $h$.
	%
\end{proof}

Theorem~\ref{thm: dual_pla} explains the convergence guarantees of the method; c.f. Proposition~\ref{prop:basic_prox_lin}. 

\begin{theorem}[Convergence of the inexact prox-linear method: near-stationarity] Supposing $t\leq \mu^{-1}$, the iterates generated by  Algorithm~\ref{alg: inex_prox_lin_subgrad} satisfy  
	\[ \min_{j =0, \hdots, N-1} \,  \norm{\mathcal{G}_t(x_j)}^2  \le
	\frac{4 t^{-1} \big ( F(x_0)-F^* + 4L \cdot \sum_{j=1}^{N}
		\varepsilon_{j} \big )}{N},\]
	where we set $\displaystyle F^*:=\lf_{k\to\infty} F(x_k)$.
	\label{thm: dual_pla}
\end{theorem}
\begin{proof}
	Observe the inequalities:
	\begin{align*}
		F(x_{k+1}) &\le F_{t}(x_{k+1};x_k)\\
		&\leq h \big (\zeta_{k+1} + c(x_k) +
		\nabla c(x_k) (x_{k+1}-x_k) \big ) + g(x_{k+1}) + \tfrac{1}{2t}
		\norm{x_{k+1}-x_k}^2 + L \cdot \varepsilon_{k+1}.
	\end{align*}
	Since the point $x_{k+1}$ minimizes the $\frac{1}{t}$-strongly convex function in
	\eqref{eqn:perturbed_func}, we deduce
	\begin{equation}
		\begin{aligned} 
			F(x_{k+1}) &\le h \big (\zeta_{k+1} + c(x_k) \big ) + g(x_k) + L \cdot \varepsilon_{k+1} - \tfrac{1}{2t}
			\norm{x_{k+1}-x_k}^2\\
			&\le F(x_k) + 2L \cdot \varepsilon_{k+1} - \tfrac{1}{2t} \norm{x_{k+1}-x_k}^2. 
		\end{aligned} \label{eq:something_unbounded}
	\end{equation}
	Summing along the indices $j = 0, \hdots, N-1$ yields
	$$\sum^{N-1}_{j=0} \|t^{-1}(x_{j+1}-x_j)\|^2\leq \frac{2}{t}\left(F(x_0)-F^*+2L\sum^{N-1}_{j=0} \varepsilon_{j+1}\right).$$
	Taking into account Lemma~\ref{lem:pass_tostat}, we deduce
	\begin{equation} \displaystyle\min_{j=0,1,\ldots,N-1}
		\|\mathcal{G}_t(x_j)\|^2\leq \frac{1}{N} \sum^{N-1}_{j=0}
		\|\mathcal{G}_t(x_j)\|^2\leq
		\frac{4t^{-1}(F(x_0)-F^*+4L\sum^{N}_{j=1}
			\varepsilon_{j})}{N}, \label{eq:something_unbounded_1} \end{equation}
	as claimed.
\end{proof}

In particular, to maintain the same rate in $N$ as the exact prox-linear method in Proposition~\ref{prop:basic_prox_lin}, we must be sure that the sequence $\varepsilon_k$ is summable. Hence, we can set $\varepsilon_k\sim \frac{1}{k^{1+q}}$ for any $q>0$.

\section{Overall complexity for the composite problem class}\label{sec:overall_comp}
In light of the results of Section~\ref{sec:inex_prox_lin}, we can now use the inexact prox-linear method to derive efficiency estimates for the composite problem class \eqref{eqn:comp2}, where the proximal subproblems are themselves solved by first-order methods. As is standard, we will assume that the functions $h$ and $g$ are {\em prox-friendly}, meaning that $\prox_{th}$ and $\prox_{tg}$ can be evaluated. Given a target accuracy $\varepsilon>0$, we aim to determine the number of {\em basic operations} --  evaluations of $c(x)$, matrix-vector multiplications $\nabla c(x)v$ and $\nabla c(x)^*w$, and evaluations of $\prox_{th}$, $\prox_{tg}$ -- needed to find a point $x$ satisfying $\|{\mathcal G}_t(x)\|\leq \varepsilon$. To simplify the exposition, we will ignore the cost of the evaluation $c(x)$, as it will typically be dominated by the cost of the matrix-vector products  $\nabla c(x)v$ and $\nabla c(x)^*w$.

To make progress, in this section we also assume that we have available a  real value, denoted $\|\nabla c\|$, satisfying
$$ \|\nabla c\|\geq\sup_{x\in \dom g} \|\nabla c(x)\|_{\textrm{op}}.$$
In particular, we assume that the right-hand-side is finite. Strictly speaking, we only need the inequality $\|\nabla c\|\geq \|\nabla c(x_k)\|_{\textrm{op}}$ to hold along an iterate sequence $x_k$ generated by the inexact prox-linear method. This assumption is completely expected: even when $c$ is a linear map, convergence rates of first-order methods for the composite problem \eqref{eqn:comp2} depend on some norm of the Jacobian $\nabla c$.

The strategy we propose can be succinctly summarized as follows: 
\begin{itemize}
	\item (Smoothing+prox-linear+fast-gradient) We will replace $h$ by a smooth approximation (Moreau envelope), with a careful choice of the smoothing parameter. Then we will apply an inexact prox-linear method to the smoothed problem, with the proximal subproblems approximately solved by fast-gradient methods.
\end{itemize}


The basis for the ensuing analysis  is the fast-gradient method of Nesterov \cite{nest_conv_comp} for minimizing convex additive composite problems.
The following section recalls the scheme and records its efficiency guarantees, for ease of reference.



\subsection{Interlude: fast gradient method for additive convex composite problems}\label{sec:interlude_fast}
This section discusses a scheme from \cite{nest_conv_comp} that can be applied to any problem of the form
\begin{equation}\label{eqn:conv_comp_gen_accel}
	\min_x~ f^p(x):=f(x)+p(x),
\end{equation}
where $f\colon\R^d \to\R$ is a convex $C^1$-smooth function with $L_f$-Lipschitz gradient and $p\colon\R^d \to\overline\R$ is a closed $\alpha$-strongly convex function ($\alpha\geq 0$). The setting $\alpha =0$ signifies that $p$ is just convex.

We record in Algorithm~\ref{alg:fast_grad_gen}  the so-called ``fast-gradient method'' for such problems \cite[Accelerated Method]{nest_conv_comp}.

{\LinesNotNumbered
	\begin{algorithm}[h!]
		\SetKw{Null}{NULL}
		\SetKw{Return}{return}
		\Initialize{Fix a point $x_0 \in \dom p$, set $\theta_0=0$, define the function $\psi_0(x):=\frac{1}{2}\|x-x_0\|^2$.}
		{\bf Step j:} ($j\geq 0)$ 
		Find $a_{j+1}>0$ from the equation 
		$$\tfrac{a^2_{j+1}}{\theta_j+a_{j+1}}=2\tfrac{1+\alpha \theta_j}{L_f}.$$
		Compute the following:
		\begin{align}
			\theta_{j+1}&=\theta_{j}+a_{j+1},\notag\\\	
			v_j&=\argmin_x ~\psi_j(x),	\label{eqn:v_fast_grad_gen}\\
			y_j&=\tfrac{\theta_jx_j+a_{j+1}v_j}{\theta_{j+1}},\notag\\
			x_{j+1}&=\argmin_{x}~ \{f(y_j)+\langle \nabla f(y_j),x-y_j\rangle+\tfrac{L_f}{2}\|x-y_j\|^2+p(x) \}\label{eqn:psi_fast_grad_gen}.
		\end{align}
		Define the function 
		\begin{equation}\label{eqn:fina_psi_def}
			\psi_{j+1}(x)=\psi_j(x)+a_{j+1}[f(x_{j+1})+\langle \nabla f(x_{j+1}),x-x_{j+1}\rangle+p(x)].
		\end{equation}


		%
		\caption{Fast gradient method of Nesterov \cite[Accelerated Method]{nest_conv_comp}}
		\label{alg:fast_grad_gen}
\end{algorithm}}

The method comes equipped with the following guarantee \cite[Theorem 6]{nest_conv_comp}.
\begin{theorem}\label{eff:accle_meth}
	Let $x^*$ be a minimizer of $f^p$ and suppose $\alpha>0$. Then the iterates $x_j$ generated by  Algorithm~\ref{alg:fast_grad_gen} satisfy:
	$$f^{p}(x_{j})-f^p(x^*)\leq   \left(1+\sqrt{\frac{\alpha}{2L_f}}\right)^{-2(j-1)}\frac{L_f}{4}\|x^*-x_0\|^2.$$ 
	
\end{theorem}

Let us now make a few observations that we will call on shortly. First, each iteration of Algorithm~\ref{alg:fast_grad_gen} only requires two gradient computations, $\nabla f(y_j)$ in \eqref{eqn:psi_fast_grad_gen} and $\nabla f(x_{j+1})$ in \eqref{eqn:fina_psi_def}, and two proximal operations, $\prox_{p/L_f}$ in \eqref{eqn:psi_fast_grad_gen}  and $\prox_{p}$ in \eqref{eqn:v_fast_grad_gen}. 

Secondly, let us translate the estimates in Theorem~\ref{eff:accle_meth}  to estimates based on desired accuracy. Namely, simple arithmetic shows 
that the inequality 
$$f^{p}(x_j)-f^p(x^*)\leq \varepsilon$$ holds as soon as the number of iterations $j$ satisfies
\begin{equation}\label{eqn:main_conv_estimate}
	j\geq  1+\sqrt{\frac{L_f}{2\alpha}}\cdot\log\left(\frac{L_f\|x^*-x_0\|^2}{4\varepsilon}\right).
\end{equation}
Let us now see how we can modify the scheme slightly so that it can find points $x$ with small subgradients. Given a point $x$ consider a single prox-gradient iteration $\hat x:=\prox_{\frac{p}{L_f}}\left(x-\tfrac{1}{L_f}\nabla f(x)\right)$. Then we successively deduce
$$\dist^2(0;\partial f^p(\hat x))\leq 4\|L_f(\hat x-x)\|^2\leq 8L_f(f^p(x)-f^p(\hat x))\leq 8L_f(f^p(x)-f^p(x^*))$$
where the first inequality is \eqref{eqn:stat_est} and the second is the descent guarantee of the prox-gradient method (e.g. \cite[Theorem 1]{nest_conv_comp}). Thus  the inequality $f^p(x)-f^p(x^*)\leq \varepsilon^2/(8L_f)$ would immediately imply $\dist(0;\partial f^p(\hat x))\leq \varepsilon$.
Therefore, let us add an extra prox-gradient step $\hat x_j:=\prox_{\frac{p}{L_f}}\left(x_j-\tfrac{1}{L_f}\nabla f(x_j)\right)$ to each iteration of Algorithm~\ref{alg:fast_grad_gen}. Appealing to the linear rate in \eqref{eqn:main_conv_estimate}, we then deduce that 
we can be sure 
of the inequality 
$$\dist(0;\partial f^p(\hat x_j))\leq \varepsilon$$  as soon as the number of iterations $j$ satisfies
\begin{equation}\label{eqn:subgrad_eff_est_fast}
	j\geq 1+\sqrt{\frac{L_f}{2\alpha}}\cdot\log\left(\frac{2L^2_f\|x^*-x_0\|^2}{\varepsilon^2}\right).
\end{equation}
With this modification, each iteration of the scheme requires two gradient evaluations of $f$ and three proximal operations of $p$.

\subsection{Total cost if $h$ is smooth}\label{sec:total_cost_smooth}
In this section, we will assume that $h$ is already $C^1$-smooth with the gradient having Lipschitz constant $L_h$, and calculate the overall cost of the inexact prox-linear method that wraps a linearly convergent method for the proximal subproblems. As we have discussed in Section~\ref{sec:inex_prox_lin}, the proximal subproblems can either be approximately solved by primal methods or by dual methods. The dual methods are better adapted for a global analysis, since the dual problem has a bounded domain; therefore let us look first at that setting.

\begin{remark}[Asymptotic notation]\label{rem:assympt_not}
	{\rm
		To make clear dependence on the problem's data, we will sometimes use asymptotic notation \cite[Section 3.5]{alg_o_not}. For two functions $\psi$ and $\Psi$ of a vector $\omega\in \R^{\ell}$, the symbol $\psi(\omega)=\mathcal{O}(\Psi(\omega))$ will mean that there exist constants $K,C>0$ such that the inequality, 
		$|\psi(\omega)|\leq C \cdot|\Psi(\omega)| $, holds for all $\omega$ satisfying $\omega_i\geq K$ for all $i=1,\ldots,\ell$. When using asymptotic notation in this section, we will use the vector $\omega$ to encode the data of the problem $\omega=(\|\nabla c\|,L_h, L,\beta,F(x_0)-F^*,1/\varepsilon)$. In the setting that $h$ is not differentiable, $L_h$ will be omitted from $\omega$.}
\end{remark}

\subsubsection*{Total cost based on dual near-stationarity in the subproblems}
We consider the near-stationarity model of inexactness  as in  Section~\ref{sec:near-stat}. Namely, let us compute the total cost of 
Algorithm~\ref{alg: inex_prox_lin_subgrad_exp}, when each subproblem $\min_w \varphi_k(w)$ is approximately minimized by the fast-gradient method (Algorithm~\ref{alg:fast_grad_gen}). In the notation of Section~\ref{sec:interlude_fast}, we set 
$f(w)= G_k^{\star}(A^*_k w)-\langle b_k,w\rangle$ and $p=h^{\star}$.
By Lemma~\ref{lem:expl_grad_comp},  the function $f$ is $C^1$-smooth with gradient having Lipschitz constant $L_f:=t\|\nabla c(x_k)\|^2_{\textrm{op}}$. Since $\nabla h$ is assumed to be $L_h$-Lipschitz, we deduce that $h^{\star}$ is 
$\frac{1}{L_{h}}$-strongly convex. Notice moreover that since $h$ is $L$-Lipschitz, any point in $\dom h^{\star}$ is bounded in norm by $L$; hence the diameter of $\dom h^{\star}$ is at most $2L$. Let us now apply
Algorithm~\ref{alg:fast_grad_gen} (with the extra prox-gradient step) to the problem $\min_{w} \varphi_k(w)=f(w)+p(w)$. According to the estimate  
\eqref{eqn:subgrad_eff_est_fast}, we will be able to find the desired point $w_{k+1}$ satisfying $\dist(0;\partial \varphi_k(w_{k+1})))\leq \varepsilon_{k+1}$ after at most 
\begin{equation}\label{eqn:specialize_fast_grad_to_conv}
	1+\ceil*{\sqrt{\frac{t\|\nabla c(x_k)\|^2_{\textrm{op}}L_h}{2}}\cdot\log\left(\frac{8t^2\|\nabla c(x_k)\|_{\textrm{op}}^4L^2}{\varepsilon^2_{k+1}}\right)}.
\end{equation}
iterations of the fast-gradient method. According to Lemma~\ref{lem:expl_grad_comp}, each gradient evaluation $\nabla f$ requires two-matrix vector multiplications and one proximal operation of $g$, while the proximal operation of $p$ amounts to a single proximal operation of $h$. Thus each iteration of  Algorithm~\ref{alg:fast_grad_gen}, with the extra prox-gradient step requires $9$ basic operations. Finally to complete step $k$ of Algorithm~\ref{alg:fast_grad_gen}, we must take one extra proximal map of $g$. Hence the number of basic operations needed to complete step $k$ of Algorithm~\ref{alg:fast_grad_gen} is $9\times$(equation \eqref{eqn:specialize_fast_grad_to_conv})$+1$, where we set  $t=1/\mu$.

%

Let us now compute the total cost across the outer iterations $k$. Theorem~\ref{thm: dual_pla} shows that if we set $\varepsilon_k=\frac{1}{L k^2}$ in each iteration $k$ of Algorithm \ref{alg: inex_prox_lin_subgrad_exp},
then after $N$ outer iterations we are guaranteed 
\begin{equation}\label{eqn:basic_inexact_eff_est}
	\min_{j =0, \hdots, N-1} \,  \norm{\mathcal{G}_{\frac{1}{\mu}}(x_j)}^2  \le
	\frac{4 \mu \big ( F(x_0)-F^* + 8  \big )}{N}.
\end{equation}
Thus we can find a point $x$ satisfying $$\norm{\mathcal{G}_{\frac{1}{\mu}}(x)}\leq \varepsilon$$ after at most $\mathcal{N}(\varepsilon):=\ceil*{\frac{4\mu(F(x_0)-F^*+8)}{\varepsilon^2}}$ outer-iterations and therefore after
\begin{equation}\label{eqn:dual_stat_exact}
	\ceil*{\frac{4\mu(F(x_0)-F^*+8)}{\varepsilon^2}}\left(10+9\ceil*{\sqrt{\frac{\|\nabla c\|^2L_h}{2\mu}}\cdot\log\left(\frac{8\|\nabla c\|^4L^2(1+\mathcal{N}(\varepsilon))^4}{\beta^2}\right)}\right)
\end{equation}
basic operations in total. Thus the number of basic operations is on the order of
\begin{equation}\label{eq:first_box}
	\boxed{\mathcal{O}\left({\frac{\sqrt{{\|\nabla c\|^2\cdot L_h\cdot \mu}}\cdot(F(x_0)-F^*)}{\varepsilon^2}}{~\log\left(\frac{\|\nabla c\|^2L^3\beta(F(x_0)-F^*)^2)}{\varepsilon^4}\right)}\right)}.
\end{equation}

%
%

%
%

\subsubsection*{Total cost based on approximate minimizers of the subproblems}
Let us look at what goes wrong with applying Algorithm~\ref{alg: inex_prox_lin}, with the proximal subproblems $\min_z\, F_t(z;x)$ approximately solved by a primal only method. 
To this end, notice that the objective function $F_t(\cdot;x)$ is a sum of the $\frac{1}{t}$-strongly convex and prox-friendly term $g+\frac{1}{2t}\|\cdot-x\|^2$ and the smooth convex function $z\mapsto h(c(x)+\nabla c(x)(z-x))$. The gradient of the smooth term 
is Lipschitz continuous with constant $\|\nabla c(x)\|^2_{\textrm{op}}L_h.$ Let us apply the fast gradient method (Algorithm~\ref{alg:fast_grad_gen}) to the proximal subproblem directly. According to the estimate \eqref{eqn:main_conv_estimate}, Algorithm~\ref{alg:fast_grad_gen} will find an $\varepsilon$-approximate minimizer $z$ of $F_t(\cdot;x)$ after at most 
\begin{equation}\label{eqn:cost_subsolve}
	1+\sqrt{\frac{t\|\nabla c(x)\|^2_{\textrm{op}}L_h}{2}}\cdot\log\left(\frac{\|\nabla c(x)\|^2_{\textrm{op}}L_h \|x^*-z_0\|^2}{4\varepsilon}\right)
\end{equation}
iterations, where $x^*$ is the minimizer of $F_t(\cdot;x)$ and the scheme is initialized at $z_0$. 
%
The difficulty is that there appears to be no simple way to bound the distance $\|z_0-z^*\|$ for each proximal subproblem, unless we assume that $\dom g$ is bounded. We next show how we can correct for this difficulty by more carefully coupling the inexact prox-linear algorithm and the linearly convergent algorithm for solving the subproblem. In particular, in each outer iteration of the proposed scheme (Algorithm~\ref{alg: inex_prox_lin_primal}), one runs a linearly convergent subroutine $\mathcal{M}$ on the prox-linear subproblem for a fixed number of iterations; this fixed number of inner iterations depends explicitly on $\mathcal{M}$'s linear rate of convergence. The algorithmic idea behind this coupling originates in \cite{quickening}. The most interesting consequence of this scheme is on so-called finite-sum problems, which we will discuss in Section~\ref{sec:fin_sum_prob}. In this context, the algorithms that one runs on the proximal subproblems are stochastic. Consequently, we adapt our analysis to a stochastic setting as well, proving  convergence  rates on  the expected norm of the prox-gradient $\|\mathcal{G}_t(x_k)\|$. 
When the proximal subproblems are approximately solved by deterministic methods, the convergence rates are all deterministic as well. 

The following definition makes precise the types of   algorithms that we will be able to accommodate as subroutines for the prox-linear subproblems.

\begin{defn}[Linearly convergent subscheme]\label{defn:lin_conv_subscheme}
	{\rm	A method $\mathcal{M}$ is {\em a linearly convergent subscheme} for the composite problem \eqref{eqn:comp2} if the following holds. For any points $x\in \R^d$, there exist constants $\gamma\geq 0$ and $\tau\in (0,1)$ so that when $\mathcal{M}$ is applied to $\min F_t(\cdot;x)$ with an arbitrary $z_0\in\dom g$ as an initial iterate, $\mathcal{M}$ generates a sequence $\{z_i\}_{i=1}^{\infty}$ satisfying 
		\begin{align}\label{eqn:main_lin_con_ineq}
			\mathbb{E}[F_t(z_i;x) - F_t(x^*;x)] \le \gamma \left ( 1- \tau \right )^i
			\|z_{0}-x^*\|^2 \qquad \textrm{for } i=1,\ldots,\infty,
		\end{align}
		where $x^*$ is the minimizer of $F_t(\cdot;x)$. 
	}
\end{defn}

We will be applying a linearly convergent subscheme to proximal subproblems $\min F_t(\cdot; x_k)$, where $x_k$ is generated in the previous iteration of an inexact prox-linear method. We will then denote the resulting constants $(\gamma,\tau)$ in the guarantee \eqref{eqn:main_lin_con_ineq} by $(\gamma_k,\tau_k)$. 

The overall method we propose is Algorithm~\ref{alg: inex_prox_lin_primal}. It is important to note that in order to implement this method, one must know explicitly the constants $(\gamma,\tau)$ for the method $\mathcal{M}$ on each proximal subproblem.

{\LinesNotNumbered
	\begin{algorithm}[h!]
		\SetKw{Null}{NULL}
		\SetKw{Return}{return}
		\Initialize{A point $x_0 \in \dom g$, real $t>0$,  a linearly convergent subscheme $\mathcal{M}$ for \eqref{eqn:comp2}. 
		}
		{\bf Step k:} ($k\geq 1)$  \\
		Set $x_{k,0}:=x_k$. Initialize $\mathcal{M}$
		on the problem $\min_z F_t(z;x_k)$ at $x_{k,0}$, and run $\mathcal{M}$ for
		\begin{equation}\label{eqn:prop_inner_iter}
			T_k:=\ceil[\bigg]{\frac{1}{\tau_k}\log(4t\gamma_k)}\qquad \textrm{iterations},
		\end{equation}
		
		thereby generating iterates $x_{k,1},\ldots,x_{k,T_k}$.
		
		Set $x_{k+1}=x_{k,T_k}$.
		\caption{Inexact prox-linear method: primal-only subsolves I}
		\label{alg: inex_prox_lin_primal}
\end{algorithm}}

The following lemma shows that the proposed number of inner iterations \eqref{eqn:prop_inner_iter} leads to significant progress in the prox-linear subproblems, compared with the initialization. Henceforth, we let $\mathbb{E}_{x_k}[\cdot]$ denote the expectation of a quantity conditioned on the iterate $x_k$.
\begin{lemma}\label{lem:trick_dep} The iterates  $x_k$ generated by Algorithm~\ref{alg: inex_prox_lin_primal} satisfy
	\begin{equation}\label{eqn:key_decr}
		\mathbb{E}_{x_k}[ F_t(x_{k+1}; x_k)-F_t(x^*_k; x_k)] \le  \frac{1}{4t} \norm{x_k-x^*_k}^2.
	\end{equation}
\end{lemma}
\begin{proof}
	In each iteration $k$, the linear convergence of algorithm $\mathcal{M}$ implies 
	\begin{align*}
		\mathbb{E}_{x_k}[F_t(x_{k+1}; x_k) - F_t(x_k^*;x_k)] &\le   \gamma_k \left ( 1- \tau_{k} \right )^{T_k}
		\|x_{k,0}-x^*_k\|^2 \\
		&\le \gamma_k e^{-\tau_{k}
			T_k}\|x_k-x^*_k\|^2\leq \frac{1}{4t}\|x_k-x^*_k\|^2,
	\end{align*}
	as claimed.
\end{proof}

With this lemma at hand, we can establish convergence guarantees of the inexact method.

\begin{theorem}[Convergence of Algorithm~\ref{alg: inex_prox_lin_primal}]
	Supposing $t \le \mu^{-1}$, the iterates $x_k$ generated by
	Algorithm~\ref{alg: inex_prox_lin_primal} satisfy
	\[ \min_{j=0, \hdots, N-1} \mathbb{E}[\norm{\mathcal{G}_t(x_{j}) }^2] \le
	\frac{4t^{-1} \left (F(x_0)-\inf F \right ) }{N}.\]
	\label{thm:conv_couple}
\end{theorem}

\begin{proof}
	The proof follows the same outline as Theorem~\ref{thm: funct_pla}. 
	Observe 
	\begin{align*}
		\mathbb{E}_{x_{k}}[F(x_{k}) -F(x_{k+1})] &= \mathbb{E}_{x_{k}}[F_t(x_{k}; x_{k})- F(x_{k+1})] \\
		&\ge \mathbb{E}_{x_{k}} [F_t(x_{k}^*; x_{k})- F(x_{k+1}) +
		\frac{1}{2t}
		\norm{x_{k}-x_{k}^*}^2]\\
		&\geq  \mathbb{E}_{x_{k}} [F_t(x_{k}^*; x_{k})-F_t(x_{k+1};x_{k})]+\frac{1}{2t} \norm{x_{k}-x_{k}^*}^2\\
		&\ge  - \frac{1}{4t} \norm{x_{k}-x^*_{k}}^2 +
		\frac{1}{2t} \norm{x_{k}-x^*_{k}}^2\\
		&\ge \frac{t}{4} \norm{\mathcal{G}_t(x_{k})}^2,
	\end{align*}
	%
	where the second line follows from strong convexity of $F_t(\cdot;x_{k})$, the third  from Lemma~\ref{lem:upper_lower}, and the fourth from Lemma~\ref{lem:trick_dep}. Taking expectations of both sides, and using the tower rule, we deduce
	$$\mathbb{E}[F(x_{k}) -F(x_{k+1})] \geq \frac{t}{4}\mathbb{E}[\norm{\mathcal{G}_t(x_{k})}^2].$$
	Summing up both sides, we deduce 
	\begin{align*}
		\frac{t}{4} \min_{j =0, \hdots, N-1}
		\mathbb{E}[\norm{\mathcal{G}_t(x_{j})}^2] \le \frac{t}{4N} \sum_{j=0}^{N-1}
		\mathbb{E}[\norm{\mathcal{G}_t(x_{j})}^2] \le \frac{1}{N} \sum_{j=0}^{N-1}
		\mathbb{E}[F(x_{j})-F(x_{j+1})] \le \frac{F(x_0)-\inf F}{N}, 
	\end{align*}
	as claimed. 
\end{proof}


It is clear from the proof that if the inner algorithm $\mathcal{M}$ satisfies \eqref{eqn:main_lin_con_ineq} with the expectation $\mathbb{E}_{x_k}$ omitted, then  Theorem~\ref{thm:conv_couple} holds with $\mathbb{E}$ omitted as well and with $\inf F$ replaced by $F^*:=\lf_{k\to\infty} F(x_i)$. In particular, let us suppose that we set $t=\mu^{-1}$ and let $\mathcal{M}$ be the fast-gradient method (Algorithm~\ref{alg:fast_grad_gen}) applied to the primal problem. 
Then in each iteration $k$, we can set $L_f=\|\nabla c(x_k)\|^2_{\textrm{op}}L_h$ and $\alpha=\mu$. Let us now determine $\gamma_k$ and $\tau_k$. Using the inequality $(1+\sqrt{\tfrac{\alpha}{2L_f}})^{-1}\leq 1-\sqrt{\tfrac{\alpha}{2L_f}}$ along with Theorem~\ref{eff:accle_meth}, we deduce we can set $\gamma_k=\frac{L_f}{4}$ and $\tau_k=\sqrt{\tfrac{\alpha}{2L_f}}$ for all indices $k$. Then each iteration of Algorithm~\ref{alg: inex_prox_lin_primal} performs 
$T=\ceil*{\sqrt{\frac{2\|\nabla c(x_k)\|^2_{\textrm{op}}L_h}{\mu}}\log\left(\|\nabla c(x_k)\|^2_{\textrm{op}}L_h/\mu\right)}$ iterations of the fast-gradient method, Algorithm~\ref{alg:fast_grad_gen}. Recall that each iteration of Algorithm~\ref{alg:fast_grad_gen} requires $8$ basic operations. Taking into account Theorem~\ref{alg: inex_prox_lin_primal}, we deduce that the overall scheme will produce a point $x$ satisfying $$\norm{\mathcal{G}_{\frac{1}{\mu}}(x)}\leq \varepsilon$$ 
after at most 
\begin{equation}\label{eqn:primal_only_total_exact}
	8\ceil*{  \frac{4\mu \left (F(x_0)- F^* \right ) }{\varepsilon^2}}\ceil*{\sqrt{\frac{2\|\nabla c\|^2L_h}{\mu}}\,\log\left(\frac{\|\nabla c\|^2L_h}{\mu}\right)}
\end{equation}
basic operations. Thus the number of basic operations is on the order of
\begin{equation}\label{eqn:boxed_eff_est_smooth_sing}
	\boxed{\mathcal{O}\left(  \frac{ \sqrt{\|\nabla c\|^2\cdot L_h\cdot \mu}\cdot\left (F(x_0)- F^* \right ) }{\varepsilon^2}~\log\left(\frac{\|\nabla c\|^2L_h}{\mu}\right)\right)}.
\end{equation}
Notice this estimate is better than \eqref{eq:first_box}, but only in terms of logarithmic dependence.
%

Before moving on, it is instructive to comment on the functional form of the linear convergence guarantee in \eqref{eqn:main_lin_con_ineq}. The right-hand-side depends on the initial squared distance $\|z_0-x^*\|^2$. Convergence rates of numerous algorithms, on the other hand, are often stated with the right-hand-side instead depending on the initial functional error  $\displaystyle F_t(z_0;x)-\inf_z F_t(z; x)$. In particular, this is the case for  algorithms for finite sum problems discussed in Section~\ref{sec:fin_sum_prob}, such as SVRG \cite{svrg} and SAGA \cite{saga}, and their accelerated extensions \cite{catalyst,accsvrg,frostig}. The following easy lemma shows how any such algorithm can be turned into a linearly convergent subscheme, in the sense of Definition~\ref{defn:lin_conv_subscheme}, by taking a single extra prox-gradient step. We will use this observation in Section~\ref{sec:fin_sum_prob}, when discussing finite-sum problems.

\begin{lemma}\label{lem:put_stand_form}
	Consider an optimization problem having the convex additive composite form \eqref{eqn:conv_comp_gen_accel}. Suppose $\mathcal{M}$ is an algorithm for $\min_z f^p(z)$ satisfying: there exist constants $\gamma\geq 0$ and $\tau\in (0,1)$ so that on any input $z_0$, the method $\mathcal{M}$ generates a sequence $\{z_i\}_{i=1}^{\infty}$ satisfying 
	\begin{equation}\label{eqn:cray_znot}
		\mathbb{E}[f^p(z_i) - f^p(z^*)] \le \gamma \left ( 1- \tau \right )^i
		(f^p(z_{0})-f^p(z^*) ) \qquad \textrm{for } i=1,\ldots,\infty,
	\end{equation}
	where $z^*$ is a minimizer of $f^p$. Define an augmented method $\mathcal{M}^+$ as follows: given input $z_0$,  initialize $\mathcal{M}$ at the point $\prox_{p/L_f}(z_0-\frac{1}{L_f}\nabla f(z_0))$ and output the resulting points $\{z_i\}_{i=1}^\infty$. Then the iterates generates by $\mathcal{M}^+$ satisfy
	$$\mathbb{E}[f^p(z_i) - f^p(z^*)] \le \frac{\gamma L_f}{2} \left ( 1- \tau \right )^i\|z_0-z^*\|^2 \qquad \textrm{for } i=1,\ldots,\infty,$$
\end{lemma}
\begin{proof}
	Set $\hat z:=\prox_{p/L_f}(z_0-\frac{1}{L_f}\nabla f(z_0))$. Then convergence guarantees \eqref{eqn:cray_znot} of $\mathcal{M}$, with $\hat z$ in place of $z_0$, read
	$$\mathbb{E}[f^p(z_i) - f^p(z^*)] \le \gamma \left ( 1- \tau \right )^i
	(f^p(\hat z)-f^p(z^*) ) \qquad \textrm{for } i=1,\ldots,\infty.$$
	Observe the inequality  $f^p(\hat z)\leq f(z_0)+\langle \nabla f(z_0),\hat z-z_0\rangle+p(\hat z)+\frac{L_{f}}{2}\|\hat z-z_0\|^2$. By definition, $\hat z$ is the minimizer of the function $z\mapsto f(z_0)+\langle \nabla f(z_0),z-z_0\rangle+p(z)+\frac{L_{f}}{2}\|z-z_0\|^2$, and hence we deduce $f^p(\hat z)\leq  f(z_0)+\langle \nabla f(z_0),z^*-z_0\rangle+p(z^*)+\frac{L_{f}}{2}\|z^*-z_0\|^2\leq f^p(z^*)+\frac{L_{f}}{2}\|z^*-z_0\|^2$, with the last inequality follows from convexity of $f$. The result follows.
\end{proof}

\subsection{Total cost of the smoothing strategy}\label{sec:tot_cost_smoothing_strat}
The final ingredient is to replace $h$ by a smooth approximation
and then minimize the resulting composite function by an inexact prox-linear method (Algorithms~\ref{alg: inex_prox_lin_subgrad_exp} or \ref{alg: inex_prox_lin_primal}). 
Define the smoothed composite function
\begin{equation}\label{eqn:smoothed_prob}
	F^{\nu}(x):=g(x)+h_{\nu}(c(x)),
\end{equation}
where $h_{\nu}$ is the Moreau envelope of $h$.
Recall from Lemma~\ref{lem:lip_cont} the three key properties of the Moreau envelope: $$\lip(h_{\nu})\leq L,\qquad\lip(\nabla h_\nu)\leq \frac{1}{\nu},$$ and 
$$0\leq h(z)-h_{\nu}(z)\leq \frac{L^2\nu}{2}\qquad \textrm{ for all } z\in \R^m.$$
Indeed, these are the only properties of the smoothing we will use; therefore, in the analysis, any smoothing satisfying the analogous properties can be used instead of the Moreau envelope.


Let us next see how to choose the smoothing parameter $\nu>0$ based on a target accuracy $\varepsilon$ on the norm of the prox-gradient $\|\mathcal{G}_{t}(x)\|$. Naturally, we must establish a relationship between the step-sizes of the prox-linear steps on the original problem and its smooth approximation. To distinguish between these two settings, we will use the notation
\begin{align*}
	x^+ &=\argmin_z \left \{ h \big (c(x) + \nabla c(x) (z-x) \big ) +
	g(z) + \tfrac{1}{2t} \norm{z-x}^2 \right \},\\
	\widehat x &= \argmin_z \left \{ h_\nu\big ( c(x) + \nabla c(x) (z-x)
	\big ) + g(z) + \tfrac{1}{2t} \norm{z-x}^2 \right \}, \\
	\mathcal{G}_t(x)&=t^{-1}(x^+-x),\\
	\mathcal{G}^{\nu}_t(x)&=t^{-1}(\widehat x-x).
\end{align*}
Thus $\mathcal{G}_t(x)$ is the prox-gradient on the target problem \eqref{eqn:comp2} as always, while $\mathcal{G}^{\nu}_t(x)$ is the prox-gradient on the smoothed problem \eqref{eqn:smoothed_prob}. The following theorem will motivate our strategy for choosing the smoothing parameter $\nu$.

\begin{theorem}[Prox-gradient comparison]\label{thm:smooth_est}
	For any point $x$, the inequality holds:
	$$\left\|\mathcal{G}_{t}(x)\right\|\leq \left\|\mathcal{G}^{\nu}_{t}(x)\right\|+\sqrt{\frac{L^2\nu}{2t}}.$$
\end{theorem}
\begin{proof} Applying Lemma~\ref{lem:lip_cont} and strong convexity of the proximal subproblems, we deduce
	\begin{align*}
		F_{t}(x^+;x)
		& \le F_{t}(\widehat x;x) - \frac{1}{2t} \norm{\widehat{x}-x^+}^2\\
		&\le \Big(h_{\nu} \big (c(x) + \nabla c(x) (\widehat{x} -x) \big )  + g(\widehat{x})+ \frac{1}{2t}
		\norm{\widehat{x}-x}^2  \Big)+ \frac{L^2 \nu}{2} - \frac{1}{2t} \norm{\widehat{x}-x^+}^2\\
		&\le \Big(h_\nu \big ( c(x) + \nabla c(x) (x^+-x) \big )  + g(x^+) + \frac{1}{2t} \norm{x^+-x}^2 \Big)+ \frac{L^2
			\nu}{2} - t^{-1}
		\norm{\widehat{x}-x^+}^2\\
		&\le F_{t}(x^+;x)  + \frac{L^2
			\nu}{2}  - t^{-1} \norm{\widehat{x}-x^+}^2.
	\end{align*}
	Canceling out like terms, we conclude
	$ t^{-1}\norm{\widehat{x}-x^+}^2 \le \frac{L^2 \nu}{2}$. 
	The triangle inequality then yields
	\[ t^{-1} \norm{x^+-x} \le t^{-1}\norm{\widehat{x}-x} +\sqrt{\frac{L^2\nu}{2t}}, \]
	as claimed.`
\end{proof}

Fix a target accuracy $\varepsilon>0$. The strategy for choosing the smoothing parameter $\nu$ is now clear. Let us set $t=\frac{1}{\mu}$ and then ensure $\frac{\varepsilon}{2}=\sqrt{\frac{L^2\nu}{2t}}$ by setting $\nu:=\frac{\varepsilon^2}{2L^3\beta}$.
Then by Theorem~\ref{thm:smooth_est}, any point $x$ satisfying $\|\mathcal{G}^{\nu}_{1/\mu}(x)\|\leq \frac{\varepsilon}{2}$ would automatically satisfy the desired condition $\|\mathcal{G}_{1/\mu}(x)\|\leq \varepsilon$. Thus we must only estimate the cost of obtaining such a point $x$. Following the discussion in Section~\ref{sec:total_cost_smooth}, we can apply either of the Algorithms~\ref{alg: inex_prox_lin_subgrad_exp} or \ref{alg: inex_prox_lin_primal}, along with the fast-gradient method (Algorithm~\ref{alg:fast_grad_gen}) for the inner subsolves, to the problem $\min_x F^{\nu}(x)=g(x)+h_{\nu}(c(x))$. We note that for a concrete implementation, one needs the following formulas, complementing Lemma~\ref{lem:expl_grad_comp}.

\begin{lemma}\label{lem:der_more_env}
	For any point $x$ and real $\nu$, $t>0$, the following are true:
	$$\prox_{th_{\nu}}(x)=(\tfrac{\nu}{t+\nu})\cdot x+(\tfrac{t}{t+\nu})\cdot\prox_{(t+\nu)h}(x)\qquad \textrm{and}\qquad \nabla{h_{\nu}}(x)=\tfrac{1}{\nu}(x-\prox_{\nu h}(x)).$$
\end{lemma}
\begin{proof}
	The expression $\nabla{h_{\nu}}(x)=\tfrac{1}{\nu}(x-\prox_{\nu h}(x))$ was already recorded in Lemma~\ref{lem:lip_cont}. 
	Observe the chain of equalities
	\begin{align}
		\min_y~ \left\{h_{\nu}(y)+\frac{1}{2t}\|y-x\|^2\right\}&=\min_y\min_z~ \left\{h(z)+\frac{1}{2\nu}\|z-y\|^2+\frac{1}{2t}\|y-x\|^2\right\}\label{eqn:where take_der}\\
		&=\min_z \left\{h(z)+\frac{1}{2(t+\nu)}\left\|z-x\right\|^2\right\},\notag
	\end{align}
	where the last equality follows by exchanging the two mins in \eqref{eqn:where take_der}. By the same token, taking the derivative with respect to $y$ in \eqref{eqn:where take_der}, we conclude that the optimal pair $(y,z)$ must satisfy the equality $0=\nu^{-1}(y-z)+t^{-1}(y-x)$. Since the optimal $y$ is precisely $\prox_{th_{\nu}}(x)$ and the optimal $z$ is given by $\prox_{(t+\nu)h}(x)$, the result follows.
	%
	%
	%
\end{proof}

Let us apply Algorithm~\ref{alg: inex_prox_lin_subgrad_exp} with the fast-gradient dual subsolves, as described in Section~\ref{sec:total_cost_smooth}.
Appealing to \eqref{eqn:dual_stat_exact} with $L_h=\tfrac{1}{\nu}=\frac{2L^3\beta}{\varepsilon^2}$ and $\varepsilon$ replaced by $\varepsilon/2$, we deduce that the scheme will find a point $x$ satisfying 
$\|\mathcal{G}_{1/\mu}(x)\|\leq \varepsilon$
after at most 
$$\mathcal{N}(\varepsilon)\cdot\left(10+9\ceil*{{\frac{\|\nabla c\|L}{\varepsilon}}\cdot\log\left(\frac{8\|\nabla c\|^4L^2(1+\mathcal{N}(\varepsilon))^4}{\beta^2}\right)}\right)$$
basic operations, where $\mathcal{N}(\varepsilon):=\ceil*{\frac{16\mu\left(F(x_0)-\inf F+8+\tfrac{\varepsilon^2}{4\mu}\right)}{\varepsilon^2}}$.
Hence the total cost is on the order\footnote{\label{foot:assympt_2} Here, we use the asymptotic notation described in Remark~\ref{rem:assympt_not} with $\omega=(\|\nabla c\|, L,\beta,F(x_0)-\inf F,1/\varepsilon)$.} of
\begin{equation}\label{eqn:final_cost1}
	\boxed{\mathcal{O}\left({\frac{L^2\beta\|\nabla c\|\cdot(F(x_0)-\inf F)}{\varepsilon^3}}{~\log\left(\frac{\|\nabla c\|^2L^3\beta(F(x_0)-\inf F)^2}{\varepsilon^4}\right)}\right)}.
\end{equation}

Similarly, let us apply Algorithm~\ref{alg: inex_prox_lin_primal} with fast-gradient primal subsolves, as described in Section~\ref{sec:total_cost_smooth}.
Appealing to \eqref{eqn:primal_only_total_exact}, we deduce that the scheme will find a point $x$ satisfying $\|\mathcal{G}_{1/\mu}(x)\|\leq \varepsilon$
after at most 
$$8\ceil*{  \frac{16\mu \left (F(x_0)- \inf F+\tfrac{\varepsilon^2}{4\mu} \right ) }{\varepsilon^2}}\ceil*{{\frac{2\|\nabla c\| L}{\varepsilon}}\,\log\left(\frac{2\|\nabla c\|^2L^2}{\varepsilon^2}\right)}$$
basic operations. Thus the cost is on the order\cref{foot:assympt_2} of
\begin{equation}\label{eqn:final_cost2}
	\boxed{\mathcal{O}\left(  \frac{ L^2\beta\|\nabla c\|\cdot\left (F(x_0)- \inf F \right ) }{\varepsilon^3}~\log\left(\frac{\|\nabla c\| L}{\varepsilon}\right)\right)}.
\end{equation}
Notice that the two estimates \eqref{eqn:final_cost1} and \eqref{eqn:final_cost2} are identical up to a logarithmic dependence on the problem data.
To the best of our knowledge, these are the best-known efficiency estimates of any first-order method for the composite problem class \eqref{eqn:comp2}. 

The logarithmic dependence in the estimates \eqref{eqn:final_cost1} and \eqref{eqn:final_cost2} can be removed entirely, by a different technique, provided we have available an accurate estimate on $F(x_0)-\inf F$ and an a priori known estimate $\|\nabla c\|$ to be used throughout the procedure. Since we feel that the resulting scheme is less practical than the ones outlined in the current section, we have placed the details in Appendix~\ref{sec:app_dual_meth}.

\section{Finite sum problems}\label{sec:fin_sum_prob}
In this section, we extend the results of the previous sections to so-called ``finite sum problems'', also often called ``regularized empirical risk minimization''.
More precisely, throughout the section instead of minimizing a single composite function, we will be interested in minimizing an average of $m$ composite  functions:
\begin{equation}\label{eqn:sum_problem}
	\min_{x} ~ F(x):=\frac{1}{m}\sum_{i=1}^m h_i(c_i(x))+g(x)
\end{equation}
In line with the previous sections, we make the following assumptions on the components of the problem:
\begin{enumerate}
	\item $g$ is a closed convex function;
	\item 	$h_i\colon\R\to\R$ are convex, and $ L$-Lipschitz continuous;
	\item $c_i\colon\R^{d}\to\R$ are $C^1$-smooth with the gradient map $\nabla c_i$ that is $\beta$-Lipschitz continuous.
\end{enumerate}
We also assume that we have available a real value, denoted $\|\nabla c\|$, satisfying
$${\|\nabla c\|}\geq \sup_{x\in \dom g} \max_{i=1,\ldots, m}\|\nabla c_i(x)\|.$$

The main conceptual premise here is that $m$ is large and should be treated as an explicit parameter of the problem. Moreover,
notice the Lipschitz data is stated for the individual functional components of the problem. Such finite-sum problems are ubiquitous in machine learning and data science, where $m$ is typically the (large) number of recorded measurements of the system. Notice that we have assumed that $c_i$ maps to the real line. This is purely for notational convenience. Completely analogous results, as in this section, hold when $c_i$ maps into a higher dimensional space.

Clearly, the finite-sum problem \eqref{eqn:sum_problem} is an instance of the composite problem class \eqref{eqn:comp2} under the identification
\begin{equation}\label{eqn:h_c}
	h(z_i,\ldots,z_m):=\frac{1}{m}\sum_{i=1}^m h_i(z_i)\qquad \textrm{and}\qquad c(x):= (c_1(x),\ldots,c_m(x)).
\end{equation}
Therefore, given a target accuracy $\varepsilon>0$, we again seek to find a point $x$ with a small prox-gradient $\|\mathcal{G}_t(x)\|\leq \varepsilon$. In contrast to the previous sections, by a {\em basic operation} we will mean individual evaluations of $c_i(x)$ and $\nabla c_i(x)$, dot-products $\nabla c_i(x)^Tv$, and proximal operations  $\prox_{th_i}$ and $\prox_{tg}$. 

Let us next establish baseline efficiency estimates by simply using the inexact prox-linear schemes discussed in Sections~\ref{sec:total_cost_smooth} and \ref{sec:tot_cost_smoothing_strat}. To this end, the following lemma derives Lipschitz constants of $h$ and $\nabla c$ from the problem data $ L$ and ${\beta}$. The proof is elementary and we have placed it in Appendix~\ref{sec:proofs_append}.
Henceforth, we set $\lip(\nabla c):=\sup_{x\neq y} \frac{\|\nabla c(x)-\nabla c(y)\|_{\textrm{op}}}{\|x-y\|}$. 

\begin{lemma}[Norm comparison]\label{lem:norm_comp}
	The inequalities hold:
	\begin{align*}
		&\lip(h)\leq {L}/\sqrt{m},  \qquad\lip(\nabla c)\leq {\beta}\sqrt{m},\qquad \|\nabla c(x)\|_{\textrm{op}}\leq \sqrt{m}\left(\max_{i=1,\ldots,m} \|\nabla c_i(x)\|\right)~~\forall x.
	\end{align*}	
	If in addition each $h_i$ is $C^1$-smooth with ${L_h}$-Lipschitz derivative $t\mapsto h_i'(t)$, then the inequality, $\lip(\nabla h)\leq{L_h}/m$, holds as well. 
\end{lemma}

\begin{remark}[Notational substitution]\label{rem:not_caut}
	{\rm
		We will now apply the results of the previous sections to the finite sum problem \eqref{eqn:sum_problem} with $h$ and $c$ defined in \eqref{eqn:h_c}. In order to correctly interpret results from the previous sections, according to Lemma~\ref{lem:norm_comp}, we must be mindful to replace $L$ with $L/\sqrt{m}$, $\beta$ with $\beta\sqrt{m}$, $\|\nabla c\|$ with $\sqrt{m}\|\nabla c\|$, and $L_h$ with ${L_h}/m$. In particular, observe that we are justified in setting $\mu:=L\beta$ without any ambiguity. Henceforth, we will be using this substitution routinely.}
\end{remark}

%
%
%

\subsubsection*{Baseline efficiency when $h_i$ are smooth:}
Let us first suppose  that $h_i$ are $C^1$-smooth with ${L_h}$-Lipschitz derivative and interpret the efficiency estimate \eqref{eqn:boxed_eff_est_smooth_sing}. 
%
Notice  that each gradient evaluation $\nabla c$ requires $m$ individual gradient evaluations $\nabla c_i$. Thus multiplying \eqref{eqn:boxed_eff_est_smooth_sing} by $m$ and using Remark~\ref{rem:not_caut}, the efficiency estimate \eqref{eqn:boxed_eff_est_smooth_sing} reads: 
\begin{equation}\label{eqn:boxed_eff_est_smooth_sing_fin_sum}
	\boxed{\mathcal{O}\left(  \frac{ m\sqrt{{\|\nabla c\|}^2\cdot {L_h}\cdot {L}\cdot{\beta}}\cdot\left (F(x_0)- \inf F \right ) }{\varepsilon^2}~\log\left(\frac{{\|\nabla c\|}^2{L_h}}{{L}{\beta}}\right)\right)}
\end{equation}
 basic operations.

\subsubsection*{Baseline efficiency when $h_i$ are nonsmooth:}
Now let us apply the smoothing technique described in Section~\ref{sec:tot_cost_smoothing_strat}. Multiplying the efficiency estimate \eqref{eqn:final_cost2} by $m$ and using Remark~\ref{rem:not_caut} yields:
\begin{equation}\label{eqn:fin_sum_base_nonsmooth}
	\boxed{\mathcal{O}\left(  \frac{ m\cdot{L}^2{\beta}{\|\nabla c\|}\cdot\left (F(x_0)- \inf F \right ) }{\varepsilon^3}~\log\left(\frac{{\|\nabla c\|} {L}}{\varepsilon}\right)\right)}
\end{equation}
basic operations.

\bigskip

The two displays \eqref{eqn:boxed_eff_est_smooth_sing_fin_sum} and \eqref{eqn:fin_sum_base_nonsmooth} serve as baseline efficiency estimates for obtaining a point $x$ satisfying $\|\mathcal{G}_{1/{\mu}}(x) \|\leq \varepsilon$.
We will now see that one can improve these guarantees in expectation. The strategy is perfectly in line with the theme of the paper. We will replace $h$ by a smooth approximation, then apply an inexact prox-linear Algorithm \ref{alg: inex_prox_lin_primal}, while approximately solving each subproblem by an ``(accelerated) incremental method''. Thus the only novelty here is a different scheme for approximately solving the proximal subproblems.

\subsection{An interlude: incremental algorithms}

There are a number of popular algorithms for finite-sum problems, including SAG \cite{sag}, SAGA \cite{SAGA2}, SDCA \cite{sdca}, SVRG \cite{svrg,prox_SVRG}, FINITO \cite{finito}, and MISO \cite{miso}. All of these methods have similar linear rates of convergence, and differ only in storage requirements and in whether one needs to know explicitly the strong convexity constant. For the sake of concreteness, we will focus on SVRG following \cite{prox_SVRG}. This scheme applies to finite-sum problems
\begin{equation}\label{eqn:fin_sum_prob_conv_comp}
	\min_x~ f^p(x):= \frac{1}{m}\sum_{i=1}^m f_i(x)+p(x),
\end{equation}
where $p$ is a closed, $\alpha$-strongly convex function ($\alpha> 0$) and each $f_i$ is convex and $C^1$-smooth with $\ell$-Lipschitz gradient $\nabla f_i$. For notational convenience, define the condition number $\kappa:=l/\alpha$. Observe that when each $h_i$ is smooth, each proximal subproblem indeed has this form: 
\begin{equation}\label{eqn:prox_subprob_fin_sum}
	\min_{z} ~ F_t(z;x):=\frac{1}{m}\sum_{i=1}^m h_i\Big(c_i(x)+\langle\nabla c_i(x),z-x\rangle\Big)+g(z)+\frac{1}{2t}\|z-x\|^2.
\end{equation}

In Algorithm~\ref{alg:prox_SVRG}, we record the Prox-SVRG method of \cite{prox_SVRG} for minimizing the function \eqref{eqn:fin_sum_prob_conv_comp}.

{\LinesNotNumbered
	\begin{algorithm}[h!]
		\SetKw{Null}{NULL}
		\SetKw{Return}{return}
		\Initialize{A point $\widetilde{x}_0\in\R^d$, a real $\eta>0$, a positive integer $J$.
		}
		{\bf Step s:} ($s\geq 1)$  \\
		$\widetilde x=\widetilde{x}_{s-1}$;\\
		$\widetilde v= \frac{1}{m}\sum_{i=1}^m \nabla f_i(\widetilde{x})$;\\
		$x_0=\widetilde{x}$\\
		\For{$j=1,2,\ldots, J$}
		{pick $i_j\in \{1,\ldots,m\}$ uniformly at random\\
			$v_j=\widetilde v+(\nabla f_{i_j}(x_{j-1})-\nabla f_{i_j}(\widetilde{x}))$\\
			$x_j=\prox_{\eta p}(x_{j-1}-\eta v_j)$
		}
		$\widetilde{x}_s=\frac{1}{J}\sum_{j=1}^J x_j$

		\caption{The Prox-SVRG method \cite{prox_SVRG}}\label{alg:prox_SVRG}
\end{algorithm}}

The following theorem from \cite[Theorem 3.1]{prox_SVRG} summarizes convergence guarantees of Prox-SVRG.
\begin{theorem}[Convergence rate of Prox-SVRG]
	Algorithm~\ref{alg:prox_SVRG},  with the choices $\eta=\frac{1}{10\ell}$ and $J=\ceil*{100\kappa}$, will generate a sequence $\{\widetilde{x}_s\}_{s\geq 1}$ satisfying
	$$\mathbb{E}[f^p(\widetilde{x}_s)-f^p(x^*)]\leq 0.9^s(f^p(\widetilde{x}_0)-f^p(x^*)),$$
	where $x^*$ is the minimizer of $f^p$. Moreover, each step $s$ requires $m+2\ceil*{100\kappa}$ individual gradient $\nabla f_i$ evaluations. 
\end{theorem}

Thus Prox-SVRG will generate a point $x$ with $\mathbb{E}[f^p(x)-f^p(x^*)]\leq \varepsilon$ after at most 
\begin{equation}\label{eqn:prox_svrg_guarantee}
	\mathcal{O}\left(\left(m+\kappa\right)~\log\left(\frac{f^p(\widetilde{x}_0)-f^p(x^*)}{\varepsilon}\right)\right)
\end{equation}
individual gradient $\nabla f_i$ evaluations. It was a long-standing open question whether there is a method that improves the dependence of this estimate on the condition number $\kappa$. This question was answered positively by a number of algorithms, including Catalyst \cite{catalyst}, accelerated SDCA \cite{accsdca}, APPA \cite{frostig}, RPDG \cite{conjugategradient}, and Katyusha \cite{accsvrg}. For the sake of concreteness, we focus only on one of these methods, Katyusha \cite{accsvrg}. This scheme follows the same epoch structure as SVRG, while incorporating iterate history. We summarize convergence guarantees of this method, established in \cite[Theorem 3.1]{accsvrg}, in  the following theorem.

\begin{theorem}[Convergence rate of Katyusha]\label{thm:kat_rate}
	The Katyusha algorithm of \cite{accsvrg} generates a sequence of iterates $\{\widetilde{x}_s\}_{s\geq 1}$ satisfying
	$$
	\frac{\mathbb{E}[f^p(\widetilde{x}_s)-f^p(x^*)]}{f^p(\widetilde{x}_0)-f^p(x^*)}\leq
	\left\{
	\begin{aligned}
	&4\cdot\left(1+\sqrt{1/(6\kappa m)}\right)^{-2sm}&,\qquad \textrm{if } \tfrac{m}{\kappa}\leq \tfrac{3}{8}\\
	&3\cdot(1.5)^{-s}&,\qquad \textrm{if } \tfrac{m}{\kappa}> \tfrac{3}{8}
	\end{aligned}
	\right.
	$$
	where $x^*$ is the minimizer of $f^p$. Moreover, each step $s$ requires $3m$ individual gradient $\nabla f_i$ evaluations.\footnote{The constants $4$ and $3$ are hidden in the $\mathcal{O}$ notation in \cite[Theorem 3.1]{accsvrg}. They can be explicitly verified by following along the proof.}
\end{theorem}



To simplify the expression for the rate, using the inequality $(1+z)^m\geq 1+mz$ observe 
$${\left(1+\sqrt{\tfrac{1}{6\kappa m}}\right)^{-2sm}}\leq {\left(1+\sqrt{\tfrac{2m}{3\kappa}}\right)^{-s}}.$$
Using this estimate in Theorem~\ref{thm:kat_rate} simplifies the linear rate to
$$\frac{\mathbb{E}[f^p(\widetilde{x}_s)-f^p(x^*)]}{f^p(\widetilde{x}_0)-f^p(x^*)}\leq 4\cdot \max\left\{ \left(1+\sqrt{\tfrac{2m}{3\kappa}}\right)^{-s}, 1.5^{-s}\right\}.$$
Recall that  each iteration of Katyusha requires $3m$ individual gradient $\nabla f_i$ evaluations. Thus the method will
generate a point $x$ with $\mathbb{E}[f^p(x)-f^p(x^*)]\leq \varepsilon$ after at most 
$$\mathcal{O}\left(\left(m+\sqrt{m\kappa}\,\right)~\log\left(\frac{f^p(\widetilde{x}_0)-f^p(x^*)}{\varepsilon}\right)\right)$$ 
individual gradient $\nabla f_i$ evaluations. Notice this efficiency estimate is  significantly better than the  guarantee \eqref{eqn:prox_svrg_guarantee} for Prox-SVRG only when $m\ll \kappa$. This setting is very meaningful in the context of smoothing. Indeed, since we will be applying accelerated incremental methods to proximal subproblems after a smoothing, the condition number $\kappa$ of each subproblem can be huge.

\subsubsection*{Improved efficiency estimates when $h_i$ are smooth:}
Let us now suppose that each $h_i$ is $C^1$-smooth with ${L_h}$-Lipschitz derivative $h_i'$. We seek to determine the efficiency of the inexact prox-linear method (Algorithm~\ref{alg: inex_prox_lin_primal}) that uses either Prox-SVRG or Katyusha as the linearly convergent subscheme $\mathcal{M}$. Let us therefore first look at the efficiency of Prox-SVRG and Katyusha on the prox-linear subproblem \eqref{eqn:prox_subprob_fin_sum}. Clearly we can set 
$$\ell:={L_h}\cdot\left(\max_{i=1,\ldots,m}\|\nabla c_i(x)\|^2\right)\qquad \textrm{ and }\qquad \alpha=t^{-1}.$$
Notice that  the convergence guarantees for Prox-SVRG and Katyusha are not in the standard form \eqref{eqn:main_lin_con_ineq}. Lemma~\ref{lem:put_stand_form}, however, shows that they can be put into standard form by taking a single extra prox-gradient step in the very beginning of each scheme; we'll call these slightly modified schemes Prox-SVRG$^+$ and Katyusha$^+$.
Taking into account Lemma~\ref{lem:norm_comp}, observe that the gradient of the function $z\mapsto h(c(x)+\nabla c(x)(z-x))$ is $l$-Lipschitz continuous.
Thus according to Lemma~\ref{lem:put_stand_form}, Prox-SVRG$^+$ and Katyusha$^+$ on input $\widetilde{z}_0$ satisfy 
\begin{align*}
	\mathbb{E}[F_t(\widetilde{z}_s;x) - F_t(z^*;x)] &\le \frac{\ell}{2} \cdot  0.9^s\|\widetilde{z}_0-z^*\|^2,\\
	\mathbb{E}[F_t(\widetilde{z}_s;x) - F_t(z^*;x)] &\le \frac{4 \ell}{2} \cdot\max\left\{ \left(1+\sqrt{\tfrac{2m}{3\kappa}}\right)^{-s}, 1.5^{-s}\right\}\cdot\|\widetilde{z}_0-z^*\|^2 ,
\end{align*}
for  $s=1,\ldots,\infty$,
respectively, where $z^*$ is the minimizer of $F_t(\cdot;x)$.

We are now ready to compute the total efficiency guarantees.
Setting $t=1/\mu$, Theorem~\ref{thm:conv_couple} shows that Algorithm~\ref{alg: inex_prox_lin_primal} will generate a point $x$ with 
$$\mathbb{E}\left[\|\mathcal{G}_{1/\mu}(x)\|^2\right]\leq \varepsilon^2$$ after at most
$\ceil*{\frac{4\mu(F(x_0)-\inf F)}{\varepsilon^2}}$ iterations. Each iteration $k$ in turn requires at most
$$\ceil*{\frac{1}{\tau_k}\log(4t\gamma_k)}\leq\ceil[\bigg]{\frac{1}{0.1}\log\Big(4\cdot\frac{1}{\mu}\cdot\frac{{L_h}\cdot{\|\nabla c\|}^2}{2}\Big)}$$
iterations of Prox-SVRG$^+$ and at most
$$\ceil*{\frac{1}{\tau_k}\log(4t\gamma_k)}\leq\ceil[\bigg]{\max\left\{3,\left(1+\sqrt{\tfrac{3{L_h}{\|\nabla c\|}^2}{2m\mu}}\right)\right\}\log\Big(4\cdot\frac{1}{\mu}\cdot\frac{4\cdot{L_h}\cdot{\|\nabla c\|}^2}{2}\Big)}$$
iterations of Katyusha$^+$. Finally recall that each iteration $s$ of Prox-SVRG$^+$ and of Katyusha$^+$, respectively, requires $m+2\ceil*{\frac{100{L_h}{\|\nabla c\|}^2}{\mu}}$ and $3m$ evaluations of $\nabla c_i(x)^Tv$. Hence the overall efficiency  is on the order of 
\begin{equation}\label{eqn:prox-svrg_smooth_total}
	\boxed{\mathcal{O}\left(\frac{\left(\mu m+{L_h}{\|\nabla c\|}^2\right)\cdot(F(x_0)-\inf F)}{\varepsilon^2}~\log\left(\frac{{L_h}\cdot{\|\nabla c\|}^2}{\mu}\right)\right)}
\end{equation}
when using Prox-SVRG$^+$ and on the order of 
\begin{equation}\label{eqn:kat_smooth_total}
	\boxed{\mathcal{O}\left(\frac{\left(\mu m+\sqrt{{\mu m{L_h}{\|\nabla c\|}^2}}\right)\cdot(F(x_0)-\inf F)}{\varepsilon^2}~\log\left(\frac{{L_h}\cdot{\|\nabla c\|}^2}{\mu}\right)\right)}
\end{equation}
when using Katyusha$^+$.  Notice that the estimate \eqref{eqn:kat_smooth_total} is better than \eqref{eqn:prox-svrg_smooth_total} precisely when $m\ll\frac{{L_h}{\|\nabla c\|}^2}{\mu}$. 

\subsubsection*{Improved efficiency estimates when $h_i$ are nonsmooth:}

Finally, let us now no longer suppose that $h_i$ are smooth in the finite-sum problem \eqref{eqn:sum_problem} and instead apply the smoothing technique. To this end, observe the equality
$$h_{\nu}(z)=\inf_{y} \left\{\frac{1}{m}\sum_{i=1}^m h_i(y_i)+\frac{1}{2\nu}\|y-z\|^2\right\}=\sum_{i=1}^m ({h_i/m})_{\nu}(z_i).$$
Therefore the smoothed problem in \eqref{eqn:smoothed_prob} is also a finite-sum problem with
$$\min_{x} ~ \frac{1}{m}\sum_{i=1}^m m\cdot(h_i/m)_{\nu}(c_i(x))+g(x).$$
Thus we can can apply the convergence estimates we have just derived in the smooth setting with $h_i(t)$ replaced by  $\phi_i(t):= m\cdot(h_i/m)_{\nu}(t)$. Observe that $\phi_i$ is ${L}$-Lipschitz by Lemma~\ref{lem:lip_cont}, while the derivative  $\phi_i'(t)=m\cdot \nu^{-1}(t-\prox_{\frac{\nu}{m} h_i}(t))$ is Lipschitz with constant ${L_h}:=\frac{m}{\nu}$. Thus according to the recipe following Theorem~\ref{thm:smooth_est}, given a target accuracy $\varepsilon>0$ for the norm of the prox-gradient $\|\mathcal{G}_{\frac{1}{\mu}}(x)\|$, we should set 
$$\nu:=\frac{m\varepsilon^2}{2{L}^3{\beta}},$$
where we have used the substitutions dictated by Remark~\ref{rem:not_caut}.
Then Theorem~\ref{thm:smooth_est} implies
$$\left\|\mathcal{G}_{1/\mu}(x)\right\|\leq \left\|\mathcal{G}^{\nu}_{1/\mu}(x)\right\|+\frac{\varepsilon}{2} \qquad \textrm{ for all }x,$$
where $\left\|\mathcal{G}^{\nu}_{1/\mu}(x)\right\|$ is the prox-gradient for the smoothed problem. Squaring and taking expectations on both sides, we can be sure $\mathbb{E}[\left\|\mathcal{G}_{1/\mu}(x)\right\|^2]\leq \varepsilon^2$ if we find a point $x$ satisfying $\mathbb{E}\left[\left\|\mathcal{G}^{\nu}_{1/\mu}(x)\right\|^2\right]\leq \frac{\varepsilon^2}{4}$. 
Thus we must simply write the estimates \eqref{eqn:prox-svrg_smooth_total} and \eqref{eqn:kat_smooth_total} for the smoothed problem in terms of the original problem data.
Thus to obtain a point $x$
satisfying 
$$\mathbb{E}[\left\|\mathcal{G}_{1/\mu}(x)\right\|^2]\leq \varepsilon^2,$$
it suffices to perform 
\begin{equation}\label{eqn:kat_nonsmooth_total}
	\boxed{\mathcal{O}\left(\left(\frac{{L}{\beta} m}{\varepsilon^2}+\frac{{L}^2{\beta}{\|\nabla c\|}}{\varepsilon^3}\cdot\min\left\{\sqrt{m}, \frac{{L}{\|\nabla c\|}}{\varepsilon}\right\}\right)\cdot(F(x_0)-\inf F)~\log\left(\frac{{L}\cdot{\|\nabla c\|}}{\varepsilon}\right)\right)}
\end{equation}
basic operations. The $\min$ in the estimate corresponds to choosing the better of the two, Prox-SVRG$^+$ and Katyusha$^+$, in each proximal subproblem in terms of their efficiency estimates. Notice that the $1/\varepsilon^3$ term in  \eqref{eqn:kat_nonsmooth_total} scales only as $\sqrt{m}$. Therefore this estimate is an order of magnitude better than our baseline \eqref{eqn:fin_sum_base_nonsmooth}, which we were trying to improve.
The caveat is of course that the estimate \eqref{eqn:kat_nonsmooth_total} is in expectation while \eqref{eqn:fin_sum_base_nonsmooth} is deterministic.

\section{An accelerated prox-linear algorithm}\label{sec:accel_conv_comp}
Most of the paper thus far has focused on the setting when the proximal subproblems \eqref{alg: prox_lin} can only be approximately solved by first-order methods. On the other hand, in a variety of circumstances, it is reasonable to expect to solve the subproblems to a high accuracy by other means. For example, one may have available specialized methods for the proximal subproblems, or interior-point points methods may be available for moderate dimensions $d$ and $m$, or it may be that case that computing an accurate estimate of $\nabla c(x)$ may already be the bottleneck (see e.g. Example~\ref{ex:grey}). In this context, it is interesting to see if the basic prox-linear method can in some sense be ``accelerated'' by using inertial information. 
In this section, we do exactly that.

We propose an algorithm, motivated by the work of Ghadimi-Lan \cite{ghadimi_lan}, that is adaptive  to some natural constants measuring convexity of the composite function. This being said, the reader should keep in mind a downside the proposed scheme: our analysis (for the first time in the paper) requires the domain of $g$ to be bounded. 
Henceforth, define
$$\qquad\qquad M:=\sup_{x,y\in\dom g} \|x-y\|$$
and assume it to be finite.

To motivate the algorithm, let us first consider the additive composite setting \eqref{eqn:add_comp} with $c(\cdot)$ in addition convex. Algorithms in the style of Nesterov's second accelerated method (see \cite{nest_88} or \cite[Algorithm 1]{tseng_f_order}) incorporate steps of the form $v_{k+1}=\prox_{tg}\left(v_k-t\nabla c(y_k)\right)$. That is, one moves from a point $v_k$ in the direction of the negative gradient $-\nabla c(y_k)$ evaluated at a different point $y_k$, followed by a  proximal operation. Equivalently, after completing a square one can write 
$$v_{k+1}:=\argmin_{z}~\left\{ c(y_k)+ \langle \nabla c(y_k),z-v_k\rangle+\frac{1}{2t}\|z-v_k\|^2+g(z)\right\}.$$
This is also the construction used by Ghadimi and Lan \cite[Equation 2.37]{ghadimi_lan} for nonconvex additive composite problems.
The algorithm we consider emulates this operation. There is a slight complication, however, in that the composite structure requires us to incorporate an additional scaling parameter $\alpha$ in the construction. We use the following notation:
\begin{align*}
	F_{\alpha}(z; y, v) \,&:=\, g(z) + \frac{1}{\alpha} \cdot h \big (c(y) + \alpha \nabla c(y) (z-v) \big ) ,\\
	F_{t, \alpha}(z; y, v) &\,:=\, F_{\alpha}(z; y, v) + \frac{1}{2 t} \norm{z-v}^2,\\
	S_{t, \alpha} (y, v) &\,:=\, \argmin_z~ F_{t, \alpha}(z; y, v).
\end{align*}
Observe the equality $S_{t,1}(x,x)=S_t(x)$. In the additive composite setting, the mapping $S_{t, \alpha} (y, v)$ does not depend on $\alpha$ and the definition reduces to
\begin{align*}
	S_{t, \alpha} (y, v) &= \argmin_z \big \{   
	c(y)+\langle \nabla c(y),
	z-v \rangle + \frac{1}{2t} \norm{z-v}^2+ g(z)  \big \}= \prox_{tg}\left(v-t\nabla c(y)\right).
\end{align*}
The scheme we propose is summarized in Algorithm~\ref{alg: constant}.

{\LinesNotNumbered
	\begin{algorithm}[h!]
		\SetKw{Null}{NULL}
		\SetKw{Return}{return}
		\Initialize{Fix two points $x_0, v_0 \in \text{dom} \,
			g$ and a real number $\tilde{\mu} >  \mu $.}
		{\bf Step k:} ($k\geq 1)$ Compute
		\begin{align}
			a_k &= \tfrac{2}{k+1}\\
			y_k &= a_k v_{k-1} + (1-a_k) x_{k-1}\label{line:bef_mu}\\
			x_k &= S_{1/\tilde{\mu} }(y_k) \label{line: x}\\
			v_k &= S_{\tfrac{1}{ \tilde{\mu} a_k },\, a_k}(y_k, v_{k-1})\label{line: v}
		\end{align}
		%
		\caption{Accelerated prox-linear method}
		\label{alg: constant}
\end{algorithm}}

\begin{remark}[Interpolation weights]
	{\rm
		When $L$ and $\beta$ are unknown, one can instead equip Algorithm~\ref{alg: constant} with a backtracking line search. A formal description and the resulting convergence guarantees appear in  Appendix~\ref{sec:app_auxil}. 
		We also note that instead of setting $a_k=\frac{2}{k+1}$, one may use the interpolation weights used in FISTA~\cite{beck}; namely, the sequence $a_k$ may be chosen to satisfy the relation $\tfrac{1-a_k}{a_k^2} = \tfrac{1}{a_{k-1}^2}$, with similar convergence guarantees. 
		%
	}
\end{remark}

\subsection{Convergence guarantees and convexity moduli}
We will see momentarily that convergence guarantees of Algorithm~\ref{alg: constant} are adaptive to convexity (or lack thereof) of the composition $h\circ c$. To simplify notation, henceforth set $$\Phi:=h\circ c.$$ 

\subsubsection*{Weak convexity and convexity of the pair}
It appears that there are two different convexity-like properties of the composite problem that govern convergence of Algorithm~\ref{alg: constant}. The first is weak-convexity. Recall from Lemma~\ref{lem:prox_reg} that $\Phi$ is $\rho$-weakly convex for some $\rho\in [0,\mu]$. Thus there is some $\rho\in [0,\mu]$ such that for any points $x,y\in \R^d$  and $a\in[0,1]$, the approximate secant inequality holds: $$\Phi(ax+(1-a)y)\leq a \Phi(x)+(1-a)\Phi(y)+\rho a(1-a)\|x-y\|^2.$$

Weak convexity is a property of the composite function $h\circ c$ and is not directly related to $h$ nor $c$ individually.
In contrast, the algorithm we consider uses explicitly the composite structure. In particular, it seems that the extent to which the ``linearization'' $z\mapsto h(c(y)+\nabla c(y)(z-y))$ lower bounds $h(c(z))$ should also play a role. 
\begin{defn}[Convexity of the pair]
	{\rm
		A real number $r>0$ is called a {\em convexity constant of the pair $(h,c)$ on a set $U$} 
		if the inequality
		$$h\big(c(y)+\nabla c(y)(z-y)\big)\leq h(c(z))+\frac{r}{2}\|z-y\|^2\qquad \textrm{ holds for all }z,y\in U.$$ 
	}
\end{defn}

Inequalities \eqref{eqn:ineq_min_max} show that the pair $(h,c)$ indeed has a convexity constant $r\in [0,\mu]$ on $\R^d$.
The following relationship between convexity of the pair $(h,c)$ and weak convexity of $\Phi$ will be useful.
\begin{lemma}[Convexity of the pair implies weak convexity of the composition]\label{lem:con_const_comp} {\hfill \\ }
	If $r$ is a convexity constant of $(h,c)$ on a convex set $U$, then $\Phi$ is $r$-weakly convex on $U$.
\end{lemma}
\begin{proof}
	Suppose $r$ is a convexity constant of $(h,c)$ on $U$. Observe that the subdifferential of the convex function $\Phi$ and that of the linearization $h\big(c(y)+\nabla c(y)(\cdot-y)\big)$ coincide at $y=x$.
	Therefore a quick argument shows that for any $x,y\in U$ and $v\in \partial \Phi(y)$ we have 
	\begin{align*}
		\Phi(x)&\geq h(c(y)+\nabla c(y)(x-y))-\frac{r}{2}\|x-y\|^2\geq \Phi(y)+\langle v,x-y\rangle-\frac{r}{2}\|x-y\|^2 .
	\end{align*}
	The rest of the proof follows along the same lines as  \cite[Theorem 3.1]{fill_grap}. We omit the details.
\end{proof}

\begin{remark}{\rm
		The converse of the lemma is false. Consider for example setting $c(x)=(x,x^2)$ and $h(x,z)=x^2-z$. Then the composition $h\circ c$ is identically zero and hence convex. On the other hand,  one can easily check that the pair $(h,c)$ has a nonzero convexity constant. }
\end{remark}

\subsubsection*{Convergence guarantees}
Henceforth, let $\rho$ be a weak convexity constant of $h\circ c$ on $\dom g$ and let $r$ be a convexity constant of $(h,c)$ on $\dom g$. According to Lemma~\ref{lem:con_const_comp}, we can always assume $0\leq \rho\leq r\leq \mu$.
We are now ready to state and prove convergence guarantees of Algorithm~\ref{alg: constant}.

\begin{theorem}[Convergence guarantees]
	Fix a real number $\tilde \mu>\mu$ and  let $x^*$ be any point satisfying $F(x^*)\leq F(x_k)$ for all iterates $x_k$ generated by Algorithm~\ref{alg: constant}. 
	Then the efficiency estimate holds:		
	\begin{align*}
		\min_{j = 1, \ldots, N} \norm{\mathcal{G}_{1/\tilde{\mu}}(y_j)}^2  \le \frac{24\tilde{\mu}^2}{\tilde{\mu}-\mu}\left(
		\frac{ \tilde{\mu} \norm{x^*-v_0}^2}{
			N (N + 1)(2N+1) } +\frac{M^2(r+\frac{\rho}{2}(N+3))}{(N+1)(2N+1)}\right).
	\end{align*}		
	In the case $r=0$, the inequality above holds with the second summand on the right-hand-side replaced by zero (even if $M=\infty$), and moreover the efficiency bound on function values holds:	
	\begin{align*}
		F(x_N)-F(x^*) &\le \frac{2 \tilde{\mu} \norm{x^*-v_0}^2 }{(N+1)^2}.
	\end{align*}
	\label{thm: cnx}
\end{theorem}

Succinctly, setting $\tilde \mu:=2\mu$, Theorem~\ref{thm: cnx} guarantees the bound \[  \min_{j = 1, \ldots, N} \norm{\mathcal{G}_{1/\tilde{\mu}}(y_j)}^2 \le \mathcal{O}\left(\frac{\mu^2\|x^*-v_0\|^2}{N^3}\right)
+ \frac{r}{\mu}\cdot \mathcal{O}\left(\frac{ \mu^2 M^2 }{ N^2 }\right) +\frac{\rho}{\mu}\cdot \mathcal{O}\left(\frac{ \mu^2 M^2 }{ N }
\right).
\]
The fractions $0\leq\frac{\rho}{\mu}\leq\frac{r}{\mu}\leq 1$ balance the three terms, corresponding to different levels of ``convexity''. 

Our proof of Theorem~\ref{thm: cnx} is based on two basic lemmas, as is common for accelerated methods \cite{tseng_f_order}.
\begin{lemma}[Three-point comparison] Consider the point $z := S_{t, \alpha}(y,v)$ for some points $y,v\in \R^d$ and real numbers $t,\alpha >0$. Then  for all  $w\in\R^d$ the inequality holds:
	\begin{align*}
		F_{\alpha}(z; y, v)   \le F_{\alpha}(w; y, v)  + \frac{1}{2 t} \left (  \norm{w-v}^2 -
		\norm{w-z}^2 - \norm{z-v}^2 \right ).
	\end{align*} \label{lem: scnx}
\end{lemma}
\begin{proof} This follows immediately by noting that the function $F_{t,\alpha}(\cdot;y,v)$ 
	is strongly convex with constant $1/t$ and $z$ is its minimizer by definition.
\end{proof}

\begin{lemma}[Telescoping] \label{lem: telescoping} Let $a_k$, $y_k$, $x_k$, and $v_k$ be the iterates generated by Algorithm~\ref{alg: constant}.  
	Then for any point $x\in \R^d$ and any index $k$, the inequality holds: 
	\begin{equation}\label{eqn: bound_tele}
		\begin{aligned}
			F(x_k)\leq a_k F(x)+&(1-a_k)F(x_{k-1})+\frac{\tilde \mu a_k^2}{2}(\|x-v_{k-1}\|^2-\|x-v_k\|^2)\\
			&- \frac{\tilde \mu-\mu}{2}\|y_k-x_k\|^2+\rho a_k\|x-x_{k-1}\|^2+\frac{ra_k^2}{2}\|x-v_{k-1}\|^2.
		\end{aligned} 
	\end{equation}
\end{lemma}
\begin{proof}
	Notice that all the points $x_k$, $y_k$, and $v_k$  lie in $\dom g$. 
	From inequality \eqref{eqn:ineq_min_max}, we have 
	\begin{equation} F(x_k) \le h\big ( c(y_k) + \nabla c(y_k) (x_k-y_k) \big ) + g(x_k) +
		\frac{\mu}{2} \norm{x_k-y_k}^2.  \label{eq: comp1}
	\end{equation}
	Define the point $w_k:=a_k v_k + (1-a_k) x_{k-1}$.
	Applying
	Lemma~\ref{lem: scnx} to $x_k=S_{1/\tilde{\mu},1}(y_k, y_k)$
	with $w=w_k$  yields the inequality
	\begin{equation}\label{eqn:rev_new_eq1}
		\begin{aligned}
			h(c(y_k)+\nabla c(y_k)(x_k-y_k))+g(x_k)&\leq h(c(y_k)
			+\nabla c(y_k)(w_k-y_k))\\
			&+\frac{\tilde\mu}{2}(\|w_k-y_k\|^2-\|w_k-x_k\|^2-\|x_k-y_k\|^2)\\
			&+a_kg(v_k)+(1-a_k)g(x_{k-1}).
		\end{aligned}
	\end{equation}
	Note the equality
	$w_k-y_k=a_k(v_k-v_{k-1})$.
	Applying Lemma~\ref{lem: scnx} again with $v_k=S_{\frac{1}{\tilde{\mu} a_k}, a_k}(y_{k}, v_{k-1})$ and $w=x$ yields
	\begin{equation}\label{eqn:inter_mid}
		\begin{aligned}
			h \big ( c(y_k) + a_k\nabla c(y_k)(v_k - v_{k-1}) \big ) +
			&a_kg(v_k) \le h \big ( c(y_k) +
			a_k\nabla c(y_k)(x-v_{k-1})  \big )  + a_kg(x) \\
			& \quad + \frac{ \tilde{\mu} a_k^2}{2} \left (  \norm{x-v_{k-1}}^2 -
			\norm{x-v_k}^2 - \norm{v_k-v_{k-1}}^2 \right ).
		\end{aligned}
	\end{equation}
	Define  the point $\hat x:=a_kx+(1-a_k)x_{k-1}$.
	Taking into account $a_k(x-v_{k-1})=\hat x-y_k$, we conclude
	\begin{equation}\label{eqn:fin_step}
		\begin{aligned}
			h(c(y_k)+\nabla c(y_k)(\hat x-y_{k}))&\leq  (h\circ c)(\hat x) +\frac{r}{2}\|\hat x-y_k\|^2\\
			&\leq a_kh(c(x))+(1-a_k)h(c(x_{k-1}))\\
			&+\rho a_k(1-a_k)\|x-x_{k-1}\|^2+\frac{ra_k^2}{2}\|x-v_{k-1}\|^2.
		\end{aligned}
	\end{equation}
	Thus combining inequalities \eqref{eq: comp1}, \eqref{eqn:rev_new_eq1}, \eqref{eqn:inter_mid}, and \eqref{eqn:fin_step}, and upper bounding $1-a_k\leq 1$ and $-\|w_k-x_k\|^2\leq 0$, we obtain
	\begin{align*}
		F(x_k)\leq a_k F(x)+(1-a_k)&F(x_{k-1})+\frac{\tilde \mu a_k^2}{2}(\|x-v_{k-1}\|^2-\|x-v_k\|^2)\\
		&- \frac{\tilde \mu-\mu}{2}\|y_k-x_k\|^2+\rho a_k\|x-x_{k-1}\|^2+\frac{ra_k^2}{2}\|x-v_{k-1}\|^2.
	\end{align*}
	The proof is complete.
\end{proof}

The proof of Theorem~\ref{thm: cnx} now quickly follows.
\begin{proof}[Proof of Theorem~\ref{thm: cnx}]
	Set $x=x^*$ in inequality \eqref{eqn: bound_tele}. Rewriting \eqref{eqn: bound_tele} by subtracting $F(x^*)$ from both sides, we obtain 
	\begin{align}
		\frac{F(x_k)-F(x^*)}{a_k^2} + \frac{\tilde{\mu}}{2} \norm{x^*-v_k}^2 &\le
		\frac{1-a_k}{a_k^2} \big ( F(x_{k-1})-F(x^*) \big ) + \frac{\tilde{\mu} }{2}
		\norm{x^*-v_{k-1}}^2 \nonumber \\
		& \quad + \frac{ \rho M^2}{a_k}+\frac{rM^2 }{2}- \frac{\tilde{\mu}-\mu}{2 a_k^2}
		\norm{x_k-y_k}^2. \label{eq: something}
	\end{align}
	Using the inequality $\frac{1-a_k}{a_k^2} \le \frac{1}{a_{k-1}^2}$ and
	recursively applying the inequality above $N$ times, we get 
	\begin{align}
		\frac{F(x_N)-F(x^*)}{a_N^2} + \frac{\tilde{\mu}}{2} \|x^*-v_N&\|^2 \le
		\frac{1-a_1}{a_1^2} \big ( F(x_0)-F(x^*) \big ) + \frac{\tilde{\mu} }{2}
		\norm{x^*-v_0}^2
		\nonumber \\
		& + \rho M^2  \left (
		\sum_{j=1}^N \frac{1}{a_j} \right ) +\frac{NrM^2}{2}-
		\frac{\tilde{\mu}-\mu}{2} \sum_{j=1}^N
		\frac{\norm{x_j-y_j}^2}{a_j^2}. \label{eq: iteration}
	\end{align}
	Noting $F(x_N)-F(x^*) > 0$ and $a_1 = 1$, we obtain
	\begin{align}
		\frac{\tilde{\mu}-\mu}{2} \sum_{j=1}^N \frac{\norm{x_j-y_j}^2}{a_j^2} &\le
		\frac{\tilde{\mu}}{2} \norm{x^*-v_0}^2 + \rho M^2   \left (
		\sum_{j=1}^N \frac{1}{a_j} \right ) +\frac{NrM^2}{2}\label{eq: something_1}
	\end{align}
	and hence
	\begin{align*}
		\frac{\tilde{\mu}-\mu}{2} \left ( \sum_{j=1}^N \frac{1}{a_j^2} \right )
		\min_{j=1, \hdots, N}  \norm{x_j-y_j}^2 
		&\le \frac{\tilde{\mu}}{2} \norm{x^*-v_0}^2 + \rho M^2  \left (
		\sum_{j=1}^N \frac{1}{a_j} \right ) +\frac{NrM^2}{2}.
	\end{align*}
	Using the definition  $a_k = \frac{2}{k+1}$, we conclude
	\[\sum_{j=1}^N \frac{1}{a_j^2} = \frac{1}{4} \sum_{j=1}^N (j+1)^2 \ge \frac{1}{4}
	\sum_{j=1}^N j^2 = \frac{N (N+1) (2N+1)}{24}\]
	and 
	\[ \sum_{j=1}^N \frac{1}{a_j} = \sum_{j=1}^N \frac{j+1}{2} =
	\frac{N(N+3)}{4}.\]
	With these bounds, we finally deduce
	\begin{align*}
		\min_{j = 1, \hdots N} ~ \norm{x_j-y_j}^2 \le \frac{24}{\tilde{\mu}-\mu}\left(
		\frac{ \tilde{\mu} \norm{x^*-v_0}^2}{
			N (N + 1)(2N+1) } +\frac{M^2(r+\frac{\rho}{2}(N+3)}{(N+1)(2N+1)}\right),
	\end{align*}
	thereby establishing the first claimed efficiency estimate in Theorem~\ref{thm: cnx}. 
	
	Finally suppose $r=0$, and hence we can assume $\rho=0$ by Lemma~\ref{lem:con_const_comp}. 
	Inequality \eqref{eq: iteration} then becomes
	\begin{align*}
		\frac{F(x_N)-F(x^*)}{a_N^2} &+ \frac{\tilde{\mu}}{2} \norm{x^*-v_N}^2 
		\leq\frac{\tilde{\mu} }{2} \norm{x^*-v_0}^2  -
		\frac{\tilde{\mu}-\mu}{2} \sum_{j=1}^N \frac{\norm{x_j-y_j}^2}{a_j^2}
		\label{eq: 1}.
	\end{align*}
	Dropping terms, we deduce
	$\frac{F(x_N)-F(x^*)}{a_N^2}\leq \frac{\tilde{\mu} }{2} \norm{x^*-v_0}^2,$
	and the claimed efficiency estimate follows.
\end{proof}


\subsection{Inexact computation}

Completely analogously, we can consider an inexact accelerated prox-linear method based on approximately solving the duals of the prox-linear subproblems (Algorithm~\ref{alg: constant_inexact}). 

{\LinesNotNumbered
	\begin{algorithm}[h!]
		\SetKw{Null}{NULL}
		\SetKw{Return}{return}
		\Initialize{Fix two points $x_0, v_0 \in \text{dom} \,
			g$ and a real number $\tilde{\mu} >  \mu $.}
		{\bf Step k:} ($k\geq 1)$ Compute
		\begin{align*}
			a_k &= \tfrac{2}{k+1}\\
			y_k &= a_k v_{k-1} + (1-a_k) x_{k-1}
		\end{align*}
		- Find $(x_{k},\zeta_{k})$ such that $\|\zeta_{k}\|\leq \varepsilon_{k}$ and $x_k$ is the minimizer of the function 
		\begin{equation}
			z\mapsto g(z)+h\Big(\zeta_k+c(y_k)+\nabla c(y_k)(z -y_k)\Big)+\frac{\tilde\mu}{2}\|z-y_k\|^2.
		\end{equation}
		
		- Find $(v_{k},\xi_{k})$ such that $\|\xi_{k}\|\leq \delta_{k}$ and $v_k$ is the minimizer of the function 
		\begin{equation}
			v\mapsto g(v)+\frac{1}{a_k}h\Big(\xi_k+c(y_k)+a_k\nabla c(y_k)(v -v_{k-1})\Big)+\frac{\tilde{\mu}a_k}{2}\|v-v_{k-1}\|^2.
		\end{equation}			
		%
		\caption{Inexact accelerated prox-linear method: near-stationarity}
		\label{alg: constant_inexact}
\end{algorithm}}

\begin{theorem}[Convergence of inexact accelerated prox-linear method:  near- stationarity] {\hfill \\ }
	Fix a real number $\tilde \mu\geq \mu$ and let $x^*$ be any point satisfying $F(x^*)\leq F(x_k)$ for iterates $x_k$ generated by Algorithm~\ref{alg: constant_inexact}. 	
	Then for any $N \ge 1$, the iterates $x_k$ satisfy the inequality:
	\begin{align*}
		\min_{i=1,\ldots,N} \|\mathcal{G}_{1/\tilde\mu}(y_j)\|^2\leq 
		&
		\frac{48\tilde\mu^2}{\tilde{\mu}-\mu}\left( \frac{\|x^*-v_0\|^2}{N(N+1)(2N+1)}
		+\frac{M^2(r+\frac{\rho}{2}(N+3))}{(N+1)(2N+1)} + \frac{4L\sum_{j=1}^N\tfrac{2\varepsilon_j+\delta_j}{a_j^2}}{N(N+1)(2N+1)}\right).
	\end{align*}
	Moreover, in the case $r = 0$, the inequality above holds with the second summand on the right-hand-side replaced by zero (even if $M=\infty$) and the following complexity bound on function values holds:
	$$F(x_N)-F(x^*)\leq \frac{2\tilde\mu \|v_0-x^*\|^2+8L\sum^N_{j=1} \frac{\varepsilon_j+\delta_j}{a_j^2}}{(N+1)^2}.$$
	\label{thm: dual_apla}
\end{theorem}

The proof appears in Appendix~\ref{sec:proofs_append}.
Thus to preserve the rate in $N$ of the exact accelerated prox-linear method in Theorem~\ref{thm: cnx}, it suffices to require the sequences $\frac{\varepsilon_j}{a_j^2},\frac{\delta_j}{a_j^2}$ to be summable. Hence we can set $\varepsilon_j,\delta_j\sim \frac{1}{j^{3+q}}$ for some $q>0$.

Similarly, we can consider an inexact version of the accelerated prox-linear method based on approximately solving the primal problems in function value. The scheme is recorded in Algorithm~\ref{alg: constant_inex}. 

{\LinesNotNumbered
	\begin{algorithm}[h!]
		\SetKw{Null}{NULL}
		\SetKw{Return}{return}
		\Initialize{Fix two points $x_0, v_0 \in \text{dom} \,
			g$, a real number $\tilde{\mu} >  L\beta $, and two sequences $\varepsilon_i,\delta_i\geq 0$ for $i=1,2,\ldots,\infty$.}
		{\bf Step k:} ($k\geq 1)$ Compute
		\begin{align}
			&a_k = \tfrac{2}{k+1}\\
			&y_k = a_k v_{k-1} + (1-a_k) x_{k-1}\\
			&\textrm{Set }x_k \textrm{ to be a }\varepsilon_k\textrm{-approximate minimizer of } F_{1/\tilde{\mu}}(\cdot;y_k)\\
			&\textrm{Set }v_k \textrm{ to be a }\delta_k\textrm{-approximate minimizer of } F_{\frac{1}{\tilde{\mu}a_k},a_k}(\cdot;y_k,v_{k-1})
		\end{align}
		%
		\caption{Accelerated prox-linear method: near-optimality}
		\label{alg: constant_inex}
\end{algorithm}}

Theorem~\ref{thm: funct_apla} presents convergence guarantees of Algorithm~\ref{alg: constant_inex}. 
The statement of Theorem~\ref{thm: funct_apla} is much more cumbersome than the analogous Theorem~\ref{thm: dual_apla}. The only take-away message for the reader is that to preserve the rate of the exact accelerated prox-linear method in Theorem~\ref{thm: cnx} in terms of $N$, it sufficies for the sequences $\{\sqrt{i \delta_i}\}$, $\{i \delta_i\}$, and $\{i^2\varepsilon_i\}$ to be summable. Thus it suffices to take $\varepsilon_i, \delta_i\sim \frac{1}{i^{3+q}}$ for some $q>0$.

The proof of Theorem~\ref{thm: funct_apla} appears in Appendix~\ref{sec:proofs_append}. Analysis of inexact accelerated methods of this type for additive convex composite problems has appeared in a variety of papers, including \cite{inex_comput,villa,catalyst}.  In particular, our proof shares many features with that of  \cite{inex_comput}, relying on approximate subdifferentials and the recurrence relation  \cite[Lemma 1]{inex_comput}.


\begin{theorem}[Convergence of the accelerated prox-linear algorithm: near-optimality]	 Fix a real number $\tilde \mu> \mu$, and let $x^*$ be any point satisfying $F(x^*)\leq F(x_k)$ for iterates $x_k$ generated by Algorithm~\ref{alg: constant_inex}. 
	Then  the iterates $x_k$  satisfy the inequality:
	\begin{align*}
		\min_{i=1,\ldots,N} \|\mathcal{G}_{1/\tilde{\mu}}(y_i)\|^2\leq &\frac{96\tilde{\mu}^2}{\tilde\mu-\mu}\Big( \frac{\tilde\mu\|x^*-v_0\|^2}{2N(N+1)(2N+1)}+
		+\frac{M^2(r+\frac{\rho}{2}(N+3))}{2(N+1)(2N+1)}\\
		&+\frac{\sum_{i=1}^N(\tfrac{\delta_ia_i+3\varepsilon_i}{a_i^2})
			+A_N\sqrt{2\tilde\mu}\sum_{i=1}^N \sqrt{\tfrac{\delta_i}{a_i}}}{N(N+1)(2N+1)}\Big)
	\end{align*} 
	with
	\begin{align*}
		A_N:=&\sqrt{\tfrac{2}{\tilde\mu}}\sum_{i=1}^N\sqrt{\tfrac{\delta_i}{a_i}}
		&+\left(\|x^*-v_0\|^2+\tfrac{ M^2N(r+\frac{\rho}{2}(N+3))}{\tilde\mu}+\tfrac{2}{\tilde\mu}\sum^N_{i=1}\tfrac{\delta_ia_i+2\varepsilon_i}{a_i^2}+\tfrac{2}{\tilde\mu}\left(\sum_{i=1}^N \sqrt{\tfrac{\delta_i}{a_i}}\right)^2\right)^{1/2}.
	\end{align*}
	%
	Moreover, in the case $r = 0$, the inequality above holds with the second summand on the right-hand-side replaced by zero (even if $M=\infty$), and the following complexity bound on function values holds:
	\[ F(x_N)-F(x^*) \le \frac{ 2 \tilde{\mu} \norm{x^*-v_0}^2 + 4  \sum_{i=1}^N \frac{\delta_ia_i+2\varepsilon_i}{a_i^2}+4A_N\sqrt{2\tilde{\mu}}\sum_{i=1}^N\sqrt{\frac{\delta_i}{a_i}}}{(N+1)^2}.\] 
	\label{thm: funct_apla}
\end{theorem}

Note that with the choices $\varepsilon_i$, $\delta_i\sim \frac{1}{i^{3+q}}$, the quantity $A_N$ remains bounded. Consequently, in the setting $r=0$, the functional error $F(x_N)-F(x^*)$ is on the order of $\mathcal{O}(1/N^2)$.

\section*{Acknowledgements}
We thank the two anonymous referees for their meticulous reading of the manuscript. Their comments and suggestions greatly improved the quality and readability of the paper.
We also thank Damek Davis and Zaid Harchaoui for their insightful comments on an early draft of the paper.

\bibliographystyle{plain}
\bibliography{bibliography}

\appendix

\section{Proofs of Lemmas~\ref{lem:expl_grad_comp}, \ref{lem:norm_comp} and Theorems \ref{thm: dual_apla}, \ref{thm: funct_apla}}\label{sec:proofs_append}
In this section, we prove Lemmas~\ref{lem:expl_grad_comp}, \ref{lem:norm_comp} and Theorems \ref{thm: dual_apla}, \ref{thm: funct_apla} in order.

\begin{proof}[Proof of Lemma~\ref{lem:expl_grad_comp}]
Observe for any $t>0$ and any proper, closed, convex function $f$, we have
\begin{equation} \label{eqn:main_prop_instup_lem} 
\prox_{(tf)^{\star}}(w)= \argmin_{z}~\{tf^{\star}(z/t)+\tfrac{1}{2}\|z-w\|^2\}=t\cdot \prox_{f^{\star}/t}(w/t),
\end{equation}
where the first equation follows from the definition of the proximal map and from  \cite[Theorem 16.1]{rock}. From \cite[Theorem 31.5]{rock}, we obtain $\prox_{th^{\star}}(w)=w-\prox_{(th^{\star})^{\star}}(w)$, while an application of \eqref{eqn:main_prop_instup_lem}  with $f=h^{\star}$ then directly implies \eqref{dual_gradient_formula}.

The fact that the gradient map $\nabla \Big(G^{\star}\circ A^* - \ip{b, \cdot}\Big)$ is Lipschitz with constant $t\|\nabla c(x)\|^2_{\textrm{op}}$ follows directly from $\nabla G^{\star}$ being $t$-Lipschitz continuous. The chain rule, in turn, yields 
$$\nabla \Big(G^{\star}\circ A^* - \ip{b, \cdot}\Big)(w)=A\nabla G^{\star}(A^*w)-b.$$
Thus we must analyze the expression
$\nabla G^{\star}(z)=\nabla (g+\tfrac{1}{2t}\|\cdot-x\|^2)^{\star}(z)$. Notice that the conjugate of  $\tfrac{1}{2t}\|\cdot-x\|^2$ is the function $\frac{t}{2}\|\cdot\|^2+\langle \cdot,x \rangle$. Hence, using \cite[Theorem 16.4]{rock}  we deduce 
$$(g+\tfrac{1}{2t}\|\cdot-x\|^2)^{\star}(z)=\inf_{y}~ \{g^{\star}(y)+\tfrac{t}{2}\|z-y\|^2+\langle z-y,x \rangle\}=(g^{\star})_{1/t}(z+x/t)-\tfrac{1}{2t}\|x\|^2,$$
where the last equation follows from completing the square.
We thus conclude
\begin{equation*} 
\nabla G^{\star}(z)=\nabla (g^{\star})_{1/t}(z+x/t)=t\cdot\prox_{(g^{\star}/t)^{\star}}(z+x/t)=\prox_{tg}(x+tz),
\end{equation*}
where the second equality follows from Lemma~\ref{lem:lip_cont} and 
the third from \eqref{eqn:main_prop_instup_lem}. The expressions \eqref{eqn:dual_subdiff} and \eqref{eqn:grad_express} follow.
\end{proof}

\begin{proof}[Proof of Lemma~\ref{lem:norm_comp}]
	Observe $$\|h(y)-h(z)\|\leq \frac{1}{m}\sum_{i=1}^m |h_i(y_i)-h_i(z_i)|\leq \frac{{L}}{m}\sum_{i=1}^m \|y-z\|_1\leq  \frac{{L}}{\sqrt{m}}\|y-z\|,$$
	where the last equality follows from the $l_p$-norm comparison 
	$\|\cdot\|_1\leq \sqrt{m}\|\cdot\|_2$. This proves $\lip(h)\leq {L}/\sqrt{m}$. Next for any point $x$ observe 
	$$\|\nabla c(x)\|_{\textrm{op}}=\max_{v:\|v\|=1} \|\nabla c(x)v\|\leq \sqrt{\sum_{i=1}^m \|\nabla c_i(x)\|^2}\leq \sqrt{m}\max_{i=1,\ldots,m} \|\nabla c_i(x)\|$$
	By an analogous argument, we have 
	$$\|\nabla c(x)-\nabla c(z)\|_{\textrm{op}}\leq \sqrt{\sum_{i=1}^m \|\nabla c_i(x)-\nabla c_i(z)\|^2}\leq \beta\sqrt{m}\|x-z\|,$$
	and hence $\lip(\nabla c)\leq {\beta}\sqrt{m}$. Finally, suppose that each $h_i$ is $C^1$-smooth with ${L_h}$-Lipschitz gradient $\nabla h_i$.
	Observe then 
	$$\|\nabla h(y)-\nabla h(z)\|=\frac{1}{m}\sqrt{\sum^m_{i=1} |h_i'(y_i)-h_i'(z_i)|^2}\leq \frac{{L_h}}{m}\|y-z\|^2.$$ The result follows.
\end{proof}

\begin{proof}[Proof of Theorem~\ref{thm: dual_apla}]
The proof is a modification of the proof Theorem~\ref{thm: cnx}; as such, we skip some details.
For any point $w$, we successively deduce
\begin{align*}
F(x_k) &\le h \big ( \zeta_{k} + c(y_k) + \nabla c(y_k) (x_k -y_k) \big
  ) + g(x_k) + \frac{\mu}{2} \norm{x_k-y_k}^2  + L \cdot \varepsilon_k
  \\
&\le \Big(h \big ( \zeta_{k} + c(y_k) + \nabla c(y_k) (x_k -y_k) \big
  ) + g(x_k) + \frac{\tilde{\mu}}{2} \norm{x_k-y_k}^2\Big) -
  \frac{\tilde{\mu}-\mu}{2} \norm{x_k-y_k}^2 + L \cdot \varepsilon_k
  \\
&\le h \big ( \zeta_{k} + c(y_k) + \nabla c(y_k) (w -y_k) \big
  ) + g(w)\\
&\qquad  + \frac{\tilde{\mu}}{2} \left ( \norm{w-y_k}^2 -
  \norm{w-x_k}^2 \right ) -
  \frac{\tilde{\mu}-\mu}{2} \norm{x_k-y_k}^2 + L \cdot \varepsilon_k \\
& \le h \big ( c(y_k) + \nabla c(y_k) (w -y_k) \big
  ) + g(w) \\
&\qquad + \frac{\tilde{\mu}}{2} \left ( \norm{w-y_k}^2 - \norm{w-x_k}^2 \right )-
  \frac{\tilde{\mu}-\mu}{2} \norm{x_k-y_k}^2 + 2L \cdot \varepsilon_k.
\end{align*}
Setting $w:=a_kv_k+(1-a_k)x_{k-1}$ and noting the equality $w-y_k=a_k(v_k-v_{k-1})$ then yields
\begin{align*}
F(x_k)\leq& h(c(y_k)+a_k\nabla c(y_k)(v_k-v_{k-1}))+a_kg(v_k)+(1-a_k)g(x_{k-1})\\
&+\frac{\tilde{\mu}}{2} \left ( \norm{a_k(v_k-v_{k-1})}^2 - \norm{w-x_k}^2 \right )-
  \frac{\tilde{\mu}-\mu}{2} \norm{x_k-y_k}^2 + 2L \cdot \varepsilon_k.
\end{align*}
Upper bounding $-\|w-x_k\|^2$ by zero and using Lipschitz continuity of $h$ we obtain for any point $x$ the inequalities 
\begin{align*}
F(x_k)&\leq  a_k\Big(\frac{1}{a_k}h(\xi_k+c(y_k)+a_k\nabla c(y_k)(v_k-v_{k-1}))+g(v_k)\Big)+(1-a_k)g(x_{k-1})\\
&+\frac{\tilde{\mu} a_k^2}{2}  \norm{v_k-v_{k-1}}^2 -
  \frac{\tilde{\mu}-\mu}{2} \norm{x_k-y_k}^2 + L\cdot\delta_k+2L \cdot \varepsilon_k.\\
&\leq  a_k\Big(\frac{1}{a_k}h(\xi_k+c(y_k)+a_k\nabla c(y_k)(x-v_{k-1}))+g(x)+\frac{\tilde\mu a_k}{2}(\|x-v_{k-1}\|^2
-\|v_k-v_{k-1}\|^2\\
&-\|v_k-x\|^2)\Big)+(1-a_k)g(x_{k-1})+\frac{\tilde{\mu} a_k^2}{2}  \norm{v_k-v_{k-1}}^2 -
  \frac{\tilde{\mu}-\mu}{2} \norm{x_k-y_k}^2+L\delta_k+2L\varepsilon_k.\\
  &\leq  h(c(y_k)+a_k\nabla c(y_k)(x-v_{k-1}))+a_kg(x)+\frac{\tilde\mu a_k^2}{2}(\|x-v_{k-1}\|^2
-\|v_k-x\|^2)\\
&+(1-a_k)g(x_{k-1}) -
  \frac{\tilde{\mu}-\mu}{2} \norm{x_k-y_k}^2+2L\delta_k+2L\varepsilon_k.
\end{align*}
Define $\hat x:=a_kx+(1-a_k)x_{k-1}$ and note $a_k(x-v_{k-1})=\hat x-y_k$. The same argument as that of \eqref{eqn:fin_step} yields
\begin{align*}
h(c(y_k)+\nabla c(y_k)(\hat x-y_{k}))\leq &
a_k h(c(x))+(1-a_k)h(c(x_{k-1}))+\\
&\rho a_k(1-a_k)\|x-x_{k-1}\|^2+\frac{r a_k^2}{2}\|x-v_{k-1}\|^2.
\end{align*}
Hence upper bounding $1-a_k\leq 1$  we deduce
\begin{align*}
F(x_k)\leq &a_k F(x)+(1-a_k)F(x_{k-1})+\frac{\tilde \mu a_k^2}{2}(\|x-v_{k-1}\|^2-\|x-v_k\|^2)\\
&- \frac{\tilde \mu-\mu}{2}\|y_k-x_k\|^2+\rho a_k\|x-x_{k-1}\|^2+\frac{ra_k^2}{2}\|x-v_{k-1}\|^2+ 2L(\delta_k+\varepsilon_k).
\end{align*} 
This expression is identical to that of \eqref{eqn: bound_tele} except for the error term $2L(\delta_k+\varepsilon_k)$. The same argument as in the proof of Theorem~\ref{thm: cnx} then shows
\begin{align*}
\frac{F(x_N)-F(x^*)}{a_N^2} + \frac{\tilde{\mu}}{2} \|x^*-v_N\|^2 \le
 & \frac{\tilde{\mu} }{2}
\norm{x^*-v_0}^2
 + \rho M^2  \left (
\sum_{j=1}^N \frac{1}{a_j} \right ) \\
&+\frac{NrM^2}{2}-
\frac{\tilde{\mu}-\mu}{2} \sum_{j=1}^N
\frac{\norm{x_j-y_j}^2}{a_j^2}+2L\sum_{j=1}^N\frac{\varepsilon_j+\delta_j}{a_j^2}. 
\end{align*}
Hence appealing to Lemma~\ref{lem:pass_tostat}, we deduce 
\begin{align*}
\sum_{j=1}^N &\frac{\|\mathcal{G}_{1/\tilde\mu}(y_j)\|^2}{a_j^2}\leq 8 L\tilde{\mu}\sum_{j=1}^N  \frac{\varepsilon_j}{a_j^2} + 2 \sum_{j=1}^N  \frac{\norm{\tilde{\mu}(x_j-y_j)}^2}{a_j^2}\\
&\leq 8 L\tilde{\mu}\sum_{j=1}^N  \frac{\varepsilon_j}{a_j^2}+\frac{4\tilde{\mu}^2}{\tilde\mu-\mu}\left(\frac{\tilde\mu}{2}\|x^*-v_0\|^2+\frac{NM^2(r+\frac{\rho}{2}(N+3))}{2}+2L\sum_{j=1}^N\frac{\varepsilon_j+\delta_j}{a_j^2}\right).
\end{align*}
Therefore
\begin{align*}
\min_{i=1,\ldots,N} &\|\mathcal{G}_{1/\tilde\mu}(y_j)\|^2\leq \frac{8\cdot 24 L\tilde{\mu}\sum_{j=1}^N  \frac{\varepsilon_j}{a_j^2}}{N(N+1)(2N+1)}\\
&+
\frac{48\tilde\mu^2}{\tilde{\mu}-\mu}\left( \frac{\|x^*-v_0\|^2}{N(N+1)(2N+1)}
+\frac{M^2(r+\frac{\rho}{2}(N+3))}{(N+1)(2N+1)} + \frac{4L\sum_{j=1}^N\frac{\varepsilon_j+\delta_j}{a_j^2}}{N(N+1)(2N+1)}\right)
\end{align*}
Combining the first and fourth terms and using the inequality $\tilde\mu\geq \mu$ yields the claimed efficiency estimate on $\|\mathcal{G}_{1/\tilde\mu}(y_j)\|^2$.
Finally, the claimed efficiency estimate on the functional error $F(x_N)-F^*$ in the setting $r=0$ follows by the same reasoning as in 
Theorem~\ref{thm: cnx}.
%
\end{proof}

We  next prove Theorem~\ref{thm: funct_apla}. To this end, we will need the following lemma.

\begin{lemma}[Lemma 1 in \cite{inex_comput}]
Suppose the following recurrence relation is satisfied
\[ d_k^2 \le d_0^2 + c_k + \sum_{i=1}^k
\beta_i d_i \]
for some sequences $d_i,\beta_i \ge 0$ and an increasing sequence $c_i\geq 0$. Then the inequality holds:
\[ d_k \le A_k:= \frac{1}{2} \sum_{i=1}^k \beta_i + \left (
  d_0^2 + c_k + \left ( \frac{1}{2} \sum_{i=1}^k \beta_i   \right )^{2}
\right )^{1/2}. \]
Moreover since the terms on the right-hand side increase in
$k$, we also conclude for any $k \le N$ the inequality
 $d_k \le A_N$.
\label{lem: recurrence} 
\end{lemma}
The \emph{$\varepsilon$-subdifferential} of a function $f\colon\R^d\to\overline\R$ at a point  $\bar x$ is the set
$$\partial_\varepsilon f(\bar x):=\{v\in \R^d: f(x)-f(\bar x)  \ge \ip{v, x-\bar x} - \varepsilon \quad \text{ for all } x\in\R^d\}.$$
In particular, notice that $\bar x$ is an $\varepsilon$-approximate minimizer of $f$ if and only if the inclusion $0\in \partial_\varepsilon f(\bar x)$ holds.
For the purpose of analysis, it is useful to decompose the function $F_{t,
	\alpha}(z,y,v)$ into a sum
$$F_{t,
  \alpha}(z;y,v)=F_{\alpha}(z;y,v)+
\frac{1}{2t}\|z-v\|^2$$
The sum rule
for $\varepsilon$-subdifferentials \cite[Theorem 2.1]{eps_subdif_calc} guarantees $$\partial_\varepsilon
F_{t, \alpha}(\cdot; y, v) \subseteq \partial_\varepsilon F_{\alpha}(\cdot;y,v)
+ \partial_\varepsilon \left(\frac{1}{2t}\|\cdot -v\|^2\right).$$ 

\begin{lemma}
The $\varepsilon$-subdifferential $\partial_\varepsilon \left(\frac{1}{2t}\|\cdot -v\|^2\right)$ at a point $\bar z$ is the set
\begin{align*} 
\left \{  t^{-1}  (z-v+\gamma) \, : \,
\frac{1}{2t} \norm{\gamma}^2 \le \varepsilon \right \}.
\end{align*} 
\label{lem: eps_sub}
\end{lemma}

\begin{proof} This follows  by completing the square in the definition of the $\varepsilon$-subdifferential. 
\end{proof}

In particular, suppose that $z^+$ is an $\varepsilon$-approximate minimizer of $F_{t, \alpha}(\cdot; y, v)$.
Then Lemma~\ref{lem: eps_sub} shows that there is a vector $\gamma$ satisfying $\norm{\gamma}^2
\le 2t \varepsilon$ and 
\begin{equation}
 t^{-1} (v-z^+-\gamma) \in \partial_\varepsilon F_{\alpha}(z^+;y,v) \label{eq: approxsub}.
\end{equation}
We are now ready to prove Theorem~\ref{thm: funct_apla}.

\begin{proof}[Proof of Theorem~\ref{thm: funct_apla}] 
	Let $x_k$, $y_k$, and $v_k$ be the iterates generated by Algorithm~\ref{alg: constant_inex}.
	We imitate the proof of Theorem~\ref{thm: cnx}, while taking into account inexactness. 
  First, inequality \eqref{eq: comp1} is still valid:
   $$F(x_k) \le F(x_k;y_k) +
   \tfrac{\mu}{2} \norm{x_k-y_k}^2.$$
    Since $x_k$ is an $\varepsilon_k$-approximate minimizer of the function $F(\cdot;y_k)=F_{1/\tilde{\mu},1}(\cdot; y_k,y_k)$, from \eqref{eq: approxsub}, we obtain a vector
    $\gamma_k$ satisfying $\norm{\gamma_k}^2 \le 2
    \varepsilon_k \tilde{\mu}^{-1}$ and $\tilde{\mu} ( y_k-x_k- \gamma_k) \in \partial_{\varepsilon_k} F(x_k;
    y_k)$. Consequently for
  all points $w$ we deduce the inequality
\begin{align}
F(x_k) 
&\le F(w;y_k) +
  \tfrac{\mu}{2} \norm{x_k-y_k}^2 
+ \ip{\tilde{\mu}(y_k-x_k-\gamma_k), x_k-w}  + \varepsilon_k 
\label{eq: aos_1}.
\end{align}
Set $w_k := a_k v_k + (1-a_k)
x_{k-1}$ and define  $c_k:=x_k-w_k$. Taking into account $w_k-y_k=a_k(v_k-v_{k-1})$, the previous inequality with $w=w_k$ becomes   
\begin{align}
F(x_k) &\le h(c(y_k) + a_k\nabla c(y_k) (v_k-v_{k-1}) ) + a_k
         g(v_k) + (1-a_k) g(x_{k-1}) + \tfrac{\mu}{2}
         \norm{x_k-y_k}^2 \nonumber\\
& + \tilde{\mu} \ip{y_k-x_k, c_k} - \tilde{\mu}\ip{\gamma_k,
 c_k} + \varepsilon_k. \label{eq: aos_2}
\end{align}
By completing the square, one can check
\begin{align*}
\tilde{\mu} \ip{y_k-x_k, c_k} &=\tfrac{\tilde{\mu}}{2}
\big ( \norm{a_kv_k-a_kv_{k-1}}^2 - \norm{x_k-y_k}^2 -
\norm{c_k}^2 \big ).
\end{align*}
Observe in addition
\[ -\tilde{\mu}\ip{\gamma_k, c_k} - \tfrac{\tilde{\mu}}{2} \norm{c_k}^2 =
-\tfrac{\tilde{\mu}}{2} \norm{ \gamma_k + c_k}^2 +
\tfrac{\tilde{\mu}}{2} \norm{\gamma_k}^2.\]
By combining the two equalities with \eqref{eq: aos_2} and dropping the
term $\tfrac{\tilde{\mu}}{2}
\norm{\gamma_k+c_k}^2$, we deduce
\begin{equation}\label{eqn:main_ineq_inexact}
 \begin{aligned}
F(x_k) &\le h(c(y_k) + a_k \nabla c(y_k)(v_k-v_{k-1}) ) + a_k g(v_k) +
         (1-a_k) g(x_{k-1})\\
& + \tfrac{\tilde{\mu} a_k^2}{2} \norm{v_k-v_{k-1}}^2 -
  \tfrac{\tilde{\mu}-\mu}{2} \norm{x_k-y_k}^2 + \varepsilon_k +
\tfrac{\tilde{\mu}}{2} \norm{\gamma_k}^2.
\end{aligned}
\end{equation}
Next recall that $v_k$ is a $\delta_k$-approximate minimizer of $F_{(\tilde{\mu} a_k)^{-1}, a_k}(\cdot; y_k, v_{k-1})$.
Using \eqref{eq: approxsub}, we obtain 
a vector $\eta_k$ satisfying $\norm{\eta_k}^2 \le \tfrac{2 \delta_k}{a_k
  \tilde{\mu}}$ and $a_k \tilde{\mu} (v_{k-1}-v_k-\eta_k)
\in \partial_{\delta_k} F_{a_k}(v_k;y_k,v_{k-1})$. Hence, we
conclude for all the points $x$ the inequality
\begin{equation}\label{eqn:mild_upper_d}
\begin{aligned}
 F_{a_k}(v_k;y_k,v_{k-1})&\leq   \frac{1}{a_k}h(c(y_k) + a_k \nabla c(y_k) (x-v_{k-1})+ g(x)\\
 &  + \tilde{\mu} 
  a_k \ip{v_{k-1}-v_k-\eta_k, v_k-x}  + \delta_k. 
\end{aligned}
\end{equation}
Completing the square, one can verify
 $$\ip{v_{k-1}-v_k, v_k-x}=\frac{1}{2}(\|x-v_{k-1}\|^2-\|x-v_k\|^2-\|v_k-v_{k-1}\|^2).$$
Hence combining this with \eqref{eqn:main_ineq_inexact} and \eqref{eqn:mild_upper_d}, while taking into account the inequalities $\|\gamma_k\|^2\leq 2\varepsilon_k\tilde{\mu}^{-1}$ and 
$\norm{\eta_k}^2 \le \tfrac{2 \delta_k}{a_k
	\tilde{\mu}}$, we deduce
\begin{align*}
F(x_k)\leq &h(c(y_k) + a_k \nabla c(y_k) (x-v_{k-1})+ a_kg(x)+(1-a_k) g(x_{k-1})\\
&+\tfrac{\tilde{\mu} a_k^2}{2} ( \norm{x-v_{k-1}}^2 - \norm{x-v_k}^2
)
+a_k\delta_k
-
\tfrac{\tilde{\mu}-\mu}{2} \norm{x_k-y_k}^2 + 2\varepsilon_k \\
&+a_k^{3/2}\sqrt{2\tilde{\mu}\delta_k} \cdot\|v_k-x\|.
\end{align*}

Following an analogous part of the proof of
Theorem~\ref{thm: cnx}, 
define now the point $\hat x=a_kx+(1-a_k)x_{k-1}$.
Taking into account $a_k(x-v_{k-1})=\hat x-y_k$, we conclude
\begin{equation*}
\begin{aligned}
h(c(y_k)+\nabla c(y_k)(\hat x-y_{k}))&\leq  (h\circ c)(\hat x) +\frac{r}{2}\|\hat x-y_k\|^2\\
&\leq a_kh(c(x))+(1-a_k)h(c(x_{k-1}))\\
&+\rho a_k(1-a_k)\|x-x_{k-1}\|^2+\frac{ra_k^2}{2}\|x-v_{k-1}\|^2.
\end{aligned}
\end{equation*}
Thus we obtain
\begin{align*}
F(x_k)\leq &a_kF(x)+(1-a_k)F(x_{k-1})+\rho a_k\|x-x_{k-1}\|^2+\frac{ra_k^2}{2}\|x-v_{k-1}\|^2\\
&+\tfrac{\tilde{\mu} a_k^2}{2} ( \norm{x-v_{k-1}}^2 - \norm{x-v_k}^2
)
+a_k\delta_k
-
\tfrac{\tilde{\mu}-\mu}{2} \norm{x_k-y_k}^2 + 2\varepsilon_k \\
&+a_k^{3/2}\sqrt{2\tilde{\mu}\delta_k} \cdot\|v_k-x\|.
\end{align*}
As in the proof of Theorem~\ref{thm: cnx}, setting $x=x^*$, we deduce

\begin{align*}
\frac{F(x_{N})-F^*}{a_N^2}+\frac{\tilde \mu}{2}\|x^*-v_N\|^2\leq &\frac{\tilde\mu}{2}\|x^*-v_0\|^2+\rho M^2\sum_{i=1}^N\frac{1}{a_i} +\frac{NrM^2}{2}+\sum^N_{i=1}\frac{\delta_i}{a_i}\\
&-\frac{\tilde\mu-\mu}{2}\sum_{i=1}^N\frac{\|x_i-y_i\|^2}{a_i^2}+2\sum_{i=1}^N\frac{\varepsilon_i}{a_i^2}+\sqrt{2\tilde\mu}\sum_{i=1}^N \|x^*-v_{i}\|\cdot\sqrt{\frac{\delta_i}{a_i}}.
\end{align*}
In particular, we have 
\begin{equation}\label{eqn:key_step-eqn}
\begin{aligned}
\frac{\tilde\mu-\mu}{2}\sum_{i=1}^N\frac{\|x_i-y_i\|^2}{a_i^2}\leq &\frac{\tilde\mu}{2}\|x^*-v_0\|^2+\frac{\rho M^2N(N+3)}{4} +\frac{NrM^2}{2}+\sum^N_{i=1}\frac{\delta_i}{a_i}\\
&+2\sum_{i=1}^N\frac{\varepsilon_i}{a_i^2}+\sqrt{2\tilde\mu}\sum_{i=1}^N \|x^*-v_{i}\|\cdot\sqrt{\frac{\delta_i}{a_i}}.
\end{aligned}
\end{equation}
and
\begin{align*}
\frac{\tilde \mu}{2}\|x^*-v_N\|^2\leq &\frac{\tilde\mu}{2}\|x^*-v_0\|^2+\frac{\rho M^2N(N+3)}{4} +\frac{NrM^2}{2}+\sum^N_{i=1}\frac{\delta_i}{a_i}\\
&+2\sum_{i=1}^N\frac{\varepsilon_i}{a_i^2}+\sqrt{2\tilde\mu}\sum_{i=1}^N \|x^*-v_{i}\|\cdot\sqrt{\frac{\delta_i}{a_i}}.
\end{align*}
Appealing to Lemma~\ref{lem: recurrence} with $d_k=\|x^*-v_k\|$, we conclude
$\|x^*-v_N\|\leq A_N$
for the constant
\begin{align*}
 A_N:=&\sqrt{\frac{2}{\tilde\mu}}\sum_{i=1}^N\sqrt{\frac{\delta_i}{a_i}}+\\
 &+\left(\|x^*-v_0\|^2+\frac{ M^2N(r+\frac{\rho}{2}(N+3))}{\tilde\mu}+\frac{2}{\tilde\mu}\sum^N_{i=1}\frac{\delta_i}{a_i}+\frac{4}{\tilde\mu}\sum_{i=1}^N\frac{\varepsilon_i}{a_i^2}+\frac{2}{\mu}\left(\sum_{i=1}^N \sqrt{\frac{\delta_i}{a_i}}\right)^2\right)^{1/2}.
 \end{align*}
 Finally, combining  inequality \eqref{eqn:key_step-eqn} with Lemma~\ref{lem: funct_grad} we deduce
\begin{align*}
\frac{\tilde{\mu}-\mu}{2}\sum^N_{i=1}\frac{\|\mathcal{G}_{1/\tilde{\mu}}(y_i)\|^2}{a_i^2}&\leq 2\tilde{\mu}(\tilde{\mu}-\mu)\sum_{i=1}^N \frac{\varepsilon_i}{a_i^2}+2\tilde{\mu}^2\Big(\frac{\tilde\mu}{2}\|x^*-v_0\|^2+\frac{\rho M^2N(N+3)}{4} +\frac{NrM^2}{2}+\\
&+\sum^N_{i=1}\frac{\delta_i}{a_i}
+2\sum_{i=1}^N\frac{\varepsilon_i}{a_i^2}+A_N\sqrt{2\tilde\mu}\sum_{i=1}^N \sqrt{\frac{\delta_i}{a_i}}\Big).
\end{align*}
Hence 
\begin{align*}
\min_{i=1,\ldots,N} \|\mathcal{G}_{1/\tilde{\mu}}(y_i)\|^2\leq \frac{96\tilde{\mu}\sum_{i=1}^N \frac{\varepsilon_i}{a_i^2}}{N(N+1)(2N+1)}+ &\frac{96\tilde{\mu}^2}{\tilde\mu-\mu}\Big( \frac{\tilde\mu\|x^*-v_0\|^2}{2N(N+1)(2N+1)}+\frac{M^2(r+\frac{\rho}{2}(N+3))}{2(N+1)(2N+1)}+
\\
&+\frac{\sum_{i=1}^N(\frac{\delta_ia_i+2\varepsilon_i}{a_i^2})
+A_N\sqrt{2\tilde\mu}\sum_{i=1}^N \sqrt{\frac{\delta_i}{a_i}}}{N(N+1)(2N+1)}\Big).
\end{align*}
Combining the first and the fourth terms, the result follows. The efficiency estimate on $F(x_N)-F^*$ in the setting $r=0$ follows by the same argument as in the proof of Theorem~\ref{thm: cnx}.
\end{proof}

\section{Backtracking}\label{sec:app_auxil}
In this section, we present a variant of Algorithm~\ref{alg: constant} where the constants $L$ and $\beta$ are unknown. The scheme is recorded as Algorithm~\ref{alg: acc_prox_linear_back} and  
relies on a backtracking line-search, stated in Algorithm~\ref{alg: backtrack}. 

{\LinesNotNumbered
	\begin{algorithm}[h!]
		\SetKw{Null}{NULL}
		\SetKw{Return}{return}
		\Initialize{A point $y$ and real numbers $\eta, \alpha \in
			(0,1)$ and $t > 0$. }
		\While { {\footnotesize $F(S_{\alpha t}(y)) >  F_t( S_{\alpha t}(y))  $} } {
			$t \leftarrow \eta t$  
		}
		Set $\tilde{\mu}= \tfrac{1}{\alpha t}$ and $x = S_{\alpha t}(y)$\\
		\Return $\tilde{\mu}, t, x$\;
		\caption{Backtracking$(\eta,\alpha,t,y)$}
		\label{alg: backtrack}
\end{algorithm}}

{\LinesNotNumbered
	\begin{algorithm}[h!]
		\SetKw{Null}{NULL}
		\SetKw{Return}{return}
		\Initialize{Fix two points $x_0, v_0 \in \text{dom} \,
			g$ and real numbers $t_0 > 0$ and $\eta, \alpha \in (0,1)$.}
		{\bf Step k:} ($k\geq 1)$ Compute
		\begin{align*}
			a_k &= \tfrac{2}{k+1}\\
			y_k &= a_k v_{k-1} + (1-a_k) x_{k-1}\\
			(\tilde{\mu}_k, t_k, x_k) &= \text{Backtracking}(\eta, \alpha, t_{k-1}, y_k)\\
			v_k &= S_{\tfrac{1}{ \tilde{\mu}_k a_k },\, a_k}(y_k, v_{k-1}) 
		\end{align*}
		\caption{Accelerated prox-linear method with backtracking}
		\label{alg: acc_prox_linear_back}
\end{algorithm}}

The backtracking procedure completes after only logarithmically many iterations.
\begin{lemma}[Termination of backtracking line search] 
	Algorithm~\ref{alg: backtrack} on input $(\eta, \alpha,t,y)$ terminates after at most 
	$ 1+\left \lceil \frac{\log(t \mu)}{\log (\eta^{-1})} \right \rceil $
	evaluations of $S_{\alpha\, \boldsymbol{\cdot}}(y)$. 
\end{lemma}
\begin{proof}
	This follows immediately by observing that the loop in Algorithm~\ref{alg: backtrack}  terminates as soon as $t\leq\mu^{-1}$. 	
\end{proof}

We now establish convergence guarantees of Algorithm~\ref{alg: acc_prox_linear_back}, akin to those of Algorithm~\ref{alg: constant}.
\begin{theorem}[Convergence guarantees with backtracking] Fix real
	numbers $t_0 > 0$ and $\eta, \alpha \in (0, 1)$ and let $x^*$ be any point satisfying $F(x^*) \le
	F(x_k)$ for all iterates $x_k$ generated by Algorithm~\ref{alg:
		acc_prox_linear_back}. Define $\tilde{\mu}_{\max} := \max \{
	(\alpha  t_0)^{-1}, (\alpha \eta)^{-1} \mu \}$ and $\tilde{\mu}_0 :=
	(\alpha 
	t_0)^{-1}$. 
	Then the efficiency estimate holds:
	\[ \min_{j=1, \hdots, N} \norm{\mathcal{G}_{1/\tilde{\mu}_j}(y_j)}^2
	\le \frac{24 \tilde{\mu}_{\max}}{1-\alpha} \left
	( \frac{\tilde{\mu}_0 \norm{x^*-v_0}^2 }{N(N+1)(2N+1)} + \frac{M^2
		\left (r + \tfrac{\rho}{2} (N+3) \right ) }{(N+1)(2N+1)} \right ).\] 
	In the case $r = 0$, the inequality above holds with the second
	summand on the right-hand-side replaced by zero (even if $M =
	\infty$), and moreover the efficiency bound on function values holds:
	\[ F(x_N)-F(x^*) \le \frac{2\tilde{\mu}_{\max} \norm{x^*-v_0}^2 }{(N+1)^2}.\]
\end{theorem}

\begin{proof} We closely follow the proofs of Lemma~\ref{lem: telescoping} and Theorem~\ref{thm: cnx},
	as such, we omit some details. For $k \ge 1$, the stopping
	criteria of the backtracking algorithm guarantees that analogous
	inequalities \eqref{eq: comp1} and \eqref{eqn:rev_new_eq1} hold,
	namely, 
	\begin{equation}
		F(x_k) \le h \big ( c(y_k) + \nabla c(y_k)(x_k-y_k)  \big
		) + g(x_k) + \frac{1}{2t_k} \norm{x_k - y_k}^2 \label{eq: back_1}
	\end{equation}
	and 
	\begin{equation} \label{eq: back_2}
		\begin{aligned}
			h \big ( c(y_k) + \nabla c(y_k) (x_k-y_k) \big ) + g(x_k) &\le  h\big ( c(y_k)
			+ \nabla c(y_k) (w_k-y_k) \big )\\
			& + \frac{\tilde{\mu}_k}{2} \left (
			\norm{w_k-y_k}^2 - \norm{w_k-x_k}^2 - \norm{x_k-y_k}^2 \right )\\
			& + a_k g(v_k) + (1-a_k) g(x_{k-1})
		\end{aligned}
	\end{equation}
	where $w_k := a_k v_k + (1-a_k)x_{k-1}$. By combining \eqref{eq:
		back_1} and \eqref{eq: back_2} together with the definition that $\tilde{\mu}_k =
	(\alpha t_k)^{-1}$, we conclude 
	\begin{equation} \label{eq: back_3}
		\begin{aligned}
			F(x_k) &\le h \big ( c(y_k) + \nabla c(y_k) (w_k-y_k) \big ) + a_k g(v_k) + (1-a_k) g(x_{k-1})\\
			& + \frac{\tilde{\mu}_k}{2} \left ( \norm{w_k-y_k}^2 -
			\norm{w_k-x_k}^2 \right ) +  \frac{\left (1-
				\alpha^{-1} \right )}{2t_k} \norm{x_k-y_k}^2.
		\end{aligned}
	\end{equation}
	We note the equality $w_k -y_k = a_k(v_k-v_{k-1})$. 
	Observe that~\eqref{eqn:inter_mid} holds by replacing $\tfrac{\tilde{\mu}}{2}$ with
	$\tfrac{\tilde{\mu}_k}{2}$; hence, we
	obtain for all points $x$
	\begin{equation}\label{eq: back_4}
		\begin{aligned}
			h \big ( c(y_k) + a_k\nabla c(y_k)(v_k - v_{k-1}) \big ) &+
			a_kg(v_k) \le h \big ( c(y_k) +
			a_k\nabla c(y_k)(x-v_{k-1})  \big )  + a_kg(x) \\
			& \quad + \frac{ \tilde{\mu}_k a_k^2}{2} \left (  \norm{x-v_{k-1}}^2 -
			\norm{x-v_k}^2 - \norm{v_k-v_{k-1}}^2 \right ).
		\end{aligned}
	\end{equation}
	Notice also that \eqref{eqn:fin_step} holds as stated. Combining the inequalities \eqref{eqn:fin_step},
	\eqref{eq: back_3}, and
	\eqref{eq: back_4}, we deduce 
	\begin{equation} \label{eq: back_5}
		\begin{aligned}
			F(x_k) \le a_k F(x) + &(1-a_k)F(x_{k-1}) + \frac{\tilde{\mu}_k a_k^2}{2} \left (
			\norm{x-v_{k-1}}^2 - \norm{x-v_k}^2
			\right )\\
			&- \frac{(\alpha^{-1}-1)}{2t_k}\|y_k-x_k\|^2+\rho a_k(1-a_k)\|x-x_{k-1}\|^2+\frac{ra_k^2}{2}\|x-v_{k-1}\|^2.
		\end{aligned}
	\end{equation}
	Plugging in $x=x^*$, subtracting $F(x^*)$ from both sides, and rearranging yields
	\begin{align*}
		\frac{F(x_k)-F(x^*)}{a_k^2}+\frac{\tilde{\mu}_k}{2}\|x^*-v_k\|^2&\leq \frac{1-a_k}{a_k^2}(F(x_{k-1})-F(x^*))+\frac{\tilde{\mu}_k}{2}\|x^*-v_{k-1}\|^2\\
		&\quad+\frac{\rho M^2}{a_k}+\frac{r M^2}{2}-\frac{(\alpha^{-1}-1)}{2t_k a_k^2}\|y_k-x_k\|^2.
	\end{align*}
	This is exactly inequality \eqref{eq: something} with 
	$\tfrac{\tilde{\mu}}{2}$ replaced by $\tfrac{\tilde{\mu}_k}{2}$ and
	$\tfrac{\tilde{\mu}-\mu}{2}$ replaced by $\tfrac{(\alpha^{-1}-1)}{2t_k}$;
	Using the fact that the sequence $\{\tilde{\mu}_k\}_{k=0}^\infty$ is nondecreasing and
	$\tfrac{1-a_k}{a_k^2} \le \frac{1}{a_{k-1}^2}$, we deduce 
	\begin{equation} \label{eq: back_5.5}
		\begin{aligned}
			\frac{F(x_k)-F(x^*)}{a_k^2} + \frac{\tilde{\mu}_k}{2} \norm{x^*-v_k}^2 
			&\le \frac{\tilde{\mu}_k}{\tilde{\mu}_{k-1}} \bigg (  \frac{ F(x_{k-1})-F(x^*) }{a_{k-1}^2} +
			\frac{\tilde{\mu}_{k-1}}{2} \norm{x^*-v_{k-1}}^2\\
			&    \qquad \qquad + \frac{\rho
				M^2}{a_k} + \frac{rM^2}{2} - \frac{(\alpha^{-1}-1)}{2t_k a_k^2}
			\norm{y_k-x_k}^2 \bigg ).
		\end{aligned}
	\end{equation}
	Notice $\tilde{\mu}_k\le
	\alpha^{-1} \max \left
	\{t_0^{-1}, \eta^{-1} \mu \right \} =:
	\tilde{\mu}_{\max}$. Recursively applying \eqref{eq: back_5.5} $N$
	times, we get 
	\begin{equation} \label{eq: back_7}
		\begin{aligned}
			\frac{F(x_N)-F(x^*)}{a_N^2} + \frac{\tilde{\mu}_N}{2} \norm{x^*-v_N}^2 \le \left ( \prod_{j=1}^N
			\frac{\tilde{\mu}_j}{\tilde{\mu}_{j-1}} \right ) \bigg ( \frac{\tilde{\mu}_0}{2} &\norm{x^*-v_0}^2 + \sum_{j=1}^N \frac{\rho M^2}{a_j} +
			\frac{NrM^2}{2}\\
			& - \sum_{j=1}^N \frac{(\alpha^{-1}-1)}{2 t_j} \cdot \frac{\norm{x_j-y_j}^2}{a_j^2}\bigg )
		\end{aligned}
	\end{equation}
	By the telescoping property of $\prod_{j=1}^N
	\tfrac{\tilde{\mu}_j}{\tilde{\mu}_{j-1}} \le \tfrac{\tilde{\mu}_{\max}}{\tilde{\mu}_0}$, we conclude
	\begin{align}
		\frac{\tilde{\mu}_{\max}}{\tilde{\mu}_0} \sum_{j=1}^N \frac{(\alpha^{-1}-1)}{2t_j} \cdot
		\frac{\norm{x_j-y_j}^2}{a_j^2} &\le
		\frac{\tilde{\mu}_{\max}}{\tilde{\mu}_0}
		\left (
		\frac{\tilde{\mu}_0}{2} \norm{x^*-v_0}^2 + \rho M^2   \left (
		\sum_{j=1}^N \frac{1}{a_j} \right ) +\frac{NrM^2}{2} \right ). \label{eq: back_6}
	\end{align}
	Using the inequality \eqref{eq: back_6} and $\alpha t_j  = \tilde{\mu}_j^{-1} \ge
	\tilde{\mu}_{\max}^{-1}$ for all $j$, we conclude
	\begin{align*}
		\tfrac{(\alpha^{-1}-1) \alpha}{2 \tilde{\mu}_0} \cdot \left ( \sum_{j=1}^N \tfrac{1}{a_j^2} \right )
		\min_{j=1, \hdots, N}  \norm{\tilde{\mu}_j (x_j-y_j)}^2 
		&\le \tfrac{\tilde{\mu}_{\max}}{\tilde{\mu}_0} \left ( \tfrac{\tilde{\mu}_0}{2} \norm{x^*-v_0}^2 + \rho M^2  \left (
		\sum_{j=1}^N \tfrac{1}{a_j} \right )  +\tfrac{NrM^2}{2} \right ).
	\end{align*}
	The result follows by mimicking the rest of the proof in
	Theorem~\ref{thm: cnx}. 
	Finally, suppose $r = 0$, and hence we can assume $\rho =
	0$. Inequality \eqref{eq: back_7} then implies
	\[ \frac{F(x_N)-F(x^*)}{a_N^2} + \frac{\tilde{\mu}_N}{2}
	\norm{x^*-v_N}^2 \le \frac{\tilde{\mu}_{\max}}{\tilde{\mu}_0} \cdot
	\frac{\tilde{\mu}_0}{2} \norm{x^*-v_0}^2.\]
	The claimed efficiency
	estimate follows. 
\end{proof}

\section{Removing the logarithmic dependence when an estimate on $F(x_0)-\inf F$ is known.}\label{sec:app_dual_meth}
In this section, we show that if a good estimate on the error $F(x_0)-\inf F$ is available, then there is a first-order method for the composite problem class \ref{eqn:comp2} with efficiency $\mathcal{O}\left({\frac{L^2\beta\|\nabla c\|\cdot(F(x_0)-\inf F)}{\varepsilon^3}}\right)$. Notice that this is an improvement over \eqref{eqn:final_cost2} since there is no logarithmic term. The outline is as follows. We will fix at the very beginning a budget of basic operations we are willing to tolerate. We will then perform a constant number of iterations of the inexact prox-linear Algorithm~\ref{alg: inex_prox_lin} with a constant number of iterations of an accelerated primal-dual first-order method on the proximal subproblem. Before delving into the details, it is important to note two  downsides of the scheme, despite the improved worst-case efficiency over the smoothing technique. First, we must have a good estimate on  $F(x_0)-\inf F$. Secondly, the number of inner iterations we are willing to tolerate depends on $\|\nabla c\|$, rather than on the norms $\|\nabla c(x_k)\|_{\textrm{op}}$ along the generated iterate sequences $x_k$. The reason is that the number of iterations (both outer and inner) must be set a priori, without knowledge of the iterates that will be generated. This is in direct contrast to the algorithms discussed in Section~\ref{sec:overall_comp}, where the dependences on $\|\nabla c\|$ could always be replaced by an upper bound on $\max_k \|\nabla c(x_k)\|_{\textrm{op}}$ along the generated iterate sequence $x_k$. Nonetheless, from the complexity viewpoint, the improved efficiency estimate  is notable.

We now describe the outlined strategy in detail. In order to find approximate minimizers of the proximal subproblems \eqref{eqn:target_rewritten}, let us instead focus on the dual \eqref{eqn:fench_dual}, and apply a (fast) primal-dual method with sublinear guarantees. To specify precisely the method  we will use on the subproblems, we follow the exposition in \cite{tseng_f_order}.
Recall that $\displaystyle G^{\star}$ is $C^1$-smooth with $t$-Lipschitz gradient. Moreover since $h$ is $L$-Lipschitz,  the domain of the function $w\mapsto h^{\star}(w)-\langle b,w\rangle$ has diameter upper bounded by $2L$. 
In Algorithm~\ref{alg:prim_dual}, we record the specialization of  \cite[Algorithm~1]{tseng_f_order} to our target problem \eqref{eqn:fench_dual}.\footnote{In the notation of \cite{tseng_f_order}, we set
	$\phi(w,v):=\langle v,A^*w\rangle-{G}(v)$ and $p(w):=h^{\star}(w)-\langle b,w\rangle$, and note $\prox_{tp}(\cdot)=\prox_{th^{\star}}(\cdot+tb)$.}
%

{\LinesNotNumbered
	\begin{algorithm}[h!]
		\SetKw{Null}{NULL}
		\SetKw{Return}{return}
		\Initialize{Fix two points $w_0, z_0 \in \dom p$; choose a real $l\geq t\|A\|^2$; set $v_{-1}:=0$ and $a_0:=1$.}
		{\bf Step j:} ($j\geq 0)$ Compute
		\begin{align*}
		y_j &= (1-a_j) w_{j} + a_j z_{j}\\
		z_{j+1}&=\prox_{\tfrac{h^{\star}}{a_jl}}\left(z_j-\tfrac{1}{a_j l}(\nabla G^{\star}(y_j)-b)\right)\\		
		w_{j+1}&= (1-a_j)w_j+a_jz_{j+1}\\
		a_{j+1}&=\tfrac{\sqrt{a_j^4+4a_j^2}-a_j^2}{2}
		\end{align*}
		Update the primal iterate
		$$v_j = (1-a_j)v_{j-1}+a_j\nabla G^{\star}(A^*y_j)$$

		%
		\caption{Optimal method (Auslender-Teboulle \cite{AT_fast}, Tseng \cite[Algorithm~1]{tseng_f_order})}
		\label{alg:prim_dual}
	\end{algorithm}}
	
	Algorithm~\ref{alg:prim_dual} comes equipped with the following guarantee \cite[Corollary $1(b)$]{tseng_f_order}.
	\begin{theorem}\label{thm:bas_conv_prim_dual}
		For every index $j$, the iterates generated by Algorithm~\ref{alg:prim_dual} satisfy:
		$$F_t(v_j;x)-\inf   F_t(\cdot;x)\leq \frac{8l L^2}{(j+2)^2}.$$ 
	\end{theorem}

	Set $t=1/\mu$ and fix a real $q>0$, which will appear in the final efficiency estimate. Suppose that we aim to run a total of at most $T$ iterations of Algorithm~\ref{alg:prim_dual} over all the proximal subproblems. Suppose moreover that $T$ is sufficiently large to satisfy $T\geq \frac{4(1.5)^{3/2}\|\nabla c\|}{\sqrt{2\beta q/L}}$.
	
	Consider now the following procedure. Define
	$$N:=\ceil*{\left(\frac{T\sqrt{2\beta q/L}}{4\|\nabla c\|}\right)^{2/3}}-2$$
	and note $N\geq 0$.
	Let us now run the inexact prox-linear Algorithm~\ref{alg: inex_prox_lin} for $k=0,\ldots,N$ iterations  with each prox-linear subproblem approximately solved by running $$\ceil*{4\|\nabla c\|\sqrt{\tfrac{L(N+1)}{2\beta q}}}$$ 
	iterations of Algorithm~\ref{alg:prim_dual}; we will determine an estimate on the incurred errors $\varepsilon_k>0$ shortly. Observe that the total number of iterations of Algorithm~\ref{alg:prim_dual} is indeed at most $$(N+1)\cdot\ceil*{4\|\nabla c\|\sqrt{\tfrac{L(N+1)}{2\beta q}}}\leq 4\|\nabla c\|\sqrt{L}(N+1)^{3/2}/\sqrt{2\beta q}\leq T.$$ Appealing to Theorem~\ref{thm:bas_conv_prim_dual}, we deduce
\begin{align*}
	F_{1/\mu}(x_{k+1};x_{k})-\inf   F_{1/\mu}(\cdot;x_k)&\leq\frac{8\|\nabla c\|^2 L^2/\mu}{\ceil{4\|\nabla c\|\sqrt{L(N+1)/(2\beta q)}}^2}\\
	&\leq \frac{8\|\nabla c\|^2 L^2/\mu}{16 L\|\nabla c\|^2(N+1)/(2\beta q)}=\frac{q}{N+1}.
	\end{align*}
	Thus, in the notation of Algorithm~\ref{alg: inex_prox_lin} we can set $\varepsilon_k:=\frac{q}{N+1}$ for each index $k$. Theorem~\ref{thm: funct_pla} then yields the estimate
	\[ \min_{i=0, \hdots, N-1}  \norm{\mathcal{G}_{1/\mu}(x_i)}^2  \le
	\frac{2\mu \big ( F(x_0)-\inf F +  q\big
		)}{N}\leq \frac{2\mu \big ( F(x_0)-\inf F +  q\big
		)}{{\left(\frac{T\sqrt{2\beta q/L}}{4\|\nabla c\|}\right)^{2/3}}-2},\]
	Thus to find a point $x$ with $\|\mathcal{G}_{1/\mu}(x)\|\leq \varepsilon$ it suffices to choose $T$ satisfying 
	$$T\geq \frac{8\|\nabla c\|}{\sqrt{\beta q/L}}\cdot \left(1+\frac{\mu(F(x_0)-\inf F+q)}{\varepsilon^2}\right)^{3/2}.$$
	Notice that the assumed bound $T\geq \frac{4(1.5)^{3/2}\|\nabla c\|}{\sqrt{2\beta q/L}}$ holds automatically for this choice of $T$.
	
	In particular, if $q$ can be chosen to satisfy $\frac{q}{F(x_0)-\inf F}\in [\gamma_1,\gamma_2]$ for some fixed constants $\gamma_2\geq\gamma_1\geq 1$, the efficiency estimate becomes on the order of
	$$	\boxed{\mathcal{O}\left({\frac{L^2\beta\|\nabla c\|\cdot(F(x_0)-\inf F)}{\varepsilon^3}}\right)},$$	
	as claimed.

\end{document}